\documentclass[11pt]{article}
\usepackage{graphicx, amsmath, amsthm,amssymb,tikz,color}
\usetikzlibrary{arrows}
\usepackage[a4paper]{geometry}

\theoremstyle{plain}
\newtheorem{thm}{Theorem}[section]
\newtheorem{cor}[thm]{Corollary}
\newtheorem{prop}[thm]{Proposition}
\newtheorem{lem}[thm]{Lemma}

\theoremstyle{definition}

\newtheorem{rem}[thm]{Remark}

\makeatletter

\def\captionheadfont@{\scshape}
\def\captionfont@{\small}
\long\def\@makecaption#1#2{%
  \setbox\@tempboxa\vbox{\color@setgroup
    \advance\hsize-3pc\noindent
    \captionfont@\captionheadfont@#1\@xp\@ifnotempty\@xp
        {\@cdr#2\@nil}{.\captionfont@\upshape\enspace#2}%
    \unskip\kern-3pc\par
    \global\setbox\@ne\lastbox\color@endgroup}%
  \ifhbox\@ne 
    \setbox\@ne\hbox{\unhbox\@ne\unskip\unskip\unpenalty\unkern}%
  \fi
  \ifdim\wd\@tempboxa=\z@ 
    \setbox\@ne\hbox to\columnwidth{\hss\kern-3pc\box\@ne\hss}%
  \else 
    \setbox\@ne\vbox{\unvbox\@tempboxa\parskip\z@skip
        \noindent\unhbox\@ne\advance\hsize-3pc\par}%
\fi
  \ifnum\@tempcnta<64 
    \addvspace\abovecaptionskip
    \moveright 1.5pc\box\@ne
  \else 
    \moveright 1.5pc\box\@ne
    \nobreak
    \vskip\belowcaptionskip
  \fi
\relax
}

\def\overbar#1{\skewbar{#1}{-1}{-1}{.25}}
\def\skewbar#1#2#3#4{\preciseskewbar{#1}{#2}{#3}{#2}{#3}{#2}{#3}{#4}1}
\def\preciseskewbar#1#2#3#4#5#6#7#8#9{{\mathchoice
 {\makeoverbar\textfont\displaystyle{#1}1{#2}{#3}{#8}{#9}}
 {\makeoverbar\textfont\textstyle{#1}1{#2}{#3}{#8}{#9}}
 {\makeoverbar\scriptfont\scriptstyle{#1}{.7}{#4}{#5}{#8}{#9}}
 {\makeoverbar\scriptscriptfont
  \scriptscriptstyle{#1}{.5}{#6}{#7}{#8}{#9}}}#1}
\def\makeoverbar#1#2#3#4#5#6#7#8{{%
 \setbox0=\hbox{$\m@th#2\mkern#5mu{#3}\mkern#6mu$}%
 \setbox1=\null \dimen@=#4\fontdimen8#13 \dimen@=3\dimen@ 
 \advance\dimen@ by \ht0 \dimen@=-#7\dimen@ \advance\dimen@ by \wd0
 \wd1=\dimen@ \dp1=\dp0 
 \dimen@=#4\fontdimen8#13
 \dimen@i=\fontdimen8#13
 \fontdimen8#13=#8\dimen@
 \advance\dimen@ by -\fontdimen8#13 \dimen@=3\dimen@
 \advance\dimen@ by \ht0 \ht1=\dimen@ 
 \rlap{\hbox to \wd0{$\m@th\hss#2{\overline{\box1}}\mkern#5mu$}}
 \fontdimen8#13=\dimen@i}}

\makeatother

\def\R{\mathbb{R}}

\def\N{\mathbb{N}}
\def\P{\mathbb{P}}
\def\E{\mathbb{E}}

\def\I{\infty}

\newcommand{\be}{\begin{equation}}
\newcommand{\ee}{\end{equation}}
\newcommand{\bea}{\begin{eqnarray}}
\newcommand{\eea}{\end{eqnarray}}
\newcommand{\beann}{\begin{eqnarray}}
\newcommand{\eeann}{\end{eqnarray}}
\newcommand{\benn}{\begin{equation}}
\newcommand{\eenn}{\end{equation}}

\def\ra{\rightarrow}
\def\I{\infty}

\def\stoch{\text{stoch}}

\def\figref#1{Figure~\ref{#1}}
\def\writefig#1 #2 #3 {\rlap{\kern #1 truecm
\raise #2 truecm \hbox{\protect{\small #3}}}}

\DeclareMathSymbol{\leqsymb}{\mathalpha}{AMSa}{"36}
\def\leqs{\mathrel\leqsymb}
\DeclareMathSymbol{\geqsymb}{\mathalpha}{AMSa}{"3E}
\def\geqs{\mathrel\geqsymb}

\DeclareMathOperator{\dd}{d}            
\DeclareMathOperator{\e}{e}             
\DeclareMathOperator{\icx}{i}           
\DeclareMathOperator{\Tr}{Tr}           
\DeclareMathOperator{\diam}{diam}       
\DeclareMathOperator{\sign}{sign}       
\DeclareMathOperator{\dist}{dist}       

\def\6#1{\dd\!#1}                       
\def\dtot#1#2{\frac{\6{#1}}{\6{#2}}}  
\def\one{{\mathchoice {\rm 1\mskip-4mu l} {\rm 1\mskip-4mu l}
{\rm 1\mskip-4.5mu l} {\rm 1\mskip-5mu l}}}     
\def\abs#1{\lvert#1\rvert}                      
\def\norm#1{\left\|#1\right\|}          
\def\Order#1{{\mathcal O}(#1)}                  
\def\transpose#1{#1^{\mathrm T}}                  

\newcommand{\cA}{{\mathcal A}}  
\newcommand{\cD}{{\mathcal D}}  
\newcommand{\cF}{{\mathcal F}}  
\newcommand{\cG}{{\mathcal G}}  
\newcommand{\cI}{{\mathcal I}}  
\newcommand{\cO}{{\mathcal O}}  

\def\eps{\epsilon}

\def\const{\textit{const }}
\def\Kbar{\overbar K}

\def\setsuch#1#2{\{#1\colon #2\}}                

\def\intpartplus#1{\lceil#1\rceil}              

\def\biggintpartplus#1{\biggl\lceil#1\biggr\rceil}

\numberwithin{equation}{section}

\definecolor{darkgreen}{rgb}{0,0.6,0}

\makeatletter

\def\enum{\ifnum \@enumdepth >3 \@toodeep\else
        \advance\@enumdepth \@ne 
        \edef\@enumctr{enum\romannumeral\the\@enumdepth}\list
        {\csname label\@enumctr\endcsname}
        {\setlength{\topsep}{1mm}
        \setlength{\parsep}{0mm}
        \setlength{\itemsep}{0mm}
        \setlength{\labelsep}{2mm}
        \settowidth{\leftmargin}{M.}
        \addtolength{\leftmargin}{\labelsep}
        \usecounter{\@enumctr}
        \def\makelabel##1{\hss\llap{##1}}}\fi}

\def\enumH{\ifnum \@enumdepth >3 \@toodeep\else
        \advance\@enumdepth \@ne 
        \edef\@enumctr{enum\romannumeral\the\@enumdepth}\list
        {\csname label\@enumctr\endcsname}
        {\setlength{\topsep}{1mm}
        \setlength{\parsep}{0mm}
        \setlength{\itemsep}{0mm}
        \setlength{\labelsep}{2mm}
        \settowidth{\leftmargin}{H5.}
        \addtolength{\leftmargin}{\labelsep}
        \usecounter{\@enumctr}
        \def\makelabel##1{\hss\llap{##1}}}\fi}

\def\itemiz{\ifnum \@itemdepth >3 \@toodeep\else \advance\@itemdepth \@ne
        \edef\@itemitem{labelitem\romannumeral\the\@itemdepth}%
        \list{\csname\@itemitem\endcsname}{
        \setlength{\topsep}{1mm}
        \setlength{\parsep}{0mm}
        \setlength{\parsep}{0mm}
        \setlength{\itemsep}{0mm}
        \setlength{\labelsep}{2mm}
        \settowidth{\leftmargin}{(IV)}
        \addtolength{\leftmargin}{\labelsep}
        \def\makelabel##1{\hss\llap{##1}}}\fi}

\def\itemizz{\ifnum \@itemdepth >3 \@toodeep\else \advance\@itemdepth \@ne
        \edef\@itemitem{labelitem\romannumeral\the\@itemdepth}%
        \list{\csname\@itemitem\endcsname}{
        \setlength{\topsep}{1mm}
        \setlength{\parsep}{0mm}
        \setlength{\parsep}{1mm}
        \setlength{\itemsep}{1mm}
        \setlength{\labelsep}{2mm}
        \settowidth{\leftmargin}{M.}
        \addtolength{\leftmargin}{\labelsep}
        \def\makelabel##1{\hss\llap{##1}}}\fi}

\makeatother

\begin{document}

\date{December 22, 2013 \\ Revised, November 14, 2014}

\author{Nils Berglund\thanks{MAPMO,
CNRS -- UMR 7349, Universit\'{e} d'Orl\'{e}ans, F\'{e}d\'{e}ration Denis
Poisson -- FR 2964, B.P. 6759, 45067 Orl\'{e}ans Cedex 2, France.} 
\and Barbara Gentz\thanks{Faculty of Mathematics, University of Bielefeld, Universit\"atsstr.~25, 33615~Bielefeld, Germany.}
{}\thanks{Research supported by CRC 701 \lq\lq Spectral Structures and Topological Methods in Mathematics\rq\rq.}
\and Christian Kuehn\thanks{Vienna University of Technology, Institute for
Analysis and Scientific Computing, 1040 Vienna, Austria.}}%

\title{
From random Poincar\'e maps to\\ stochastic mixed-mode-oscillation
patterns
}

\maketitle

\begin{abstract}
\noindent
We quantify the effect of Gaussian white noise on fast--slow dynamical 
systems with one fast and two slow variables, which display mixed-mode 
oscillations owing to the presence of a folded-node singularity. The 
stochastic system can be described by a continuous-space, discrete-time 
Markov chain, recording the returns of sample paths to a Poincar\'e section.
We provide estimates on the kernel of this Markov chain, depending on the 
system parameters and the noise intensity. These results yield predictions 
on the observed random mixed-mode oscillation patterns. Our analysis shows that
there is an intricate interplay between the number of small-amplitude
oscillations and the global return mechanism. 
In combination with a local saturation phenomenon near the folded node, this
interplay can modify the number of small-amplitude oscillations after a
large-amplitude oscillation. Finally, sufficient conditions are derived which
determine when the noise increases the number of small-amplitude oscillations
and when it decreases this number.
\end{abstract}




\section{Introduction}  
\label{sec:intro}

Oscillation patterns with large variations in amplitude occur frequently in dynamical systems,
differential equations and their applications. A class of particular interest
are mixed-mode oscillations (MMOs) which are patterns consisting of alternating
structures of small-amplitude and large-amplitude oscillations. Typical
applications arise from chemical systems such as the Belousov--Zhabotinskii
reaction~\cite{HudsonHartMarinko}, the peroxidase--oxidase reaction~\cite{DegnOlsenPerram} and autocatalytic reactions~\cite{PetrovScottShowalter}
as well as from neuroscience, e.g.~stellate cells~\cite{Dicksonetal},
Hodgkin--Huxley-type neurons~\cite{RubinWechselberger1} and pituitary cells~\cite{VanGoorZivadinovicMartinez-FuentesStojilkovic}. A remarkable number of
models for these phenomena lead to differential equations with multiple timescales see
{e.g.}~\cite{MilikSzmolyanMap, DesrochesKrauskopfOsinga1,WechselbergerWeckesser,GuckenheimerScheper}. Frequently, it suffices to consider two timescales and
study fast--slow ordinary differential equations (ODEs) which already provide
many generic mechanisms leading to MMOs. For a detailed review of the topic we
refer to the survey~\cite{KuehnMMO}, the special issue~\cite{BronsKrupaRotstein}, 
and references therein. The basic idea is that a local mechanism induces the
small-amplitude oscillations (SAOs) while a global return mechanism leads to
large-amplitude oscillations (LAOs). In this introduction, we shall just outline
the main ideas; the precise development of our set-up and results starts in
Section~\ref{sec:setup}.

For a deterministic trajectory, we can symbolically write an MMO as a sequence
\be
\label{eq:MMO_pattern}
\cdots  L_{j-1}^{s_{j-1}} L_j^{s_j} L_{j+1}^{s_{j+1}}\cdots
\ee
where $L_{j}^{s_{j}}$ denotes $L_{j}$ LAOs followed by $s_{j}$ SAOs, etc. For example, a
periodic solution alternating between $2$ SAOs and $1$ LAO would be $\cdots
1^21^21^2\cdots$ or simply $1^2$ with the periodicity understood. A prototypical
mechanism to generate SAOs are folded-node singularities~\cite{BronsKrupaWechselberger} which are generic in three-dimensional ODEs with
one fast and two slow variables~\cite{Wechselberger,Benoit1}. For the global
return mechanism, one frequently encounters a relaxation-type structure induced
by a cubic (or S-shaped) fast-variable nullcline, also called the critical
manifold, which was studied extensively already in the context of van der
Pol-type oscillators; see
{e.g.}~\cite{Koper,GoryachevStrizhakKapral,GuckFvdP2,BronsKrupaWechselberger}
and the references therein. Non-degenerate folds, folded-node singularities and
S-shaped critical manifolds form the basic deterministic building blocks for the
work in this paper. However, let us mention already here that the stochastic techniques
we develop in this paper could potentially be adapted to other cases such as
singular Hopf bifurcation and folded saddle-nodes~\cite{GuckenheimerSH,KrupaWechsFSN}, bursting oscillations~\cite{Izhikevich,ErmentroutTerman}, tourbillon structures~\cite{Wallet3,KuehnMMO}, and other global return mechanisms~\cite{GuckenheimerScheper,KuehnUM}. Although it is certainly of high interest to
study all these cases, it seems to us that the combination of folded
singularities and relaxation oscillations is a natural first step as both
components are basic elements which occur in a large variety of different models~\cite{KuehnMMO}.

While in some experiments, remarkably clear MMO patterns have been
observed~\cite{HauckSchneider,HudsonHartMarinko}, in many other cases 
the SAOs in the patterns appear noisy~\cite{Dicksonetal,Mikikian}. Weak noise acting on a dynamical
system is known to induce a variety of phenomena, ranging from small
fluctuations around deterministic solutions to large excursions in
phase space, as shown, e.g., in stochastic
resonance~\cite{BSV,GHM,McNW} and transitions near tipping
points~\cite{Schefferetal,KuehnCT2,AshwinWieczorekVitoloCox}.
In the context of oscillatory patterns, the effect of noise on MMO patterns
in low-dimensional prototypical models has been studied, for instance,
in~\cite{KosmidisPakdaman,SuRubinTerman,MuratovVanden-Eijnden,
YuKuskeLi,HitczenkoMedvedev,KuskeBorowski}, using numerical simulations,
bifurcation theory and asymptotic descriptions of the Fokker--Planck equation.

This work concerns the effect of noise on fast--slow differential
equations with one fast and two slow variables, containing a
folded-node singularity and an S-shaped critical manifold responsible
for the global return mechanism. The resulting stochastic differential
equations (SDEs) show a subtle interplay between noise, local and
global dynamics, which requires a careful analysis of the behaviour of
stochastic sample paths.  Our approach builds upon our earlier
work~\cite{BGK12} which in turn was based upon a pathwise approach to
fast--slow SDEs~\cite{BG6, BGbook}.

Our main focus is the derivation of estimates for the Poincar\'e (or return) map
of the stochastic system, for a conveniently chosen
two-dimensional section~$\Sigma$. Deterministic return maps in the
presence of folded-node singularities have been analyzed, e.g.,
in~\cite{GuckenheimerFNFSN,KrupaPopovicKopell}. Although the two-dimensional Poincar\'e map is
invertible, the strong contraction near attracting critical manifolds implies
that it is close to a one-dimensional, usually non-invertible map.
\figref{fig_Koper_poincare}\,(a) shows an example of such a one-dimensional
deterministic return map $z_n\mapsto z_{n+1}$. The apparent discontinuities are
in fact points where the map's graph displays very narrow dips, due to the
presence of so-called canard orbits. Canards are particular solutions of the
system staying close to both the attracting and repelling parts of the critical
manifold~\cite{BenoitCallotDienerDiener,Benoit1,Benoit4}, which separate the phase
space into sectors of rotation characterized by different numbers of
SAOs~\cite{BronsKrupaWechselberger}. 

\begin{figure}
\centerline{\includegraphics*[clip=true,width=\textwidth]{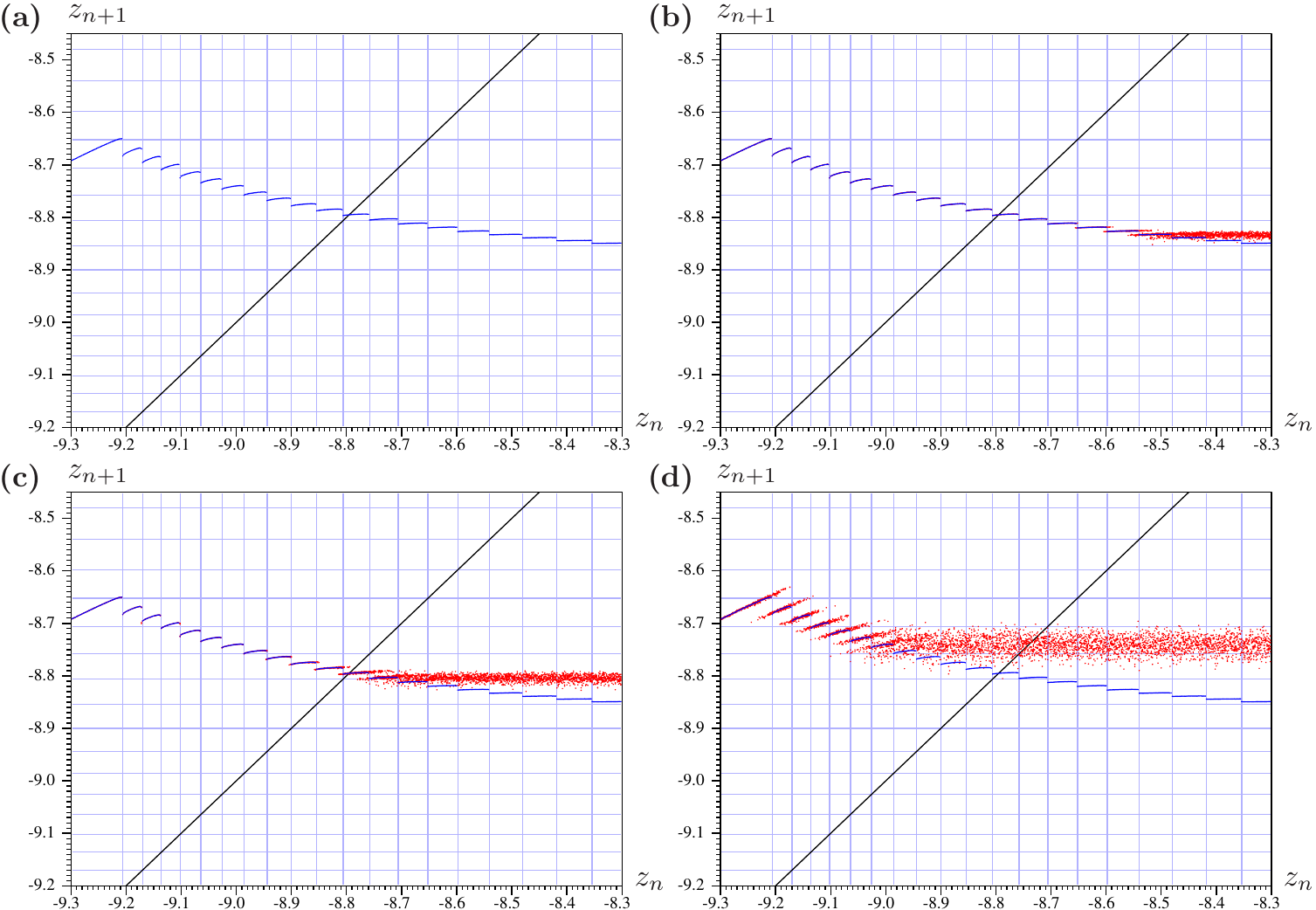}}
\caption[]{$z$-coordinate of the first-return map on the section~$\Sigma_1$ for
the Koper model~\eqref{eq:Koper}. Parameter values are~$k = -10$, $\lambda =
-7.6$, $\eps_2=0.7$, $\eps_1=0.01$, with noise intensities  
{\bf (a)}~$\sigma=\sigma'=0$, 
{\bf (b)}~$\sigma=\sigma'=2\cdot10^{-7}$,
{\bf (c)}~$\sigma=\sigma'=2\cdot10^{-5}$,
and {\bf (d)}~$\sigma=\sigma'=2\cdot10^{-3}$. The horizontal and vertical lines
mark the location of canards.}
\label{fig_Koper_poincare}
\end{figure}

The concept of return maps has been extended to stochastic systems, see for
instance~\cite{WeissKnobloch,BerglundLandon,HitczenkoMedvedev1}. This requires
some care, because the rapid fluctuations of stochastic sample paths prevent one
from using their very first return to $\Sigma$ to define the map. Instead, one
has to consider the first return after a suitably defined, sufficiently large excursion in
phase space has taken place. With these precautions, successive intersections
$X_0, X_1, X_2, \dots$ of sample paths with $\Sigma$ define a
continuous-space, discrete-time Markov chain. The distribution of $X_{n+1}$ is
obtained from the distribution of $X_n$ via integration with respect to a
transition kernel $K$. Under suitable regularity
assumptions~\cite{BenArous_Kusuoka_Stroock_1984}, the theory of harmonic
measures ensures that the kernel~$K$ admits a smooth density~$k$, so that the
evolution of $X_n$ is specified by an integral equation, namely 
\begin{equation}
 \label{eq:transition_kernel}
 \P\bigl\{ X_{n+1}\in A \mid X_n = x
 \bigr\} = \int_A k(x,y) \6y =: K(x,A) 
\end{equation} 
holds for all Borel sets $A\subset\Sigma$; see for
instance~\cite[Sections~5.2 and~5.3]{BG_periodic2}. The main aim of the present work is
to provide estimates on the kernel~$K$. Part of the mathematical challenge is due to the
fact that the deterministic flow is not a gradient flow, and thus the stochastic system is irreversible. 

\figref{fig_Koper_poincare}\,(b)--(d) shows simulated stochastic return maps for
increasing noise intensity. For each value of $z_n$, the red points indicate
the value of $z_{n+1}$ for $10$ different realizations of the noise. The
deterministic return map is plotted in blue for comparison. Several interesting
phenomena can be observed:

\begin{enum}
\item 	The size of fluctuations increases with the noise intensity;
\item 	Orbits in sectors with a small number of SAOs (inner sectors) are less affected by noise than
those in sectors with a large number of SAOs (outer sectors);
\item 	There is a saturation effect, in the sense that for large enough SAO numbers, the typical value of the stochastic return map and its spreading become independent of the sector;
\item 	The saturation effect sets in earlier for larger noise intensities. 
\end{enum}

While the first phenomenon is not surprising, the other observed features are
remarkable, and can lead to non-intuitive effects. In the example shown
in~\figref{fig_Koper_poincare}, the deterministic map has a stable fixed point
in the 11th sector, so that the deterministic system will display a stable
MMO pattern~$1^{11}$. For sufficiently strong noise, the stochastic system
will asymptotically operate in the 12th sector, with occasional transitions to
neighbouring sectors such as sectors 11 and 13. Hence, the noise shifts the
global return to a higher rotation sector.
However, two more noise-induced effects may also affect the number of observed
SAOs. First, the noise may alter the number of SAOs for orbits starting in a
given sector, by causing earlier escapes away from the critical manifold.
Second, it may produce fluctuations large enough to mask small oscillations.
All these effects must be quantified and compared to determine which oscillatory
pattern we expect to observe.

The estimates on the kernel we provide in this work yield quantitative
information on the above phenomena. In particular, we obtain estimates on the
typical size of fluctuations as a function of noise intensity and sector number,
and on the onset of the saturation phenomenon. These results complement those
already obtained in~\cite{BGK12} on the size of fluctuations near a folded-node
singularity.

The structure of this article is the following: After introducing the deterministic set-up in
Section~\ref{sec:setup}, we provide first estimates on noise-induced
fluctuations in Section~\ref{sec:global_deviate}. Sections~\ref{sec:fold}
and~\ref{sec:local_deviate} extend the analysis to a neighbourhood of the
regular fold and of the folded node, respectively. Section~\ref{sec:MC} combines all the
local estimates to provide quantitative results on the kernel. The main results
are:

\begin{itemiz}
\item 	\textbf{Theorem~\ref{thm_global_returns}} 
\emph{(global return map)}\/
quantifies the effect of noise during the global return phase; 

\item 	\textbf{Theorem~\ref{thm_local_returns_inner}} 
\emph{(local map for inner sectors)}\/ 
provides estimates on noise-induced fluctuations for orbits
starting near a folded node in sectors with small SAO number; together with
Theorem~\ref{thm_global_returns} it yields bounds on the size of fluctuations of the
Poincar\'e map in all inner sectors; 

\item 	\textbf{Theorem~\ref{thm_local_returns_outer}} 
\emph{(local map for outer sectors)}\/ 
gives similar estimates for orbits starting in sectors with a
large SAO number; in particular, it proves the saturation effect.
\end{itemiz}

\noindent
A short discussion of the consequences of these results on the observed
MMO patterns is given in Section~\ref{ssec_MMO_patterns}.
Finally, Section~\ref{sec:Koper} illustrates these results with numerical
simulations for the Koper model.

\begin{figure}
\begin{center}
\begin{tikzpicture}[->,>=stealth',shorten >=1pt,auto,node distance=2.5cm,
  thick,main
node/.style={circle,scale=0.7,fill=blue!20,draw,font=\sffamily\Large\bfseries}]

  \node[main node] (1) {1};
  \node[main node] (4) [right of=1] {4};
  \node[main node] (2) [below right of=4] {2};
  \node[main node] (5) [below left of=2] {5};
  \node[main node] (3) [right of=2] {3};

  \path[every node/.style={font=\sffamily\small}]
    (1) edge node [above] {$1$} (4)
    (4) edge node [above right] {$1$} (2)
    (5) edge node [below right] {$1$} (2)
    (2) edge node [above] {$1$} (3)
    (3) edge [loop right] node {$1$} (3)
    (5) edge [bend left, dashed, color=white] node {$\sigma$} (1)
    (2) edge [bend left, dashed, color=white] node {$\sigma$} (4)
    (3) edge [bend left, dashed, color=white] node {$\sigma$} (2)
    (3) edge [<->, bend right, dashed, color=white] node [above] {$\sigma$} (4)
    (1) edge [bend left=65, dashed, color=white] node {$\sigma$} (3)
    (3) edge [bend left=35, dashed, color=white] node {$\sigma^2$} (5)
;
\end{tikzpicture}
\hspace{5mm}
\begin{tikzpicture}[->,>=stealth',shorten >=1pt,auto,node distance=2.5cm,
  thick,main
node/.style={circle,scale=0.7,fill=blue!20,draw,font=\sffamily\Large\bfseries}]

  \node[main node] (1) {1};
  \node[main node] (4) [right of=1] {4};
  \node[main node] (2) [below right of=4] {2};
  \node[main node] (5) [below left of=2] {5};
  \node[main node] (3) [right of=2] {3};

  \path[every node/.style={font=\sffamily\small}]
    (1) edge node [above] {$1-\sigma$} (4)
    (4) edge node [above right] {$1-\sigma$} (2)
    (5) edge node [below right] {$1-\sigma$} (2)
    (2) edge node [above] {$1-\sigma$} (3)
    (3) edge [loop right] node {$1-2\sigma-\sigma^2$} (3)
    (5) edge [bend left, dashed, color=red] node {$\sigma$} (1)
    (2) edge [bend left, dashed, color=red] node {$\sigma$} (4)
    (3) edge [bend left, dashed, color=red] node {$\sigma$} (2)
    (3) edge [<->, bend right, dashed, color=red] node [above] {$\sigma$} (4)
    (1) edge [bend left=65, dashed, color=red] node {$\sigma$} (3)
    (3) edge [bend left=35, dashed, color=red] node {$\sigma^2$} (5)
;
\end{tikzpicture}
\end{center}
\vspace{-5mm}
\caption[]{A singularly perturbed Markov chain.}
\label{fig_Markovchain}
\end{figure}
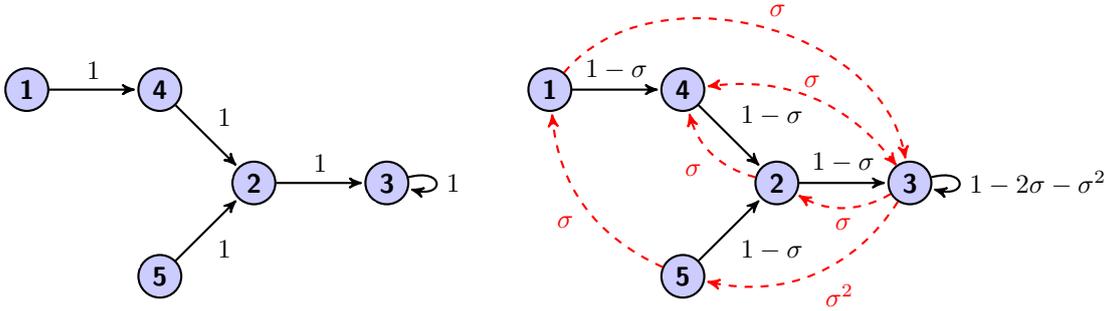

The results obtained here are a first step towards the understanding of
stochastic MMOs, that calls for further work. In particular, it would be
desirable to obtain a more precise description of the possible MMO patterns.
Let us mention two possible ways to achieve this:

\begin{enum}
\item 	\emph{Singularly perturbed Markov chains:\/} Consider the ideal case
where the sectors of rotation form a Markov partition, meaning that the image of
each sector is entirely contained in a sector. Then the dynamics can be
described by a topological Markov chain between sectors,
see~\figref{fig_Markovchain}. Such a chain will in general not be irreducible,
for instance,  for the chain shown in~\figref{fig_Markovchain}, state~$3$ is an absorbing state. 
In the presence of noise, however, new transitions between sectors appear, typically yielding an irreducible chain. In this sense,
the chain for the stochastic system is a singular perturbation of its
deterministic limit. For weak, non-vanishing noise, transitions between all states may
become possible, but transition times diverge as the noise intensity goes to
zero. Methods allowing to determine transition rates in singularly perturbed
Markov chains for small positive noise have been developed, for instance,
in~\cite{Schweitzer_68,Hassin_Haviv_92,AvrachenkovLasserre99,YinZhang1}. 

\item 	\emph{Metastable transitions between periodic orbits:\/} Consider a
situation where the deterministic system admits several stable periodic orbits,
each corresponding to an MMO pattern. Weak noise will induce rare transitions
between these orbits. The theory of large deviations~\cite{FreidlinWentzell}
provides a way to estimate the exponential asymptotics of transition rates
(Arrhenius' law~\cite{Arrhenius}), via a variational problem. In the reversible
case, Kramers' law~\cite{Eyring,Kramers} provides a more precise expression for
transition rates, which are related to exponentially small eigenvalues of the
diffusion's infinitesimal generator, see for instance~\cite{BEGK,BGK}, 
and~\cite{Berglund_irs_MPRF} for a recent review. For irreversible systems, such
precise expressions for transition rates are not available. However, the
spectral-theoretic approach may still yield useful information, as in similar irreversible problems involving random Poincar\'e
maps~\cite{BerglundLandon,BG_periodic2}. 
\end{enum}

\medskip

\noindent
\textbf{Notations:} We write $\left|\cdot\right|$ to denote  the absolute value and
$\left\|\cdot\right\|$ for the Euclidean norm. For $x\in\R$ we write $\intpartplus{x}$ for
the smallest integer not less than $x$. Furthermore, for $a,b\in\R$ we use
$a\wedge b:=\min\{a,b\}$ and $a\vee b:=\max\{a,b\}$. Regarding asymptotics, we
use $\cO(\cdot)$ in the usual way, i.e., we write $f(x)=\cO(g(x))$  as $x\ra x^*$ if and
only if $\limsup_{x\ra x^*} \bigl|f(x)/g(x)\bigr| <\I$. The shorthand $f(x)\asymp g(x)$ is used whenever $f(x)=\cO(g(x))$ and $g(x)=\cO(f(x))$ hold simultaneously. Furthermore, by $f(x)\ll g(x)$ we indicate that $\lim_{x\ra x^*}  \bigl|f(x)/g(x)\bigr|
=0$. Vectors are assumed to be column vectors and $\transpose{v}$ denotes the transpose of a vector~$v$.

\medskip

\noindent
\textbf{Acknowledgements:} {C.K.}~would like to thank the Austrian 
Academy of Sciences ({\"{O}AW}) for support via an APART fellowship as well as 
the European Commission (EC/REA) for support by a Marie-Curie International 
Re-integration Grant. {B.G.}~and {C.K.}~thank the MAPMO at the
Universit\'e d'Orl\'eans, {N.B.}~and {C.K.}~thank the CRC 701 at the University
of Bielefeld for kind hospitality and financial support.
Last but not least, we would like to thank an anonymous referee for
constructive comments that led to improvements in the presentation.


\section{Mixed-mode oscillations -- The setup} 
\label{sec:setup}

In this section, we shall outline a typical setup for deterministic mixed-mode oscillations based
upon three-dimensional fast--slow systems of the form
\begin{align}
\nonumber
\eps\tfrac{\6x}{\6s} = \eps\dot{x} &{}= f(x,y,z;\eps)\;,\\
\label{eq:global_main}
\tfrac{\6y}{\6s} = \phantom{\eps}\dot{y} &{}= g_1(x,y,z;\eps)\;,\\
\nonumber
\tfrac{\6z}{\6s} = \phantom{\eps}\dot{z} &{}= g_2(x,y,z;\eps)\;,
\end{align}
where $(x,y,z)\in\R^3$ and $0<\eps\ll1$ is a small parameter. Throughout, we shall make the
following assumption:
\begin{itemize}
 \item[(A0)] The functions $f,g_1,g_2:\R^4\ra \R$ are of class $C^3$.
\end{itemize} 
In particular, (A0) implies that on a fixed compact set there exist uniform
bounds
on $f,g_1,g_2$. We remark that the system~\eqref{eq:global_main} is 
allowed to depend smoothly upon further system parameters $\mu\in\R^p$ although we do
not indicate this dependence in the notation. The critical set of~\eqref{eq:global_main} is
\benn
C_0=\{(x,y,z)\in\R^3:f(x,y,z;0)=0\}\;.
\eenn
Motivated by several applications, such as the Hodgkin--Huxley model~\cite{HodgkinHuxley,RubinWechselberger1}, 
the Koper model~\cite{Koper,KuehnRetMaps}, the forced van der Pol equation~\cite{vanderPol_forced,GuckFT} and 
the R\"{o}ssler model~\cite{Roessler1}, we will  assume that the
geometric structure of the
critical set is an S-shaped smooth manifold; see also \figref{fig_sections}.
More
precisely, this assumption can be stated as follows:

\begin{itemize}
 \item[(A1)] Suppose $C_0$ is a smooth manifold composed of five smooth 
submanifolds, 
\benn
C_0=C^{a-}_0\cup L^-\cup C^{r}_0 \cup L^+ \cup C^{a+}_0\;,
\eenn
where the two-dimensional submanifolds $C^{a\pm}_0$ are normally hyperbolic attracting, while the two-dimensional submanifold
$C^r_0$ is normally hyperbolic repelling, {i.e.},
 \benn
  \frac{\partial f}{\partial x}(p;0)<0\ \ \forall p\in C^{a\pm}_0
  \quad \text{and}\quad
  \frac{\partial f}{\partial x}(p;0)>0\ \ \forall p\in C^{r}_0\;,
 \eenn  
and $L^\pm$ are one-dimensional smooth fold curves consisting of generic fold
points 
\benn
f(p;0)=0\;,
\quad 
\frac{\partial f}{\partial x}(p;0)=0\;,
\quad
\frac{\partial^2 f}{\partial x^2}(p;0)\neq 0\;,
\quad 
\begin{pmatrix}
\partial_y f(p;0) \\ \partial_z f(p;0)
\end{pmatrix}
\neq 
\begin{pmatrix}0\\ 0\\ \end{pmatrix}
\qquad \forall p\in L^\pm\;.
\eenn
Without loss of generality we assume from now on that 
$\partial_y f(p;0)\neq 0$ for all $p\in L^\pm$.
\end{itemize}

\begin{figure}
\centerline{\includegraphics*[clip=true,width=115mm]{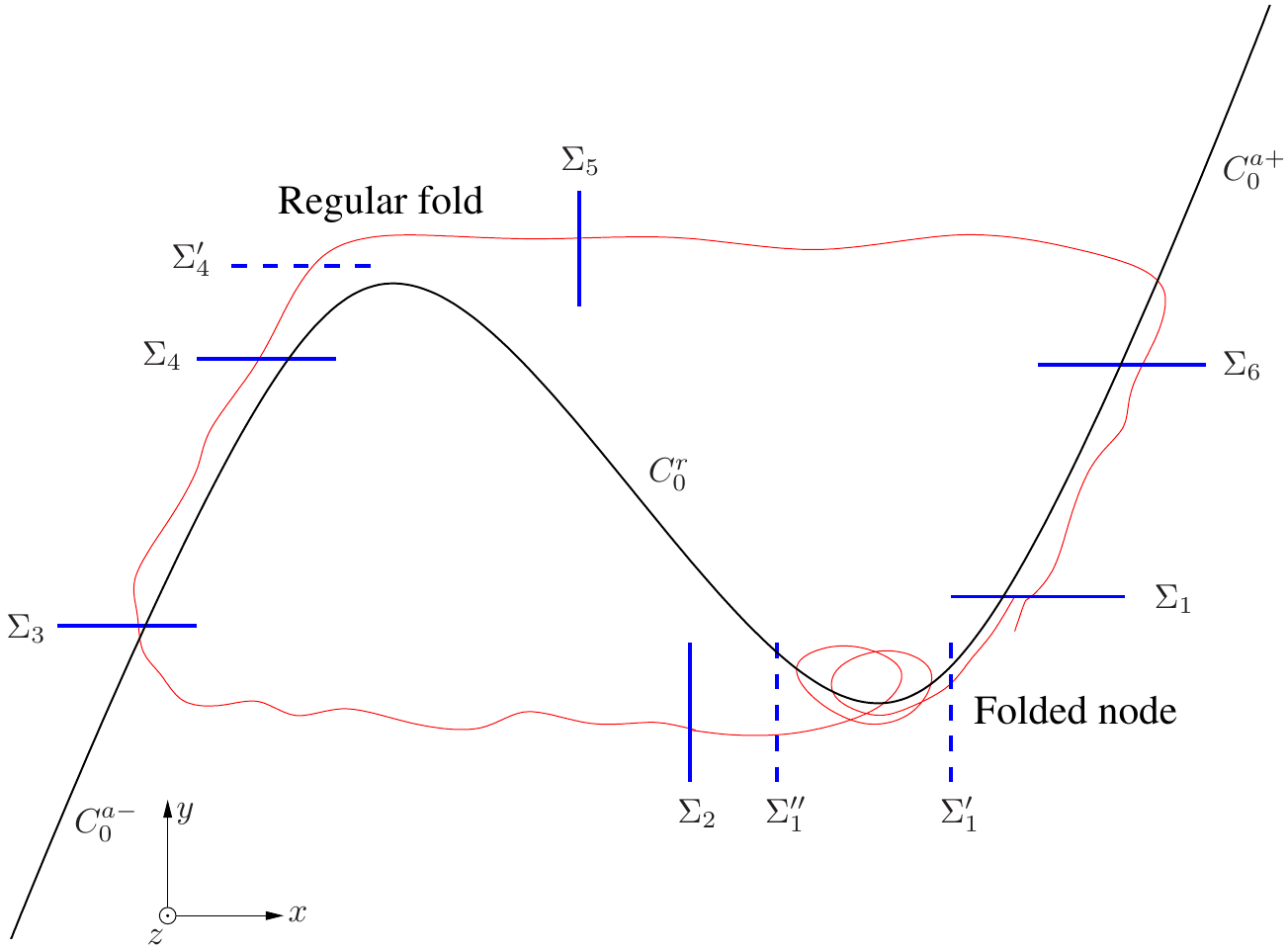}}
 \vspace{2mm}
\caption[]{Sketch illustrating the definition of the different sections. The
horizontal 
coordinate is $x$, the vertical one is $y$, and $z$ points out of the plane.}
\label{fig_sections}
\end{figure}

Fenichel theory~\cite{Fenichel} shows that for $\eps>0$, the critical 
submanifolds $C^{a\pm}_0$ and $C^{r}_0$ perturb to invariant manifolds 
$C^{a\pm}_\eps$ and $C^{r}_\eps$, which are $\eps$-close to $C^{a\pm}_0$ and
$C^{r}_0$ in points bounded away from the fold curves $L^\pm$.

Setting $\eps=0$ in~\eqref{eq:global_main} leads to the slow subsystem
\begin{align}
\nonumber
0 &{}= f(x,y,z;0)\;,\\
\label{eq:slow_sub}
\dot{y} &{}= g_1(x,y,z;0)\;,\\
\dot{z} &{}= g_2(x,y,z;0)\;,
\nonumber
\end{align}
which is solved by the so-called slow flow. Differentiating $f$ implicitly with
respect to $s$ yields
\benn
\frac{\partial f}{\partial x}\dot{x}
=-\frac{\partial f}{\partial y}\dot{y} -\frac{\partial f}{\partial z}\dot{z}
=-\frac{\partial f}{\partial  y}g_1-
\frac{\partial f}{\partial z}g_2
\eenn
for the slow flow, so that the slow subsystem~\eqref{eq:slow_sub} can be written as
\begin{align}
\nonumber
\frac{\partial f}{\partial x}~\dot{x}&{}=
 -\frac{\partial f}{\partial y}g_1
-\frac{\partial f}{\partial z}g_{2}\;,\\
\dot{z} &{}= g_2\;,
\label{eq:slow_sub1}
\end{align} 
where it is understood that all functions are evaluated at points $(x,y,z;\eps)=(p;0)$ with
$p\in C_0$. One may use that~\eqref{eq:slow_sub1} can locally be written
as a closed system by applying the implicit-function theorem to express $C_0$
as a graph, {e.g.}~$y=h(x,z)$, near the fold as $\frac{\partial f}{\partial
y}\neq 0$.

Observe that~\eqref{eq:slow_sub1} is singular on the fold curves as
$\frac{\partial f}{\partial x}=0$ on $L^\pm$. The desingularized slow subsystem
is obtained by multiplying the right-hand side of~\eqref{eq:slow_sub1} by
$\frac{\partial f}{\partial x}$ and applying a rescaling of time. It reads
\begin{align}
\nonumber
\dot{x} &{}= -\frac{\partial f}{\partial y}g_1-\frac{\partial f}{\partial z}g_2\;,\\
\dot{z} &{}= \frac{\partial f}{\partial x}g_2\;.
\label{eq:slow_sub2}
\end{align}
We make the following further assumptions: 
\begin{itemize}
 \item[(A2)] Suppose all fold points on $L^-$ satisfy the normal switching
 condition~\cite{MKKR_B,SzmolyanWechselbergerRelax}
 \benn
 \begin{pmatrix} \frac{\partial f}{\partial y} (p;0) \\ 
 \frac{\partial f}{\partial z} (p;0)\\ \end{pmatrix}\cdot
 \begin{pmatrix}g_1(p;0) \\ g_2(p;0)\\ \end{pmatrix}\neq 0
  \qquad \forall p\in L^-.
 \eenn
 Furthermore, assume that the projections of $L^\pm$ along the $x$-coordinate 
 onto $C^{a\mp}_0$, which are also called the drop curves, are transverse to 
 the slow flow.
 \item[(A3)] Assume that the normal switching condition fails only at
 a unique singularity $p^*\in L^+$ and $p^*$ is a node equilibrium point
 of~\eqref{eq:slow_sub2}; in this case, $p^*$ is called a folded node (or
 folded-node singularity)~\cite{Benoit2,Wechselberger}.  
\end{itemize}
Let us stress that the above geometric assumptions (A1)--(A3), as
well as several further assumptions to follow, provide a convenient framework
but that the deterministic and stochastic techniques we present here
apply to a much wider range of multiscale systems displaying oscillatory
patterns.

On the fast timescale $t=s/\eps$ the limit $\eps\ra 0$ of~\eqref{eq:global_main} leads to the fast subsystem
\begin{align}
\nonumber
\tfrac{\6x}{\6t}&{}=x' = f(x,y,z;0)\;,\\
\tfrac{\6y}{\6t}&{}=y' = 0\;,\\
\nonumber
\tfrac{\6z}{\6t}&{}=z' = 0\;,
\end{align}
which is solved by the fast flow. It is helpful to decompose the singular
limit flows and their perturbations into several parts; see \figref{fig_sections} 
for an illustration. In particular, we consider the sections of the form
\label{eq:sections}
\begin{align}
\Sigma_i&{}:={}\{(x,y,z)\in\R^3:x=x_i,y\in[y_{i,a},y_{i,b}],z\in[z_{i,a},z_{i,b}]\}\;,
\qquad i\in\{2,5\}\;,\\
\nonumber
\Sigma_i&{}:=\{(x,y,z)\in\R^3:y=y_i,x\in[x_{i,a},x_{i,b}],z\in[z_{i,a},z_{i,b}]\}\;,
\qquad i\in\{1,3,4,6\}\;,
\end{align}
for $x_{i,a}<x_{i,b}$, $y_{i,a}<y_{i,b}$, $z_{i,a}<z_{i,b}$ suitably chosen
to capture the return map. For an appropriate choice of the constants $x_i$ and
$y_i$
(see below or consider the approach in~\cite{KuehnRetMaps}), there are
well-defined maps from $\Sigma_i$ to $\Sigma_j$.
\begin{itemize}
 \item[(A4)] The geometry of the flow-induced maps and sections is as shown
in~\figref{fig_sections}. 
\end{itemize}
In particular, Assumption~(A4) implies that there is an $\cO(1)$ transition time
on the slow timescale from $\Sigma_3$ to $\Sigma_4$ as well as from $\Sigma_6$ to 
$\Sigma_1$. (A4)~incorporates that there is an $\cO(1)$ spatial
separation between each pair of fold/drop curves and it guarantees there is an
$\cO(1)$ transition time on the fast timescale from $\Sigma_2$ to $\Sigma_3$ 
as well as from $\Sigma_5$ to $\Sigma_6$. Furthermore, we exclude the case
of a singular Hopf bifurcation~\cite{GuckenheimerSH,GuckenheimerMeerkamp}, 
where an equilibrium of the full system~\eqref{eq:global_main} may occur in the 
neighbourhood of a folded node.

There are four distinct important parts of the flow to analyze: 
\begin{itemiz}
\item[(I)] 
the flow near the folded node $\Sigma_1\ra\Sigma_2$, 
\item[(II)] 
the fast segment $\Sigma_2\ra\Sigma_3$, 
\item[(III)] 
the slow-flow region $\Sigma_{3}\ra\Sigma_4$ near $C_\eps^{a-}$, and 
\item[(IV)] 
the non-degenerate fold via $\Sigma_4\ra\Sigma_5$. 
\end{itemiz}
The map $\Sigma_5\ra\Sigma_6$ can be covered by the same techniques as 
$\Sigma_2\ra\Sigma_3$, and $\Sigma_6\ra\Sigma_1$ is similar to
$\Sigma_3\ra\Sigma_4$.

The geometry of flow maps and the possible generation mechanisms for mixed-mode oscillations under the assumptions
(A0)--(A4) are well-known; see for example~\cite{BronsKrupaWechselberger,KuehnMMO}.
A main idea is that twisting of the slow manifolds $C_{\eps}^{a+}$ and $C_\eps^r$
near a folded node generates SAOs and the global return mechanism via the
S-shaped critical manifold induces the LAOs. Fixed points of a full return map,
say $\Sigma_1\ra \Sigma_1$, correspond to MMOs with a certain pattern
\begin{equation}
\label{eq:main_pattern}
\cdots L_{k}^{s_{k}} L_{k+1}^{s_{k+1}}\cdots L_{k+l}^{s_{k+l}}
L_{k+1}^{s_{k+1}}\cdots
\end{equation}
The main question we address in this paper is how noise influences the patterns~\eqref{eq:main_pattern}. We
are going to split  the analysis into two main parts. In Section~\ref{sec:global_deviate}
we provide basic estimates and consider the \emph{global\/} part of the return map. 
Sections~\ref{sec:fold}--\ref{sec:local_deviate} address \emph{local\/} dynamics
in the regions near the regular fold and the folded node.

\section{The stochastic system}
\label{sec:global_deviate}

\subsection{Estimating stochastic deviations}
\label{ssec:global_deviate}

As a stochastic extension to~\eqref{eq:global_main} we consider the 
fast--slow SDE
\begin{align}
\nonumber
\6x_s &{}= \frac{1}{\eps} f(x_s,y_s,z_s)\6s + \frac{\sigma}{\sqrt{\eps}}
F(x_s,y_s,z_s)\6W_s\;, \\
\label{SDE} 
\6y_s &{}= g_1(x_s,y_s,z_s)\6s + \sigma' G_1(x_s,y_s,z_s)\6W_s\;,\\[6pt]
\6z_s &{}= g_2(x_s,y_s,z_s)\6s + \sigma' G_2(x_s,y_s,z_s)\6W_s\;, 
\nonumber
\end{align}
where $(W_s)_{s\geqs0}$ is a  $k$-dimensional standard Brownian motion on a
probability space $(\Omega,\mathcal{F},\P)$. The maps
\begin{equation}
F(x,y,z)\in\R^{1\times k},\qquad
G(x,y,z)={}
\begin{pmatrix}G_1(x,y,z) \\ G_2(x,y,z)\end{pmatrix} \in\R^{2\times k}
\end{equation}
may depend on $\eps$, and are assumed to be $C^1$ and to satisfy the usual
bounded-growth condition guaranteing existence of a unique strong solution
of~\eqref{SDE}. We shall adopt the shorthand 
notation to write just $(x,y,z)$ instead of $(x,y,z;\eps)$. 

We will assume that the diffusion coefficients satisfy the following
uniform ellipticity assumption: 
\begin{itemize}
 \item[(A5)] Let 
 \begin{equation}
  \label{eq:diffusion_matrix}
  D(x,y,z) = 
  \begin{pmatrix}
  F\transpose{F}  (x,y,z) & F\transpose{G}   (x,y,z)\\
  G\transpose{F}   (x,y,z)& G\transpose{G}  (x,y,z)
  \end{pmatrix}
 \in \R^{3\times3}
 \end{equation} 
 be the diffusion matrix. There exist constants $c_+ \geqs c_- > 0$  such that 
 \begin{equation}
  \label{eq:ellipticity} 
  c_- \norm{\xi}^2 \leqs \langle \xi, D(x,y,z)\xi \rangle 
  \leqs c_+ \norm{\xi}^2 
  \qquad \forall \xi\in\R^3\quad  \forall\transpose{(x,y,z)}\in\R^3\;.
 \end{equation} 
\end{itemize}

\begin{rem}
In fact, most of our results remain valid under a weaker hypoellipticity
assumption (cf.~\cite[p.~175]{BenArous_Kusuoka_Stroock_1984} -- this weaker condition
is needed for the random Poincar\'e map to have a smooth density). The only
result that requires the lower bound in~\eqref{eq:ellipticity}
is Theorem~\ref{thm_local_returns_outer}, which relies on the early-escape
result~\cite[Theorem~6.4]{BGK12}. See~\cite{Hoepfner_Loecherbach_Thieullen} for recent work under weaker assumptions.
\end{rem}

Finally we make the following assumption on the noise intensities:
\begin{itemize}
 \item[(A6)] Assume $0<\sigma=\sigma(\eps)\ll 1$ and $0<\sigma'=\sigma'(\eps)\ll
1$. 
\end{itemize}

In fact, in the course of the analysis, we will encounter more restrictive
conditions of the form $\sigma = \Order{\eps^\alpha}$, $\sigma' =
\Order{\eps^\beta}$ with $\alpha, \beta >0$. The most stringent
of these conditions will be needed for the analysis near the folded node, and
requires $\sigma, \sigma' = \Order{\eps^{3/4}}$. 

The main goal is to establish bounds on the noise-induced deviation from a 
deterministic solution. In~\cite[Theorem~5.1.18]{BGbook},  rather precise bounds for the deviation 
near normally hyperbolic critical manifolds are derived. We want to adapt these to the
other phases of motion. As it turns out, the leading-order effect of noise occurs near the
folded-node singularity. Therefore, it will be sufficient to determine the 
order of magnitude of noise-induced deviations during other phases of the
dynamics, as a function of $\sigma, \sigma'$ and~$\eps$. 

We fix a deterministic reference solution $(x^{\det}_s,y^{\det}_s,z^{\det}_s)$ and set 
\begin{equation}
\label{def_zeta} 
\xi_s=x_s-x^{\det}_s,\qquad 
\eta_s= \begin{pmatrix}y_s\\ z_s \end{pmatrix}- \begin{pmatrix}y^{\det}_s\\ z^{\det}_s \end{pmatrix},
\qquad  \zeta_s= \begin{pmatrix}\xi_s\\ \eta_s \end{pmatrix}.
\end{equation}
As initial condition we choose $(\xi_0,\eta_0)=(0,0)$ as it corresponds to 
$(x^{\det}_0,y^{\det}_0,z^{\det}_0)$. Substituting in~\eqref{SDE} and 
Taylor-expanding, we obtain a system of the form 
\begin{equation}
\label{SDE_zeta} 
 \6\zeta_s = \frac{1}{\eps} \cA(s)\zeta_s\, \6s + 
 \begin{pmatrix}
 \frac{\sigma}{\sqrt{\eps}} \cF(\zeta_s,s) \\ \sigma' \cG(\zeta_s,s)
 \end{pmatrix}
 \6W_s + 
 \begin{pmatrix}
 \frac{1}{\eps} b_\xi(\zeta_s,s) \\ b_\eta(\zeta_s,s) 
 \end{pmatrix}
 \6s\;,
\end{equation} 
where 
\begin{equation}
\cA(s)\in\R^{3\times 3}\;,
\qquad \cF(\zeta_s,s),\ b_\xi(\zeta_s,s)\in\R\;,
\qquad  
\cG(\zeta_s,s),\ b_\eta(\zeta_s,s)\in\R^2\;.
\end{equation}
The nonlinear terms $b_{\xi}$ and $b_{\eta}$ satisfy $b_\cdot(\zeta_s,s) = \Order{\norm{\zeta}^2}$ 
as $\norm{\zeta}\ra 0$. The matrix $\cA(s)$ of the 
system linearized around the deterministic solution has the structure 
\begin{equation}
\label{At} 
\cA(s) = 
 \begin{pmatrix}
  a(s) & c_1(s) \\ \eps c_2(s) & \eps B(s)
 \end{pmatrix}\;,
\end{equation} 
where $a(s)=\frac{\partial f}{\partial x}(x^{\det}_s,y^{\det}_s,z^{\det}_s)$ and
so on, so that
in
particular $c_1(s)\in\R^{1\times 2}$, $c_2(s)\in\R^{2\times1}$ and 
$B(s)\in\R^{2\times2}$. Let 
\begin{equation}
 U(s,r) = 
 \begin{pmatrix}
 U_{\xi\xi}(s,r) & U_{\xi\eta}(s,r) \\ U_{\eta\xi}(s,r) & U_{\eta\eta}(s,r)
 \end{pmatrix}
\end{equation}
denote the principal solution of the linear system 
$\eps\dot\zeta = \cA(s)\zeta$. Then the solution of~\eqref{SDE_zeta} can be
written in the form 
\begin{align}
\nonumber
 \xi_s ={}& \frac{\sigma}{\sqrt{\eps}} \int_0^s U_{\xi\xi}(s,r) \cF(\zeta_r,r)
\6W_r + \sigma' \int_0^s U_{\xi\eta}(s,r) \cG(\zeta_r,r) \6W_r  \\
&{}+ \frac{1}{\eps} \int_0^s U_{\xi\xi}(s,r) b_\xi(\zeta_r,r) \6r + 
\int_0^s U_{\xi\eta}(s,r) b_\eta(\zeta_r,r) \6r\;, 
\label{xi} 
\end{align} 
and 
\begin{align}
\nonumber
 \eta_s ={}& \frac{\sigma}{\sqrt{\eps}} \int_0^s U_{\eta\xi}(s,r) \cF(\zeta_r,r)
\6W_r + \sigma' \int_0^s U_{\eta\eta}(s,r) \cG(\zeta_r,r) \6W_r  \\
&{}+ \frac{1}{\eps} \int_0^s U_{\eta\xi}(s,r) b_\xi(\zeta_r,r) \6r + 
\int_0^s U_{\eta\eta}(s,r) b_\eta(\zeta_r,r) \6r\;.  
\label{eta} 
\end{align} 
In both {equations}, we expect the stochastic integrals to give the leading
contribution to the fluctuations. They can be estimated by the Bernstein-type 
inequality Lemma~\ref{lem_app_Bernstein}. {The magnitude of the other integrals 
can then be shown to be smaller, using a direct estimate which is valid as long as the system does not exit from the region where the nonlinear terms are negligible}; see {e.g.}~\cite[p.~4826]{BGK12} or~\cite[Theorem~2.4]{BG6}.

In order to carry out this program, we need estimates on the elements of the
principal solution $U$. Note that the $\xi$-components are in principle larger
than the $\eta$-components,
but this is compensated by the fact that $x^{\det}_s$ spends most of the time
in the vicinity of stable critical manifolds. The following ODEs will 
play an important r\^ole:
\begin{align}
\nonumber
\eps\dot p_1 &{}= c_1(s) + a(s) p_1 - \eps p_1 B(s) - \eps p_1 c_2(s) p_1\;, \\
\eps\dot p_2 &{}= c_2(s) - a(s) p_2 + \eps B(s) p_2 - \eps p_2 c_1(s) p_2\;.
\label{s-eqn}
\end{align}
Here $p_1(s)\in\R^{1\times2}$ and $p_2(s)\in\R^{2\times1}$. If $a(s)$ is
bounded away from $0$, standard singular perturbation theory implies that 
these ODEs admit solutions $p_1(s)$ and $p_2(s)$ of order $1$ (and in fact
$p_1(s)$ close to $-a(s)^{-1}c_1(s)$). If $a(s)$ approaches $0$ or changes sign,
this need no longer be the case, but there may still be solutions such that
$\eps \abs{p_1(s) p_2(s)}$ remains small. 

\begin{lem}
\label{lem_U} 
Assume $s-r\leqs\Order{1}$ and that the ODEs~\eqref{s-eqn} admit solutions 
such that $\eps \abs{p_1(u) p_2(u)}$ is bounded for $u\in[r,s]$ by a function 
$\rho(\eps)$ satisfying $\lim_{\eps\to0}\rho(\eps)=0$. 
Let $\alpha(s,r)=\int_r^s a(u)\6u$. 
Then for sufficiently small $\eps$, 
\begin{align}
\nonumber
U_{\xi\xi}(s,r) &{}= 
\bigl[\e^{(\alpha(s,r)+\Order{\eps})/\eps} - \eps~
p_1(s)Vp_2(r)\bigr](1+\Order{\rho}) \;, \\
\nonumber
U_{\xi\eta}(s,r) &{}=
\bigl[-\e^{(\alpha(s,r)+\Order{\eps})/\eps}p_1(r) +
p_1(s)V\bigr](1+\Order{\rho}) \;, \\
\label{U_size}
U_{\eta\xi}(s,r) &{}= \eps 
\bigl[\e^{(\alpha(s,r)+\Order{\eps})/\eps}p_2(s) 
-Vp_2(r)\bigr](1+\Order{\rho}) \;, \\
U_{\eta\eta}(s,r) &{}= 
\bigl[V - \eps \e^{(\alpha(s,r)+\Order{\eps})/\eps}
p_2(s)p_1(r)\bigr](1+\Order{\rho}) \;,
 \nonumber
\end{align}
where $V=V(s,r)$ is the principal solution of the system  
\begin{equation}
 \dot\eta = \left[ B(s)+c_2(s)p_1(s)\right] \eta\;.
\end{equation} 
\end{lem}

\begin{proof}
Consider the matrix 
\begin{equation}
 S(s) = 
 \begin{pmatrix}
 1 & p_1(s) \\ \eps p_2(s) & \one
 \end{pmatrix}\;. 
\end{equation} 
Then the equations~\eqref{s-eqn} imply 
\begin{equation}
\label{p1} 
 \eps\dot S = \cA S - S D
 \quad \text{with}\quad
 D(s)=\begin{pmatrix}
 d_1(s) & 0\\
 0 & \eps D_2(s)
 \end{pmatrix}
 \;,
\end{equation} 
where the blocks $d_1(s)\in\R$ and 
$\eps D_2(s)\in\R^{2\times2}$ are given by 
\begin{align}
\nonumber
d_1(s) &= a(s) + \eps c_1(s) p_2(s)\;, \\
\eps D_2(s) &= \eps B(s) + \eps c_2(s) p_1(s)\;.
\end{align}
Consider now the variable $\zeta_1=S(s)^{-1}\zeta$. If
$\eps\dot\zeta=\cA(s)\zeta$, then~\eqref{p1} implies 
\begin{equation}
 \eps\dot\zeta_1 = D(s) \zeta_1\;.
\end{equation} 
The principal solution of this equation is block-diagonal, with blocks 
$\e^{\frac{1}{\eps}\int_r^s d_1(u)\6u}$ and $V(s,r)$, where $V$ is the principal
solution of $\dot\eta=D_2(s)\eta$. The principal
solution of the original equation is then given by 
\begin{equation}
\label{p2} 
 U(s,r) = S(s) 
 \begin{pmatrix}
 \e^{\frac{1}{\eps}\int_r^s d_1(u)\6u} & 0 \\ 0 & V(s,r)
 \end{pmatrix}
 S(r)^{-1}\;.
\end{equation}
Furthermore, we have 
\begin{equation}
 S(s)^{-1} = 
 \begin{pmatrix}
 1 & -p_1(s) \\ -\eps p_2(s) & \one 
 \end{pmatrix}
 \begin{pmatrix}
 [1-\eps p_1(s)p_2(s)]^{-1} & 0 \\ 0 & [\one-\eps p_2(s)p_1(s)]^{-1}
 \end{pmatrix}\;.
\end{equation} 
Computing the matrix product in~\eqref{p2} yields the result. Note that more
precise expressions for the matrix elements can be obtained if needed. 
\end{proof}

To describe the size of fluctuations, for given $h,h_1>0$
we introduce stopping times 
\begin{align}
\label{def_tau} 
\nonumber
\tau_\xi &= \inf \bigl\{s>0 \colon \abs{\xi_s}> h \bigr\}\;, \\
\tau_\eta &= \inf \bigl\{s>0 \colon \norm{\eta_s}> h_1 \bigr\}\;.
\end{align}

\begin{prop}
\label{prop_tau} 
Suppose the assumptions of Lemma~\ref{lem_U} are satisfied with $p_{1}$, $p_{2}$ bounded uniformly in~$\eps$. Given a finite time horizon~$T$ of order~$1$ on the slow timescale, there exist constants $\kappa, h_0>0$ such that whenever 
$h,h_1\leqs h_0$, $h_1^2\leqs h_0h$ and $h^2\leqs h_0h_1$, 
\begin{equation}
\label{pprop} 
 \P \bigl\{ \tau_\xi \wedge \tau_\eta < s \bigr\} 
 \leqs
 \biggintpartplus{\frac s\eps}\left(
 \e^{-\kappa h^2/\sigma^2}
 + \e^{-\kappa h^2/(\sigma')^2}
 + \e^{-\kappa h_1^2/(\sigma')^2}
 + \e^{-\kappa h_1^2/(\eps\sigma^2)}\right)
\end{equation} 
holds for all $s\leqs T$. 
\end{prop}

\begin{proof}
Denote by $\xi^i_s$, $i=0,1,2,3$, the four terms on the right-hand side
of~\eqref{xi}. We will start by estimating $\xi^{0}_{s}$ and $\xi^{1}_{s}$. 
Since $p_{1}$, $p_{2}$ are assumed to be bounded, we may choose $\rho(\eps)$ of order $\eps$ in~\eqref{U_size}, and $U_{\eta\xi}$ is of order~$\eps$, while the other elements of $U$ are of order~$1$ at most.
By Lemma~\ref{lem_app_Bernstein} and the bounds on $U_{\xi\xi}$
and $U_{\xi\eta}$, there exists a constant $M>0$ such that 
\begin{equation}
\label{eq:first_two_ineq}
 \P \biggl\{ \sup_{0\leqs r\leqs s} \abs{\xi^0_r} > h \biggr\}
 \leqs \biggintpartplus{\frac s\eps} \e^{-h^2/(M\sigma^2)} 
 \quad\text{and}\quad 
 \P \biggl\{ \sup_{0\leqs r\leqs s} \abs{\xi^1_r} > h \biggr\}
 \leqs \biggintpartplus{\frac s\eps} \e^{-h^2/(M(\sigma')^2)}\;.
\end{equation} 
Indeed, to estimate $\xi^{0}_{s}$ we first use that on any short time interval $s\in[s_{1},s_{2}]$ with $\abs{s_{2}-s_{1}}\leqs \eps$, the stochastic process $\xi^{0}_{s}= U_{\xi\xi}(s,s_{2})\mathcal{M}_{s}$ is close to the martingale $\mathcal{M}_{s}$, defined by
\begin{equation}
{\mathcal M}_{s} = \frac{\sigma}{\sqrt\eps} \int_{0}^{s} U_{\xi\xi}(s_{2},r)\cF(\zeta_r,r)\6r\;,
\end{equation}
since $\abs{U_{\xi\xi}(s,s_{2})}$ remains of order~1 on these time intervals. 
First using~\eqref{U_size} and (A0) and then our choice of $\rho$ and Lemma~\ref{lem_app_scaling}, we see that the martingale's variance is bounded by
\begin{equation}
\frac{\sigma^{2}}{{\eps}}\int_0^s \left| U_{\xi\xi}(s_{2},r) (\cF\transpose{\cF})(\zeta_r,r)\transpose{U_{\xi\xi}}(s_{2},r)\right|\6r
\leqs 
\widetilde{M} \frac{\sigma^2}{\eps}
\int_0^s \left[\e^{(2\alpha(s,r)+\Order{\eps})/\eps}+\rho^{2} \norm{V}^{2}\right]
\6r 
\end{equation}
for some positive constant~$\widetilde{M}$. Thus the variance is at most of order $\sigma^{2}$ for all $s\in[s_{1}, s_{2}]$. Now the first inequality in~\eqref{eq:first_two_ineq} follows immediately from the Bernstein-type estimate Lemma~\ref{lem_app_Bernstein}. The prefactor in~\eqref{eq:first_two_ineq} simply counts the number of intervals $[s_{1},s_{2}]$ needed to cover $[0,s]$, see e.g.~\cite[Proposition~3.15]{BGbook} for a detailed proof in a simpler, one-dimensional setting. The second inequality in~\eqref{eq:first_two_ineq} is shown similarly.

Furthermore, we have $\abs{\xi^2_s}+\abs{\xi^3_s}\leqs
M'(h^2+h_1^2)$ 
for a constant~$M'>0$ and~$s\leqs \tau_\xi \wedge \tau_\eta$. From this, together with Gronwall's lemma,
we deduce that there exists a constant~$M>0$ such that
\begin{equation}
 \P \bigl\{ \tau_\xi < s\wedge\tau_\eta \bigr\}
 \leqs  \biggintpartplus{\frac s\eps}
 \exp \biggl\{ - \frac{[h - M'(h^2+h_1^2)]^2}{M\sigma^2}\biggr\}
 + \biggintpartplus{\frac s\eps} \e^{-h^2/(M(\sigma')^2)}\;.
\end{equation} 
In a similar way, we find 
\begin{equation}
 \P \bigl\{ \tau_\eta < s\wedge\tau_\xi \bigr\}
 \leqs 
  \biggintpartplus{\frac s\eps}\exp \biggl\{ - \frac{[h_1 - M'(h^2+h_1^2)]^2}{M(\sigma')^2}\biggr\}
 +  \biggintpartplus{\frac s\eps}\e^{-h_1^2/(M\eps\sigma^2)}\;.
\end{equation} 
Choosing $h_0$ small enough, we can ensure that the terms $M'(h^2+h_1^2)$ 
are negligible, and the result follows by taking the sum of the last two
estimates.
\end{proof}

The size of typical fluctuations is given by the values of $h,h_1$ for which
the probability~\eqref{pprop} starts getting small, namely $h\gg\sigma\vee \sigma'$
and $h_1\gg\sigma'\vee \sigma\sqrt{\eps}$. We conclude that fluctuations have size 
$\sigma + \sigma'$ in the fast direction, and $\sigma'+\sigma\sqrt{\eps}$
in the slow direction. Note that for simplicity we ignore the logarithmic contribution arising from the prefactor $\intpartplus{s/\eps}$.

We now want to estimate the noise-induced spreading for the Poincar\'e map,
starting on the section $\Sigma_2$ after the folded node, and arriving on the 
section $\Sigma_1$ before the folded node. As described in Section~\ref{sec:setup} 
we decompose
the map into several maps, see~\figref{fig_sections}, and estimate the 
spreading for each map separately. This means that we fix an initial 
condition on each section, and estimate the deviation of the stochastic 
sample paths from the deterministic solution when it first hits the next
section. 

\subsection{The fast segments}
\label{ssec:fast} 

The fast segments are given by $\Sigma_2 \to \Sigma_3$ and $\Sigma_5 \to
\Sigma_6$. By Assumption~(A4) there exists a slow time~$T_{0}$ of order $\eps$ in which the
deterministic solution starting on $\Sigma_2$ reaches  a neighbourhood of order
$1$ of the stable critical manifold. In this neighbourhood, the linearization
$a(s)=\frac{\partial f}{\partial x}(x^{\det}_s,y^{\det}_s,z^{\det}_s)$ is
negative and of order $1$. To reach an $\eps$-neighbourhood of the critical
manifold, an additional slow time~$T_1$ of order $\eps\abs{\log\eps}$ is
required. By the drop-curve transversality assumption~(A2) and using~(A4), it
takes another slow time~$T_2$ of at most order~$1$ to reach the section $\Sigma_3$. For $T:=T_{0} +T_1+T_2$ we thus have 
\begin{equation}
 a(s) \leqs 
 \begin{cases}
 a_1 & \text{for all $s$} \\
 -a_2 & \text{for $c_1\eps \leqs s \leqs T$}
 \end{cases}
 \end{equation} 
 for some positive constants $a_1,a_2,c_1$. This implies that whenever 
 $T\geqs s>r\geqs0$, 
 \begin{equation}
 \alpha(s,r) \leqs c_2 \eps 
 \end{equation}
for a constant $c_2$, and furthermore $\alpha(s,r)$ is negative as soon as 
$s$ is larger than a constant times $\eps$. 

Consider now the equations~\eqref{s-eqn}  for $p_1$ and $p_2$. 
We will show that $p_{1}$ remains bounded on $[0,T]$ and that there exists a particular solution~$p_{2}$ which also remains bounded on~$[0,T]$. For $p_1$, we proceed in two steps:
\begin{itemize}
\item 	For $0\leqs s\leqs c_1\eps$, $p_1(s)$ can grow at most by an amount of
order $1$.
\item 	For $c_1\eps < s \leqs T$, since $a(s)$ is negative, we can use standard
singular
perturbation theory to show that $p_1(s)$ remains of order $1$, and in fact
approaches $c_1(s)/\abs{a(s)}$. 
\end{itemize}
For $p_2(s)$, we change the direction of time and consider the equation 
\begin{equation}
 \dot p_2 = -c_2(T-s) + a(T-s) p_2 - \eps B(T-s) p_2 + \eps p_2 c_1(T-s) p_2\;.
\end{equation} 
We know that $a(T-s)$ is negative, bounded away from $0$, except for a time
interval of length $c_1\eps$ near $T$. Thus we conclude that there exists a
particular solution which remains bounded, of order $1$, on the whole time
interval. Therefore Lemma~\ref{lem_U} shows that $U_{\xi\xi}$,
$U_{\xi\eta}$
and $U_{\eta\eta}$ remain bounded (in norm), of order $1$, and that
$U_{\eta\xi}$ remains of order $\eps$ for $0\leqs r<s\leqs T$. As a
consequence, we can apply Proposition~\ref{prop_tau} as is, with the
result that on the section $\Sigma_3$,  
\begin{itemize}
 \item 	the spreading in the fast direction is of order $\sigma + \sigma'$,
 \item 	the spreading in the slow $z$-direction is of order 
 $\sigma' + \sigma\sqrt{\eps}$.
\end{itemize}
 
\subsection{The slow segments}
\label{ssec:slow} 

The slow segments are given by $\Sigma_3\to\Sigma_4$ and $\Sigma_6\to\Sigma_1$. 
The analysis of the previous subsection can actually be extended to these
segments, because $a(t)$ is always negative, bounded away from $0$. The
conclusions on typical spreading are the same:
\begin{itemize}
 \item 	the spreading in the fast direction is of order $\sigma + \sigma'$,
 \item 	the spreading in the slow $z$-direction is of order 
 $\sigma' + \sigma\sqrt{\eps}$.
\end{itemize}

Note that~\cite[Theorem~5.1.18]{BGbook} provides a more precise description of
the dynamics, by constructing more precise covariance tubes. The qualitative
conclusion on typical spreading is the same as above. 

\section{The regular fold}
\label{sec:fold} 

\subsection{Approach}
\label{ssec:rfa} 

The regular fold corresponds to the transition $\Sigma_4\to\Sigma_5$. We again
fix a deterministic solution, now starting on $\Sigma_4$. We choose the origin of
the coordinate system on the regular fold $L^-$ and the origin of time in such a
way that at time $s=0$, $(y^{\det}_0,z^{\det}_0)=(0,0)$. 

Recall from the deterministic analysis (see {e.g.}~\cite{KruSzm3,MisRoz} for 
the two-dimensional and~\cite{SzmolyanWechselbergerRelax,MKKR_B} for the 
three-dimensional case) that, given $s_0<0$ of order $1$, 

\begin{itemize}
 \item 
 for $s_0 \leqs s \leqs -\eps^{2/3}$, the distance of $x^{\det}_s$ to the
critical manifold grows like $\eps/\abs{s}^{1/2}$; 
 \item 
 there exists a $c_1>0$ such that $x^{\det}_s\asymp\eps^{1/3}$
 for $-\eps^{2/3} \leqs s \leqs c_1\eps^{2/3}$;
 \item 
 there exists a $c_2>0$ such that $x^{\det}_s$ reaches order
$1$ before time $c_2\eps^{2/3}$.
\end{itemize}

In this section, we consider the transition $\Sigma_4\to\Sigma'_4$, where 
$\Sigma'_4$ is a section on which $y=c_1\eps^{2/3}$. 
In this region, the linearization $a(s)=\frac{\partial f}{\partial
x}(x^{\det}_s,y^{\det}_s,z^{\det}_s)$ satisfies 
\begin{equation}
\label{rfa01} 
 a(s) \asymp - (\abs{s}^{1/2} + \eps^{1/3})\;.
\end{equation} 

\begin{lem}
There are solutions of the equations~\eqref{s-eqn}
satisfying
\begin{equation}
\label{s1s2} 
{\norm{p_1(s)}}, \norm{p_2(s)} = \cO \biggl( \frac{1}{\abs{a(s)}} \biggr)
 = \cO \biggl( \frac{1}{\abs{s}^{1/2} + \eps^{1/3}} \biggr)
 \qquad
 \text{for $s_0\leqs s \leqs c_1\eps^{2/3}$.}
\end{equation} 
\end{lem}

\begin{proof}
For $p_1(s)$, we first consider the equation $\eps \dot p_1 = a(s)p_1+c_1(s)$,
whose solution we know behaves as above, see Lemma~\ref{lem_app_scaling} 
(or~\cite[pp.~87--88]{BGbook}). Regular perturbation theory allows us to
extend the estimate to the full equation for $p_1$. In the case of $p_2$, we change the direction of time, and thus consider an
equation similar to the equation for $p_1$ on an interval
$[-c_1\eps^{2/3},-s_0]$. The above bound can be obtained, e.g., by scaling
space by $\eps^{1/3}$ and time by $\eps^{2/3}$ on
$[-c_1\eps^{2/3},\eps^{2/3}]$, and using integration by parts on the remaining
time interval. 
\end{proof}

\begin{cor}
For all $s_0\leqs r\leqs s\leqs c_1\eps^{2/3}$, the principal solution $U(s,r)$
satisfies 
\begin{align}
\label{eq:cor_U_xixi} 
\abs{U_{\xi\xi}(s,r)} &= \cO \biggl(  \e^{\alpha(s,r)/\eps} 
+ \frac{\eps}{(\abs{s}^{1/2}+\eps^{1/3})(\abs{r}^{1/2}+\eps^{1/3})}\biggr)\;, \\
\label{eq:cor_U_xieta} 
\norm{U_{\xi\eta}(s,r)} &= \cO \biggl( 
\frac{\e^{\alpha(s,r)/\eps}}{\abs{r}^{1/2}+\eps^{1/3}} 
+ \frac{1}{\abs{s}^{1/2}+\eps^{1/3}}\biggr)\;, \\
\label{eq:cor_U_etaxi} 
\norm{U_{\eta\xi}(s,r)} &= \cO \biggl(  \eps \biggl[
\frac{\e^{\alpha(s,r)/\eps}}{\abs{s}^{1/2}+\eps^{1/3}} 
+ \frac{1}{\abs{r}^{1/2}+\eps^{1/3}}\biggr]\biggr)\;, \\[6pt]
\norm{U_{\eta\eta}(s,r)} &= \Order{1}\;.
\label{eq:cor_U_etaeta} 
\end{align}
\end{cor}

\begin{proof}
We can apply Lemma~\ref{lem_U}, since $\eps \abs{p_1(s)p_2(s)}=\Order{\eps^{1/3}}$. 
Recall that the matrix $V$ occurring in~\eqref{U_size} is the principal
solution of $ \dot{\eta} = D_2(s)\eta$, where 
\begin{equation}
D_2(s) = B(s) + c_2(s)p_1(s) = \Order{\abs{a(s)}^{-1}}\;.
\end{equation} 
It follows that 
\begin{equation}
 \frac{\6}{\6s} \norm{\eta_s}^2 
 = 2(\eta_1\dot{\eta}_1 + \eta_2\dot{\eta}_2) 
 \leqs \frac{M}{\abs{a(s)}}\norm{\eta_s}^2 
\end{equation}
for some constant $M>0$, so that Gronwall's Lemma implies
\begin{equation}
 \norm{\eta_s}^2 \leqs \norm{\eta_{s_0}}^2 
 \exp\biggl\{ \int_{s_0}^s \frac{M}{\abs{a(u)}}\6u\biggr\}\;.
\end{equation} 
A direct computation using~\eqref{rfa01} shows that {for $s\leqs c_{1}\eps^{2/3}$} the integral has order $1$,
and thus the principal solution $V$ has order $1$ as well. Then the result
follows from Lemma~\ref{lem_U}. 
\end{proof}

\begin{prop}
\label{prop_tau_fold}
There exist constants $\kappa, h_0>0$ such that whenever 
$h\leqs h_0\eps^{1/3}$, $h_1\leqs h_0$, $h^2\leqs h_0h_1$ and $h_1^2\leqs
h_0h\eps^{1/3}$, 
\begin{align}
\nonumber
&\P \bigl\{ \tau_\xi \wedge \tau_\eta < c_1\eps^{2/3} \bigr\} \\[4pt]
\label{pprop_fold} 
& \qquad {}\leqs
 \biggintpartplus{\frac1\eps}
 \left(
 \e^{-\kappa h^2/(\sigma^2\eps^{-1/3})}
 + \e^{-\kappa h^2/((\sigma')^2\eps^{-2/3})}
 + \e^{-\kappa h_1^2/(\sigma')^2}
 + \e^{-\kappa h_1^2/(\eps\abs{\log\eps}\sigma^2)}\right)\;.
\end{align} 
\end{prop}

\begin{proof}
Estimate~\eqref{eq:cor_U_xixi} and Lemma~\ref{lem_app_scaling} 
imply 
\begin{equation}
\label{U2xx} 
 \frac1\eps \int_{s_0}^s U_{\xi\xi}(s,r)^2 \6r 
 =\cO \biggl(   \frac{1}{\abs{s}^{1/2} + \eps^{1/3}}\biggr) \leqs\cO\left(\eps^{-1/3}\right)
\end{equation} 
for $s_0+\Order{1}\leqs s\leqs c_1\eps^{2/3}$. Indeed, the term
$\e^{\alpha(s,r)/\eps}$ yields a contribution of this order, while the second term in~\eqref{eq:cor_U_xixi} gives a contribution of order 
$\eps\abs{\log\eps}/(\abs{s}^{1/2} + \eps^{1/3})$, which is smaller.
Next, we estimate 
\begin{equation}
 \int_{s_0}^s \norm{U_{\xi\eta}(s,r)}^2 \6r 
 =\cO \biggl(   \frac{1}{\abs{s} + \eps^{2/3}}\biggr) \leqs \cO\left(\eps^{-2/3}\right)\;,
\end{equation} 
where the main contribution now comes from the second term in~\eqref{eq:cor_U_xieta}. We also obtain 
\begin{equation}
 \frac1\eps \int_{s_0}^s \norm{U_{\eta\xi}(s,r)}^2 \6r 
 =\cO \bigl(   \eps\abs{\log\eps}\bigr)\;,
\end{equation} 
where the main contribution comes from the second term in~\eqref{eq:cor_U_etaxi}. Finally 
\begin{equation}
 \int_{s_0}^s \norm{U_{\eta\eta}(s,r)}^2 \6r = \Order{1}\;.
\end{equation} 
Similarly, we obtain the estimates 
\begin{align}
\nonumber 
\frac1\eps \int_{s_0}^s \abs{U_{\xi\xi}(s,r)} \6r 
 &=\cO \biggl(   \frac{1}{\abs{s}^{1/2} + \eps^{1/3}}\biggr)\;,\\
\nonumber 
\int_{s_0}^s \norm{U_{\xi\eta}(s,r)} \6r 
 &=\cO \biggl(   \frac{1}{\abs{s}^{1/2} + \eps^{1/3}}\biggr)\;,\\
\label{U1} 
\frac1\eps \int_{s_0}^s \norm{U_{\eta\xi}(s,r)} \6r 
 &=\Order{1}\;, \\
\int_{s_0}^s \norm{U_{\eta\eta}(s,r)} \6r 
 &= \Order{1}\;.
\nonumber  
\end{align}
We can now adapt the proof of Proposition~\ref{prop_tau} to the present
situation. Recall the definitions of $\tau_\xi, \tau_\eta$ from~\eqref{def_tau}. 
We denote again by $\xi^i_s$, $i=0,1,2,3$, the four terms on the right-hand side
of~\eqref{xi}. Let $T=c_1\eps^{2/3}$. The Bernstein-type estimate Lemma~\ref{lem_app_Bernstein}
and~\eqref{U2xx} yield 
\begin{equation}
 \P \biggl\{ \sup_{s_0\leqs r\leqs T} \abs{\xi^0_r} > h \biggr\}
 \leqs 
  \biggintpartplus{\frac1\eps}
  \e^{-h^2/(M\sigma^2\eps^{-1/3})}\;,
\end{equation} 
and similarly 
\begin{equation}
 \P \biggl\{ \sup_{s_0\leqs r\leqs T} \abs{\xi^1_r} > h \biggr\}
 \leqs 
  \biggintpartplus{\frac1\eps}
  \e^{-h^2/(M(\sigma')^2\eps^{-2/3})}\;.
\end{equation} 
Furthermore, using~\eqref{U1} we obtain $\abs{\xi^2_s}+\abs{\xi^3_s}\leqs
M'\eps^{-1/3}(h^2+h_1^2)$ for $s\leqs \tau_\xi \wedge \tau_\eta$. From this we
can deduce 
\begin{equation}
\label{rfa20} 
 \P \bigl\{ \tau_\xi < T\wedge\tau_\eta \bigr\}
 \leqs 
  \biggintpartplus{\frac1\eps}
  \biggl(
 \exp \biggl\{ - \frac{[h -
M'\eps^{-1/3}(h^2+h_1^2)]^2}{M\sigma^2\eps^{-1/3}}\biggr\}
 + \e^{-h^2/(M(\sigma')^2\eps^{-2/3})}
 \biggr)\;.
\end{equation} 
In a similar way, we get 
\begin{equation}
 \P \bigl\{ \tau_\eta < T\wedge\tau_\xi \bigr\}
 \leqs 
   \biggintpartplus{\frac1\eps}
  \biggl(
 \exp \biggl\{ - \frac{[h_1 -
M'(h^2+h_1^2)]^2}{M\sigma^2\eps\abs{\log\eps}}\biggr\}
 + \e^{-h_1^2/(M(\sigma')^2)} \biggr)\;.
\end{equation} 
This concludes the proof. 
\end{proof}

The condition $h_1^2\leqs h_0h\eps^{1/3}$ together with $h\leqs
h_{0}\eps^{1/3}$ 
imposes that we can take $h_1$ at most of
order $\eps^{1/3}$. For the typical spreadings, we obtain
\begin{itemize}
 \item in the fast direction: 
 \begin{equation}
  \frac{\sigma}{\eps^{1/6}} + \frac{\sigma'}{\eps^{1/3}}\;,
 \end{equation} 
 \item in the slow direction:
 \begin{equation}
  \sigma' + \sigma\sqrt{\eps\abs{\log\eps}}\;.
 \end{equation} 
\end{itemize}
For the bound~\eqref{pprop_fold} to be useful, we need the spreading in the
fast direction to be small compared to $\eps^{1/3}$, because of the condition
on $h$. This yields the conditions 
\begin{equation}
 \sigma \ll \eps^{1/2}\;, 
 \qquad
 \sigma' \ll \eps^{2/3}\;.
\end{equation} 
The term $\sigma/\eps^{1/6}$ of the $x$-spreading and the condition 
$\sigma \ll \eps^{1/2}$ are expected, because they already occur when there is
no noise acting on the slow variables (see~\cite[Section~3.3]{BGbook}). 
The term $\sigma'/\eps^{1/3}$ and the condition 
$\sigma' \ll \eps^{2/3}$ are due to the coupling with the slow variables. 

\begin{rem}
By using sharper estimates on the size of the linear terms $\xi^0_r$ and
$\xi^1_r$ (cf.~Remark~\ref{rem_Bernstein}), one can in fact show that the
typical spreading in the $x$-direction grows like 
\begin{equation}
 \frac{\sigma}{\abs{s}^{1/4}+\eps^{1/6}} +
\frac{\sigma'}{\abs{s}^{1/2}+\eps^{1/3}}\;.
\end{equation} 
\end{rem}

\subsection{Normal form}
\label{ssec:rfnorm} 

Before analysing the behaviour during the jump, we 
make a preliminary transformation to normal form near the fold. Recall
that $t=s/\eps$ denotes the fast timescale.

\begin{prop}
\label{prop:nform}
Near a regular fold on $L^-$ satisfying the assumptions (A1)--(A2) there exists
a smooth change of coordinates such that~\eqref{SDE} is locally given by
\begin{alignat}{3}
\nonumber
\6x_t &{}=
\bigl[y_t+x_t^2+\cO(z_t,{\norm{\transpose{(x_t,y_t)}}}^3\!,\eps,\sigma^2)\bigr]\6t 
&&{}+ \sigma \widehat F_1(x_t,y_t,z_t) \6W_t + \sigma'\sqrt\eps\,
\widehat F_2(x_t,y_t,z_t) \6W_t\;,\\    
\nonumber
\6y_t &{}= \eps
\hat g_1(x_t,y_t,z_t;\eps,\sigma')\6t 
&&{}+ \sigma'\sqrt\eps\, \widehat G_1(x_t,y_t,z_t) \6W_t\;,\\    
\6z_t &{}= \eps \hat g_2(x_t,y_t,z_t;\eps)\6t 
&&{}+ \sigma'\sqrt\eps\, \widehat G_2(x_t,y_t,z_t)\6W_t\;,
\label{eq:nform}
\end{alignat}
where $\hat g_1 = g_1 + \Order{(\sigma')^2}$ 
and
\begin{equation}
\hat g_1(0,0,0;0,0) = 1\;,
\qquad
\hat g_2(0,0,0;0) = 0\;.
\end{equation}
\end{prop}

\begin{proof}
The result is a stochastic analogue of the transformation result for
deterministic 
systems. We extend the proof presented by Szmolyan and Wechselberger in~\cite[pp.~73--74]{SzmolyanWechselbergerRelax} and~\cite[pp.~8--10]{WechselbergerThesis} to
the stochastic case. 

First, we may use a translation of coordinates so that the neighbourhood of
$L^-$ is 
chosen with center $(0,0,0)\in L^-$. From the normal switching condition Assumption~(A2)
we may assume without loss of generality that $g_1(0,0,0;0)\neq 0$; indeed, if $g_1(0,0,0;0)=0$ then 
$g_2(0,0,0;0)\neq 0$ and we may exchange the names of the two slow variables.
Next, define a coordinate change
\benn
z=:\bar{z}+\gamma y
\quad\text{with}\quad \gamma=\frac{g_2(0,0,0;0)}{g_1(0,0,0;0)}\not=0\;.
\eenn
This yields
\begin{align}
\nonumber
\6\bar{z}_t ={} \6z_t-\gamma\6y_t
={}&
\eps \bigl[g_2(x,y,\bar{z}+\gamma y;\eps)-\gamma  g_1(x,y,\bar{z}+\gamma
y;\eps)\bigr]\6t\\[3pt]
&{}+\sigma'\sqrt\eps \bigl[G_1(x,y,\bar{z}+\gamma
y;\eps)-\gamma G_2(x,y,\bar{z}+\gamma y;\eps)\bigr]\6W_t\;.
\end{align}
Introducing new maps $\bar{g}_2=g_2-\gamma g_1$ and $\overbar{G}_2=G_1-\gamma G_2$
and then dropping all
the overbars from the notation yields a stochastic fast--slow system of the form~\eqref{SDE} which now satisfies
\benn
g_1(0,0,0;0)\neq0\qquad \text{and}\qquad g_2(0,0,0;0)=0.
\eenn
The next step is to rectify the fold curve. By the implicit-function theorem
there exists a 
parametrization of $L^-$ by $(\xi(z),\eta(z),z)$ for $z\in \cI\subset \R$ where
$\cI$ is a suitable
interval. The transformation
\benn
(x,y,z)=(\bar{x}+\xi(z),\bar{y}+\eta(z),\bar{z})
\eenn
rectifies the fold curve in new coordinates $(\bar{x},\bar{y},\bar{z})$,
{in the sense} that $\bar f(0,0,\bar z)=0$.  It\^{o}'s formula shows
\begin{align}
\nonumber
\6\bar{x}_t &= \bigl[\bar{f}(\bar{x}_t,\bar{y}_t,\bar{z}_t)+
\cO(\eps(\sigma')^2)\bigr]\6t
+\bigl[\sigma
\overbar{F}(\bar{x}_t,\bar{y}_t,\bar{z}_t)-\sigma'\sqrt\eps(\partial_z\xi)\overbar{
G}_2(\bar{x}_t,\bar{y}_t,\bar{z}_t)\bigr]\6W_t\\
\6\bar{y}_t &= \bigl[\eps\bar{g}_1(\bar{x}_t,\bar{y}_t,\bar{z}_t)+
\cO(\eps(\sigma')^2)\bigr]\6t
+\sigma'\sqrt\eps\bigl[\overbar{G}_1(\bar{x}_t,\bar{y}_t,\bar{z}
_t)-(\partial_z\eta)\overbar{G}_2(\bar{x}_t,\bar{y}_t,\bar{z}_t)\bigr]\6W_t\;,
\end{align}
where $\bar{f}(\bar{x},\bar{y},\bar{z}) = a\bar{y} + b\bar{x}^2 +
c\bar{x}\bar{y} + d\bar{y}^2 + \Order{z,\norm{\transpose{(\bar{x},\bar{y})}}^3}$. 
By a scaling of $\bar{x}, \bar{y}$ and time, we can achieve that $a=b=1$
and $\bar{g}_1(0,0,0)=1$. 

The final step is a normal-form transformation 
$\hat x=\bar{x}-\frac12 c\bar{x}^2 - d\bar{x}\bar{y}$, which eliminates the
terms of order $\bar{x}\bar{y}$ and $\bar{y}^2$ in the drift term of
$\6\bar{x}_t$. Applying again It\^o's formula yields the result.
\end{proof}

\begin{rem}
It is possible to further simplify the drift term, in such a way that 
for $\eps=0$ and $\sigma=0=\sigma'$, $g_1(x,y,z)=g(z)+g_{11}(x,y,z)$ where
$g_{11}(0,0,z)=0$ and $g_2(0,0,z)=0$, see~\cite[pp.~9--10]{WechselbergerThesis}
and~\cite[p.~73]{SzmolyanWechselbergerRelax}.
However, this introduces a diffusion term of order $\sigma$ in $\6y_t$, which we
want to avoid. 
\end{rem}

\subsection{Neighbourhood and escape}
\label{ssec:rfn} 

We determine now the size of fluctuations during the \lq\lq jump phase\rq\rq\  of
sample paths starting on $\Sigma'_4$, until they hit the section $\Sigma_5$
which is located at a distance of order~$1$ in the $x$-direction from the fold. 
Before giving a rigorous estimate, we briefly recall some well-known deterministic 
asymptotics as they are going to motivate several choices in the analysis of the stochastic dynamics. 

The lowest-order approximation for the deterministic dynamics near the planar
fold is
\begin{equation}
 \label{rfn01}
 \eps \frac{\6x}{\6y} = y + x^2\;,
\end{equation} 
which is just the classical Ricatti equation; see~\cite[pp.~68--72]{MisRoz} or~\cite[p.~100]{KruSzm1}.
Setting $y=\eps^{2/3}\theta$ and $x=\eps^{1/3}\tilde{x}$ removes $\eps$ and
yields
\begin{equation}
 \label{rfn01a}
 \frac{\6\tilde{x}}{\6\theta} = \theta + \tilde{x}^2
\end{equation}
as the system of first approximation~\cite[p.~175]{ArnoldEncy} which also appears
as
the key asymptotic problem in the blow-up analysis~\cite[p.~293]{KruSzm3} of the 
non-degenerate fold. It is known~\cite[pp.~68--72]{MisRoz} that there exists an orbit
$\theta(\tilde{x})$ 
of~\eqref{rfn01a} with 
\begin{alignat}{3}
\theta(\tilde{x})&=
-\tilde{x}^2-\frac{1}{2\tilde{x}}+\cO\left(\frac{1}{\tilde{x}^4}\right)
&&\qquad \text{as $x\ra -\I$\;,}\\
\theta(\tilde{x})&=
\theta^*-\frac{1}{\tilde{x}}+\cO\left(\frac{1}{\tilde{x}^3}\right)
&&\qquad \text{as $x\ra \I$\;,} 
\end{alignat}
which is the extension of the attracting slow manifold through the fold region;
the constant
$\theta^*$ is the horizontal asymptote which can be expressed as the zero of
suitable 
Bessel functions. However, if we look at the variational equation of~\eqref{rfn01a} around 
$\theta(\tilde{x})$ to leading order it follows that 
\begin{equation}
\dtot{\xi}{\theta}=2\,\frac{1}{\theta^*-\theta}\xi,\qquad \text{as $\tilde{x}\ra
\I$ \ (or $\theta\ra \theta^*$).}
\end{equation}
The solution is given by 
\begin{equation}
 \label{rfn05}
 \xi(\theta)\cong \frac{1}{(\theta^*-\theta)^2}\xi(\theta_0)\;.
\end{equation} 
This growth of the linearization in the fast direction turns out to be too fast
to apply directly the same method to control stochastic sample paths as in the
previous cases. 
However, we do not need such a precise control of fluctuations in the fast
direction. 
It is sufficient to show that sample paths are likely to stay in a tube around
the
deterministic solution, with some specific extension in the slow directions
$y$ and $z$. To do so, we will compare the random process with different
deterministic solutions on successive time intervals $[\theta_n,\theta_{n+1}]$
during which fluctuations in the fast direction remain bounded
(cf.~\figref{fig_regular_fold}). The expression~\eqref{rfn05} shows that a
possible
choice are geometrically accumulating $\theta_n$ of the form $\theta_n =
\theta^* - 2^{-n}$. During the interval $[\theta_n,\theta_{n+1}]$, the
deterministic solution $x(\theta)$ moves by a distance of order
$\eps^{1/3}(2^{n+1}-2^n)=\eps^{1/3}2^n$. For $x(\theta)$ to reach order~$1$, we
need to choose~$n$ of order $\abs{\log\eps}$.

\begin{figure}
\centerline{\includegraphics*[clip=true,width=140mm]{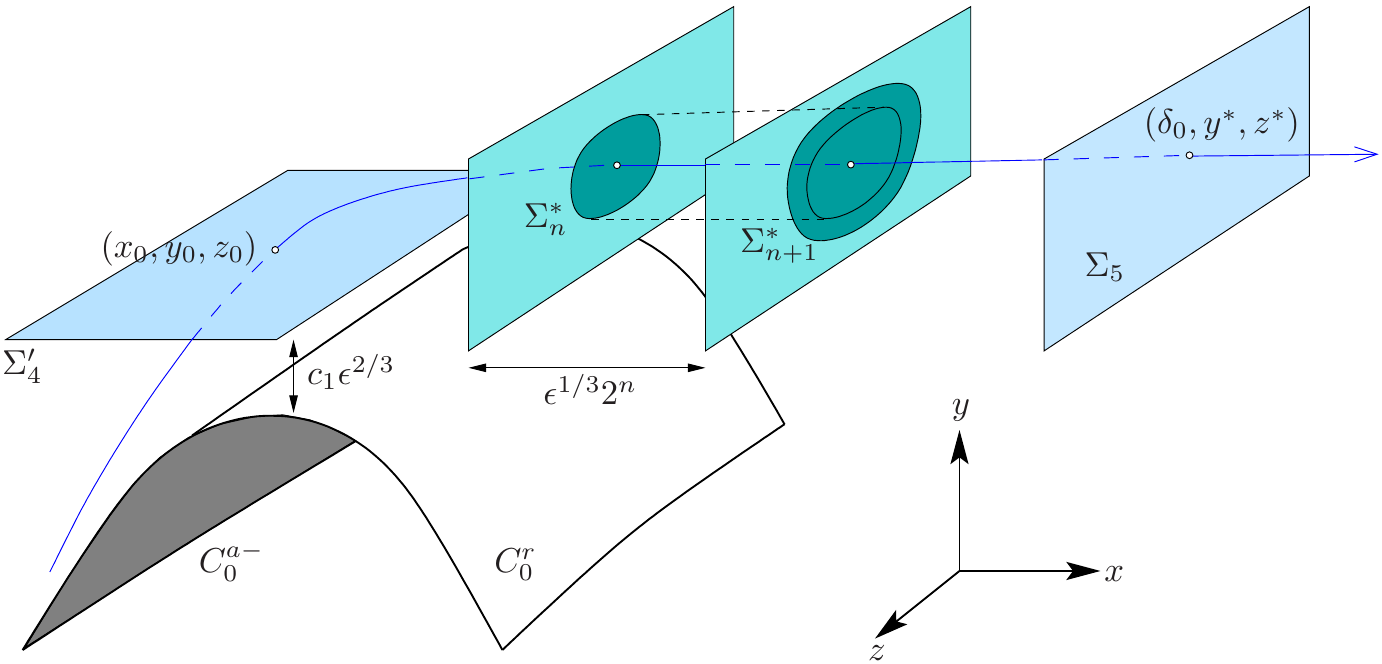}}
 \vspace{2mm}
\caption[]{Geometry of sections near the regular fold.
}
\label{fig_regular_fold}
\end{figure}

To make the last idea rigorous, we write the system~\eqref{eq:nform} on the
timescale $\theta=\eps^{1/3}t=\eps^{-2/3}s$ as 
\begin{alignat}{3}
\nonumber
\6x_\theta &{}= \frac{1}{\eps^{1/3}}
\hat f(x_\theta,y_\theta,z_\theta)\6\theta 
&&{}+ \frac{\sigma}{\eps^{1/6}} \widehat F_1(x_\theta,y_\theta,z_\theta)
\6W_\theta +
\sigma'\eps^{1/3}
\widehat F_2(x_\theta,y_\theta,z_\theta) \6W_\theta\;,\\    
\nonumber
\6y_\theta &{}= \eps^{2/3}
\hat g_1(x_\theta,y_\theta,z_\theta;\eps,\sigma')\6\theta 
&&{}+ \sigma'\eps^{1/3} \widehat G_1(x_\theta,y_\theta,z_\theta)
\6W_\theta\;,\\   
\6z_\theta &{}= \eps^{2/3} \hat g_2(x_\theta,y_\theta,z_\theta;\eps)\6\theta 
&&{}+ \sigma'\eps^{1/3} \widehat G_2(x_\theta,y_\theta,z_\theta)\6W_\theta\;,
\label{eq:nform_theta}
\end{alignat}
where $\hat f(x,y,z) = y+x^2+\cO(z,\norm{\transpose{(x,y)}}^3\!,\eps,\sigma^2)$. 
Given $\delta>0$ of order $1$, there exists a $\delta_0$ of order $1$ such
that by restricting the analysis to a cube of size
$\delta_0$, we may assume that $\abs{\hat{g}_1-1} <
\delta$ and $\abs{\hat{g}_2} < \delta$. 

For convenience, we set $\Sigma^*_0 = \{(x_0,y_0,z_0)\} \subset \Sigma'_4$, where
we recall that the initial condition satisfies $x_0\asymp-\eps^{1/3}$,
$y_0=c_1\eps^{2/3}$ for some $c_{1}>0$, and that we may assume
$\abs{z_0}\ll\eps^{2/3}$. 
For $n\geqs1$ and $\eps>0$ such that $\eps^{1/3}2^n<\delta_0$, we introduce
sets 
\begin{equation}
 \label{rfn06}
 \Sigma^*_n = \left\{ (x,y,z) \colon x=\eps^{1/3}2^n, 
 (y,z)\in D_n \right\}\;,
\end{equation}
see~\figref{fig_regular_fold}. The sets $D_n$ are defined inductively as
follows: 
\begin{equation}
 D_1=(y_1-c_2\eps^{2/3},y_1+c_2\eps^{2/3})
 \times(-c_2\eps^{2/3},c_2\eps^{2/3})\;,
\end{equation}
where $y_1$ is such that $(2\eps^{1/3},y_1,z_1)$ belongs to the deterministic
orbit starting in $(x_0,y_0,z_0)$, and $c_2<c_1$ is a sufficiently small
constant. Given $D_{n}$, the next set $D_{n+1}$ is chosen as the
$c_2\eps^{2/3}2^{-n/2}$-neighbourhood of 
the image of $D_n$ under the deterministic Poincar\'e map from $\Sigma^*_n$ to
$\Sigma^*_{n+1}$. It is not difficult to show that for sufficiently small
$\delta$ and $c_2$, the time needed for the deterministic flow to go from
$\Sigma^*_n$ to $\Sigma^*_{n+1}$ is of order $\theta=\Order{2^{-n}}$. During
this time, $y$ and
$z$ vary by $\Order{\eps^{2/3}2^{-n}}$ at most, and thus 
\begin{equation}
 \diam(D_{n+1}) \leqs \diam(D_n) + \Order{\eps^{2/3}2^{-n/2}}\;.
\end{equation} 
The geometric decay in $2^{-n/2}$ shows that the diameter of the $D_n$ has a
uniform bound of order~$\eps^{2/3}$. In fact, by taking $\delta$ small, we can
make the extension of $D_n$ in the $z$-direction small. 

We return to the stochastic system~\eqref{eq:nform_theta}. Fix~$n$.
For an initial condition $(x_n,y_n,z_n)\in\Sigma^*_n$, we 
denote by $(x^{\det}_\theta,y^{\det}_\theta,z^{\det}_\theta)$ and 
$(x_\theta, y_\theta, z_\theta)$ the deterministic and stochastic
solutions starting in $(x_n,y_n,z_n)$. We write $ \P^{(x_n,y_n,z_n)}$ whenever we wish to stress the initial condition.
Consider the stopping times
\begin{align}
\nonumber
\tau_{n+1} &= \inf \{ \theta \colon (x_\theta, y_\theta, z_\theta)\in
\Sigma^*_{n+1} \}
\;, \\
\nonumber
\tau^{\det}_{n+1}={}\tau^{\det} &= \inf \{ \theta \colon (x^{\det}_\theta, y^{\det}_\theta,
z^{\det}_\theta)\in
\Sigma^*_{n+1} \}
\;, \\
\nonumber
\tau_\xi = 
\tau^{(n)}_\xi(h) &= \inf \{ \theta \colon \abs{x_\theta-x^{\det}_\theta} > h
2^{-n/2}
\} \;, \\
\tau_\eta = 
\tau^{(n)}_\eta(h_1) &= \inf \{ \theta \colon \|(y_\theta,z_\theta) -
(y^{\det}_\theta,z^{\det}_\theta)\| > h_1 2^{-n/2} \}
\;.  
\end{align}
We first establish that sample paths are likely to go from $\Sigma^*_n$ to
$\Sigma^*_{n+1}$ in a time of order $\theta=\Order{2^{-n}}$, as in the
deterministic case. 

\begin{lem}
\label{lem_rfn1} 
There exist $h_0, c,c_{2}, \kappa>0$, not depending on~$n$, such that 
for all initial conditions $(x_n,y_n,z_n)\in\Sigma^*_n$ and
$h\leqs h_0\eps^{1/3}$, $h_1\leqs c_2\eps^{2/3}$,
\begin{equation}
 \label{rfn07}
 \P^{(x_n,y_n,z_n)}\bigl\{
\tau_{n+1}\wedge\tau^{(n)}_\xi(h)\wedge\tau^{(n)}_\eta(h_1) >
c2^{-n}
\bigr\} 
 \leqs \exp\biggl\{ -\kappa\frac{2^{3n}\eps}{\sigma^2+(\sigma')^2\eps}
\biggr\}\;.
\end{equation} 
\end{lem}

\begin{proof}
First note that $\tau_{n+1}>c 2^{-n}$ implies that either $x$ does not reach the
level $\eps^{1/3}2^{n+1}$ before time $c2^{-n}$ or that $x$ does reach
$\eps^{1/3}2^{n+1}$ at a stopping time $\tau^{x}_{n+1}\leqs c2^{-n}$ while 
$(y_{\tau^{x}_{n+1}},z_{\tau^{x}_{n+1}})\not\in D_{n+1}$. 

Let us estimate the probability that $\tau^{x}_{n+1}> c2^{-n}$. Note that  $h_0$
and $c_2$ can be chosen sufficiently small to guarantee that $\hat f\geqs
\eps^{2/3}2^{2n-2}$ for all times $\theta \leqs \tau_\xi \wedge \tau_\eta$. From
the representation
\begin{equation}
 x_{c2^{-n}} = \eps^{1/3} 2^n + \frac{1}{\eps^{1/3}} \int_0^{c2^{-n}} \hat f
\6\theta
 + \frac{\sigma}{\eps^{1/6}} \int_0^{c2^{-n}} \widehat F_1 \6W_\theta 
 + \sigma'\eps^{1/3} \int_0^{c2^{-n}} \widehat F_2 \6W_\theta
\end{equation} 
we find that
\begin{align}
\nonumber
 &\P\{ x_{c2^{-n}} < \eps^{1/3} 2^{n+1}, \tau_\xi \wedge \tau_\eta > c2^{-n} \}
\\
\nonumber
 &\quad\leqs 
 \P \biggl\{ \frac{\sigma}{\eps^{1/6}} \int_0^{c2^{-n}} \widehat F_1 \6W_\theta 
 + \sigma'\eps^{1/3} \int_0^{c2^{-n}} \widehat F_2 \6W_\theta 
 < -\left(\frac14c-1\right) \eps^{1/3} 2^n \biggr\} \\
 &\quad\leqs 
 \exp\biggl\{ - \frac{(\frac14c-1)^2 \eps^{2/3} 2^{2n}}
 {Mc2^{-n} [\sigma^2\eps^{-1/3}+(\sigma')^2\eps^{2/3}]}\biggr\}
\end{align}
for some constant $M>0$, provided $c>4$. In the last line, we used the fact that if $M_t=\int_{0}^{t}F(s,\cdot)\6W_{t}$ with integrand $F(s,\omega)$ bounded in absolute value by a constant~$K$, then Novikov's condition~\cite[{pp.}~198--199]{KaratzasShreve} is satisfied and thus
\begin{equation}
 Z_t = \exp\biggl\{ \gamma M_t - \frac{\gamma^2}{2}\int_0^t
F(s,\omega)^2\6s\biggr\}
\end{equation} 
is a martingale for any $\gamma>0$. It follows that for $h>0$,
\begin{equation}
 \P \bigl\{ M_t > h \bigr\}
 \leqs \P \bigl\{ Z_t > \e^{\gamma h - \gamma^2 K^{2}t/2}\bigr\}
 \leqs \e^{-\gamma h + \gamma^2 K^{2}t/2} \E\{Z_t\}
 = \e^{-\gamma h + \gamma^2 K^{2}t/2}\;,
\end{equation} 
where we  used Markov's inequality and the fact that a martingale has constant expectation.

Thus we obtained a bound on the probability of $x$ not reaching
$\eps^{1/3}2^{n+1}$ despite of $\xi$ and $\eta$ remaining small. It remains to
consider the case $(y_{\tau^{x}_{n+1}},z_{\tau^{x}_{n+1}})\not\in D_{n+1}$ for
$\tau^{x}_{n+1}\leqs c2^{-n}$. 

By~\eqref{eq:nform_theta}, the lower bound on $\hat f$ and the fact that
$x^{\det}_{\tau^{\det}}=\eps^{1/3}2^{n+1}=x_{\tau^{x}_{n+1}}$, we see that on
the set
$\Omega'=\setsuch{\omega\in\Omega}{\tau^{\det}\vee\tau^{x}_{n+1}(\omega)
\leqs
\tau_\xi(\omega) \wedge \tau_\eta(\omega)}$,
\begin{align}
\nonumber
 \abs{y^{\det}_{\tau^{x}_{n+1}} - y^{\det}_{\tau^{\det}}} 
 &{}\leqs \const \eps^{2/3} \abs{\tau^{x}_{n+1} - \tau^{\det}}
 \leqs \frac{\const \eps^{2/3}}{\eps^{1/3}2^{2n-2}} 
\abs{x^{\det}_{\tau^{x}_{n+1}} - x^{\det}_{\tau^{\det}}} \\
&{} \leqs \Order{\eps^{1/3} h 2^{-5n/2}} \leqs \Order{h_{0}\eps^{2/3}2^{-5n/2}}
\end{align} 
and $\abs{y_{\tau^{x}_{n+1}} - y^{\det}_{\tau^{x}_{n+1}}} \leqs h_1 2^{-n/2} 
=c_2\eps^{2/3}2^{-n/2}$.
Similar estimates hold for the $z$-coordinate. 
Since $(y^{\det}_{\tau^{\det}},z^{\det}_{\tau^{\det}})$ belongs to the image of
$D_n$ under the deterministic Poincar\'e map, we conclude that 
$(y_{\tau^{x}_{n+1}},z_{\tau^{x}_{n+1}})$ belongs to an
$\eps^{2/3}2^{-n/2}$-neighbourhood of
this image. Thus $\tau^{x}_{n+1}=\tau_{n+1}$ on $\Omega'$. Choosing $c$
large enough to guarantee $\tau^{\det} \leqs c2^{-n}$ concludes the proof.
\end{proof}

The next result gives a bound on fluctuations of sample paths, up to time
$c2^{-n}$.  

\begin{lem}
\label{lem_rfn2} 
There exist $M, h_0>0$ such that for all initial conditions
$(x_n,y_n,z_n)\in\Sigma^*_n$ and all $h, h_1>0$ satisfying
$h\leqs h_0\eps^{1/3}2^{5n/2}$, $h_1\leqs h_0\eps^{-1/3}2^{7n/2}$, 
$h^2\leqs h_0\eps^{-1/3}2^{7n/2}h_1$ and $h_1^2\leqs h_0\eps^{1/3}2^{5n/2}h$, 
\begin{align}
\nonumber
\P^{(x_n,y_n,z_n)}&\bigl\{ \tau^{(n)}_\xi(h)\wedge\tau^{(n)}_\eta(h_1)  <
c2^{-n}
\bigr\} \\
\nonumber
 \leqs{}&
 \exp\biggl\{ -\frac{h^2}{M(\sigma^2+(\sigma')^2\eps)\eps^{-1/3}} \biggr\}
 + \exp\biggl\{ -\frac{h^2}{M(\sigma')^22^{-2n}} \biggr\} \\
 &{}+ \exp\biggl\{ -\frac{h_1^2}{M(\sigma^2+(\sigma')^2\eps)\eps2^{-2n}}
\biggr\}
 + \exp\biggl\{ -\frac{h_1^2}{M(\sigma')^2\eps^{2/3}} \biggr\}
 \;.
 \label{rfn08}
\end{align} 
\end{lem}

\begin{proof}
The proof is similar to the proof of Proposition~\ref{prop_tau_fold}. 
First note that the linearization $a(\theta)=\partial_x\hat
f(x^{\det}_\theta,y^{\det}_\theta,z^{\det}_\theta)$ has order $x^{\det}_\theta$, 
satisfying $x^{\det}_\theta \leqs \const \eps^{1/3}2^{k+1}$ for $\theta\leqs
\tau^{\det}_{k+1}$ for any $k$. Since $\tau^{\det}_{k+1} \asymp 2^{-k}$ for all $k$, $x^{\det}_\theta$ remains
of order $\eps^{1/3}2^n$ for all $\theta \leqs c2^{-n}$. Thus $a(\theta)=\Order{\eps^{1/3}2^{n}}$ for all $\theta\leqs c2^{-n}$, which implies
\begin{equation}
 \alpha(\theta,\phi) = 
 \int_\phi^\theta a(u)\6u 
  \leqs \Order{\tau^{\det} \eps^{1/3}2^{n}}
 \leqs \Order{2^{-n} \eps^{1/3}2^{n}}
 = \Order{\eps^{1/3}}\;
\end{equation} 
for all $0\leqs\phi\leqs\theta\leqs c2^{-n}$.
Using this, one shows that the analogue of~\eqref{s-eqn} admits solutions
$p_1(\theta), p_2(\theta) = \Order{\eps^{-1/3}2^{-n}}$, so that by
Lemma~\ref{lem_U}, 
\begin{alignat}{3}
& U_{\xi\xi}(\theta,\phi) = \Order{1}\;, \:\nonumber
&& \qquad U_{\xi\eta}(\theta,\phi) = \Order{\eps^{-1/3}2^{-n}}\;, \:\\
& U_{\eta\xi}(\theta,\phi) = \Order{\eps^{2/3}2^{-n}}\;, \:
&& \qquad U_{\eta\eta}(\theta,\phi) = \Order{1}\;,
\label{sec4:behaviourU}
\end{alignat} 
for $0\leqs\phi\leqs\theta\leqs c2^{-n}$.
It follows from computations similar to those yielding~\eqref{rfa20} that 
\begin{equation}
 \P \bigl\{ \tau_\xi < c2^{-n}\wedge\tau_\eta \bigr\}
 \leqs 
 \exp \biggl\{ - 
 \frac{[h-M\eps^{-1/3}2^{-5n/2}(h^2+h_1^2)]^2}
{M(\sigma^2+(\sigma')^2\eps)\eps^{-1/3}}
\biggr\}
 + \e^{-h^2/(M(\sigma')^22^{-2n})}\;.
\end{equation} 
Since we are working on rather short time intervals, we can approximate the stochastic integral by the same Gaussian martingale on the whole time interval. Thus there is no subexponential prefactor of the type $\intpartplus{\cdot}$.

In a similar way, using the fact that $\eps^{1/3}\leqs \delta_{0}{2^{-n}}$, we
get 
\begin{equation}
 \P \bigl\{ \tau_\eta < c2^{-n}\wedge\tau_\xi \bigr\}
 \leqs 
 \exp \biggl\{ - 
 \frac{[h_1-M\eps^{1/3}2^{-7n/2}(h^2+h_1^2)]^2}
{M(\sigma^2+(\sigma')^2\eps)\eps2^{-2n}}
\biggr\}
 + \e^{-h_1^2/(M(\sigma')^2\eps^{2/3})}\;.
\end{equation} 
The conditions on $h, h_1$ guarantee that the terms
in $(h^2+h_1^2)$ are negligible. 
\end{proof}

\begin{figure}
\centerline{\includegraphics*[clip=true,width=80mm]{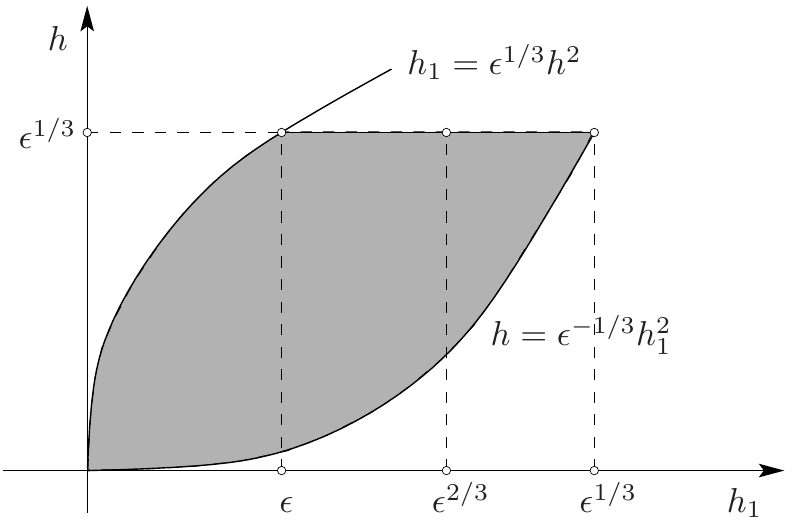}}
 \vspace{2mm}
\caption[]{The shaded area is the set of $(h_1,h)$ satisfying the 
conditions given in Lemma~\ref{lem_rfn2} (if $h_0=1$ and $n=0$). 
Lemma~\ref{lem_rfn1} requires in addition that $h_1\leqs\eps^{2/3}$.
}
\label{fig_hh1}
\end{figure}

The conditions on $h$ and $h_1$ are illustrated in~\figref{fig_hh1}. 

Putting the preceeding two results together, we obtain the following estimate on the
spreading of sample paths when they hit $\Sigma_5$. 

\begin{prop}
\label{prop_fold_escape} 
Denote by $(y^*,z^*)$ the point where the deterministic solution starting in
$(x_0,y_0,z_0)\in\Sigma'_4$ first hits $\Sigma_5 = \{
x=\delta_0 \}$. Then there exist $C, \kappa, h_0>0$ such that 
for any $h_1$ satisfying $h_1\leqs h_0\eps^{2/3}$, 
the stochastic sample path starting in $(x_0,y_0,z_0)$ first hits 
$\Sigma_5$ at time $\tau=\tau_{\Sigma_5}$ at
a point $(\delta_0, y_{\tau}, z_{\tau})$ such that 
\begin{equation}
 \label{rfn10}
 \P \left\{ \norm{(y_{\tau}, z_{\tau}) - (y^*,z^*)} > h_1 \right\}
 \leqs C\abs{\log\eps} \biggl[
 \exp \biggl\{ -\frac{\kappa h_1^2}{\sigma^2\eps +
(\sigma')^2\eps^{1/3}}\biggr\}
 + \exp \biggl\{ -\frac{\kappa\eps}{\sigma^2 + (\sigma')^2\eps}\biggr\}
 \biggr]\;.
\end{equation} 
\end{prop}

\begin{proof}
Let $N$ be the largest integer such that $\eps^{1/3}2^N\leqs\delta_0$, and  
$\tau^{(n)}=\tau^{(n)}_\xi(h)\wedge\tau^{(n)}_\eta(h_1/2)$ for $n=1,\dots,N$, 
where $h=h_0(\eps^{1/3}\wedge\eps^{-1/6}h_1^2)$ is taken as large as possible,
cf.~\figref{fig_hh1}. 

If $\tau_{h_1}$ denotes the first time the stochastic sample path leaves
a tube of size $h_1$ around the deterministic solution, the left-hand side
of~\eqref{rfn10} can be bounded above by $\P\{ \tau_{h_1} < \tau_{\Sigma_5}\}$. 
Since 
\begin{equation}
 \bigcap_{n=1}^N \bigl\{ \tau_{n+1} \leqs c2^{-n}\wedge\tau^{(n)} \bigr\} 
 \subset
 \biggl\{ \tau_{\Sigma_5} \leqs \sum_{n=1}^N c2^{-n} \wedge \tau_{h_1}
\biggr\}\;,
\end{equation} 
we have the bound
\begin{equation}
 \P\bigl\{ \tau_{h_1} < \tau_{\Sigma_5} \bigr\}
 \leqs \P\biggl\{ \tau_{h_1}\wedge\sum_{n=1}^N c2^{-n} < \tau_{\Sigma_5}
\biggr\}
\leqs \sum_{n=1}^N \P\bigl\{ \tau_{n+1} > c2^{-n}\wedge\tau^{(n)} \bigr\}\;.
\end{equation} 
Each term of the sum is bounded by 
$\P\{\tau_{n+1}\wedge\tau^{(n)}>c2^{-n}\}
+\P\{\tau^{(n)}<c2^{-n}\wedge\tau_{n+1}\}$, so that the result follows from the
last two lemmas. By distinguishing the cases $h_1\geqs\eps$ and $h_1\leqs\eps$,
one checks that our choice of $h$ implies that the terms in $h^2$ are
negligible, compared to at least one of the two summands on the right-hand side
of~\eqref{rfn10}.
\end{proof}

This result implies that the spreading in the $y$- and $z$-directions on 
$\Sigma_5$, for a given initial condition on $\Sigma'_4$, is of order 
\begin{equation}
 \label{rfn11}
 \sigma\sqrt{\eps} + \sigma'\eps^{1/6}\;.
\end{equation} 


\section{The folded node}
\label{sec:local_deviate}

In this section we analyze the transition $\Sigma_1\to\Sigma_2$ of sample paths
in a neighbourhood of the folded-node point $p^*$. For convenience, we translate
the origin of the coordinate system to $p^*$. We will decompose the transition
into three parts, by introducing further sections $\Sigma_1' =
\{x=\delta\sqrt{\eps}\,\}$ and $\Sigma_1'' = \{x=-\delta\sqrt{\eps}\,\}$,
where $\delta$ is a small constant of order~$1$. The transitions
$\Sigma_1\to\Sigma_1'$, $\Sigma_1'\to\Sigma_1''$, and $\Sigma_1''\to\Sigma_2$
are analyzed, respectively, in Subsection~\ref{ssec:fn-approach}, in
Subsections~\ref{ssec:fn-nbh-det} and~\ref{ssec:fn-nbh-stoch}, and in
Subsection~\ref{ssec:fn-escape2}. 

\subsection{Normal form}
\label{ssec:fn-nf}

We start by making a preliminary transformation to normal form near the folded
node point~$p^*$. Recall once again that $t=s/\eps$ denotes the fast timescale. 

\begin{prop}
\label{prop:fn-nform}
Near a folded-node point $p^*\in L^+$ satisfying the assumptions (A1) and (A3), 
there exist a smooth change of coordinates and a random time change 
such that~\eqref{SDE} is locally given by
\begin{alignat}{3}
\nonumber
\6x_t &{}=
\hat f(x_{t},y_{t},z_{t};\eps,\sigma,\sigma') \6t 
&&{}+ \bigl[\sigma \widehat F_1(x_t,y_t,z_t) + \sigma'\sqrt\eps\,
\widehat F_2(x_t,y_t,z_t) \bigr]\6W_t\;,\\    
\nonumber
\6y_t &{}= \eps
\hat g_1(x_t,y_t,z_t;\eps,\sigma')\6t 
&&{}+ \sigma'\sqrt\eps\, \widehat G_1(x_t,y_t,z_t) \6W_t\;,\\    
\6z_t &{}= \frac12 \eps \mu\6t 
&&{}+ \sigma'\sqrt\eps\, \widehat G_2(x_t,y_t,z_t)\6W_t\;,
\label{eq:fn-nform}
\end{alignat}
where $\mu\in(0,1)$ is the ratio of weak and strong eigenvalues at the folded
node 
(see also~\cite[p.~4793]{BGK12} or~\cite[p.~48]{BronsKrupaWechselberger}), and 
\begin{align}
\nonumber
\hat f(x,y,z;\eps,\sigma,\sigma') &= 
y-x^2+\cO(\norm{\transpose{(x,y,z)}}^3\!,\eps\norm{\transpose{(x,y,z)}},
\sigma^2,(\sigma')^2\eps)\;, \\
\label{eq:nform2_vec}
\hat g_1(x,y,z;\eps,\sigma') &= 
-(\mu+1)x - z + \cO(y,(x+z)^2,\eps,(\sigma')^2)\;,
\end{align}
while the diffusion matrices $ \widehat F_1, \widehat F_2, \widehat G_1, \widehat G_2$  all remain of order~$1$.
\end{prop}

\begin{proof}
The result is again a stochastic analogue of the transformation result for
deterministic systems, see~\cite[pp.~8--10]{WechselbergerThesis}, as well
as~\cite{WechselbergerFN}. 

We start by translating the origin of the coordinate system to the folded-node
point~$p^*$. Note that the failure of the normal-switching condition~(A2) implies that the vectors $(\frac{\partial f}{\partial y}
,\frac{\partial f}{\partial z})(0)$ and $g(0)$ are orthogonal. We may thus
rotate coordinates in such a way that $g_1(0)=0$ and $\frac{\partial f}{\partial z}(0)=0$. This
rotation does not change the order of magnitude of the diffusion coefficients
$\sigma'\sqrt{\eps}\,G_1$ and $\sigma'\sqrt{\eps}\,G_2$. 

Calculating the linearization of the desingularized slow flow~\eqref{eq:slow_sub2} and using Assumption~(A1), we see that 
 $g_2(0)\neq0$, since otherwise $p^{*}=0$ would not be a node for~\eqref{eq:slow_sub2} as required by Assumption~(A3).
We can thus carry out locally a random time change given by 
\begin{equation}
 \6\tilde t = \frac{g_2(x_t,y_t,z_t)}{g_2(0,0,0)} \6t\;. 
\end{equation} 
Lemma~\ref{lem_random_time_change} in Appendix~\ref{appendix} shows that this
time change yields a system in which all drift coefficients have been
multiplied by $g_2(0,0,0)/g_2(x,y,z)$, and all diffusion coefficients have been
multiplied by $[g_2(0,0,0)/g_2(x,y,z)]^{1/2}$. We may thus assume that
$g_2(x,y,z)$ is constant and equal to $g_2(0,0,0)$ in~\eqref{SDE}. 

The remainder of the proof is similar to the proof of
Proposition~\ref{prop:nform}. A transformation $x=\bar x+\xi(z)$, $y=\bar y +
\eta(z)$ rectifies the fold curve, i.e.~$f(\xi(z),\eta(z),z)=0$ and 
$\frac{\partial f}{\partial x}(\xi(z),\eta(z),z)=0$ in a neighbourhood of $z=0$,
and thus 
\begin{equation}
 \bar f(\bar x,\bar y,z) = a\bar y + b\bar x^2 + c\bar x\bar y + d\bar y^2 +
e\bar yz + k\eps 
 + \cO(\norm{(\bar x,\bar y,z)}^3,\eps z,
(\sigma')^2\eps)\;.
\end{equation} 
The standard form
of $f$ and $g_1$ can then be achieved by combining a translation of $x$ by
$\Order{\eps}$, a scaling of space and a near-identity transformation $\hat x =
\bar x - \frac12 c \bar x^2 - d \bar x\bar y - e \bar xz$
(cf.~\cite[pp.~9--10]{WechselbergerThesis}).
These transformations do not change the order of the diffusion coefficients for
$y$ and $z$.  
\end{proof}

\subsection{Approach}
\label{ssec:fn-approach}

In this section, we consider solutions of the normal form~\eqref{eq:fn-nform}, 
starting at a fast time $s_0\asymp-1$ on $\Sigma_1$, as long as 
$x\geqs\Order{\sqrt{\eps}\,}$. We fix a deterministic solution
$(x^{\det}_s,y^{\det}_s,z^{\det}_s)$ which is sufficiently close to the
strong canard to display SAOs when approaching the folded-node point $p^*$. From
the deterministic analysis we know that
\begin{equation}
 \label{eq:fna01}
 a(s) = \partial_x f(x^{\det}_s,y^{\det}_s,z^{\det}_s) 
= -2x^{\det}_s 
= cs + \Order{s^2}
 \qquad
 \text{for $s_0\leqs s\leqs -\sqrt{\eps}$\;,}
\end{equation} 
where $c$ is a constant of order $1$. Scaling time if necessary, we may assume
that $c=1$. 
The linearization of the deterministic system at
$(x^{\det}_s,y^{\det}_s,z^{\det}_s)$ has the form $\eps\dot\zeta =
\cA(s)\zeta$, where 
\begin{equation}
 \label{eq:fna02}
 \cA(s) = 
 \begin{pmatrix}
 A(s) & c_1(s) \\ 0 & 0
 \end{pmatrix}\;, 
 \qquad
 A(s) = 
 \begin{pmatrix}
 2a(s) & 1 + \Order{s^2} \\
 -\eps(1+\mu) + \Order{\eps s} & \Order{\eps} 
 \end{pmatrix}\;,
\end{equation} 
and $c_1(s) = \transpose{(1+\Order{s^2},\Order{\eps})}$. We have used the fact that $s^2
\geqs \eps$ to simplify the expression of the error terms. 

For $s\leqs -\sqrt{\eps}$,  the eigenvalues of $A(s)$ behave like $s$ and $\eps/\abs{s}$. 
This implies that while for $s\asymp -1$, the variable $x$ is faster than both
$y$ and $z$, $\dot x$ and $\dot y$ become of comparable order $1/\sqrt{\eps}$ as
$s$ approaches $-\sqrt{\eps}$. This is the well-known effect that one may extend
the normally hyperbolic theory slightly near fold points from $x\asymp1$ up to a
neighbourhood with
$x\asymp\sqrt\eps$, see~\cite[pp.~48--49]{BronsKrupaWechselberger}. Instead of
blocking $y$ and $z$ as in~\eqref{def_zeta}, we write 
\begin{equation}
\label{def_zeta_fn} 
\xi_s =
\begin{pmatrix}
x_s \\ y_s 
\end{pmatrix}
-
\begin{pmatrix}
x^{\det}_s \\ y^{\det}_s 
\end{pmatrix}\;,
\qquad 
\eta_s = z_s - z^{\det}_s\;,
\qquad  
\zeta_s =
\begin{pmatrix}\xi_s\\ \eta_s\end{pmatrix}\;,
\end{equation}
since $\dot x$ and $\dot y$ eventually become comparable. Then $\zeta_s$ obeys
a system of the form 
\begin{equation}
\label{SDE_zeta_fn} 
 \6\zeta_s = \frac{1}{\eps} \cA(s)\zeta_s \6s + 
 \begin{pmatrix}
 \frac{\sigma}{\sqrt{\eps}} \cF_1(\zeta_s,s) + \sigma' \cF_2(\zeta_s,s) \\ 
 \sigma' \cG_1(\zeta_s,s) \\
 \sigma' \cG_2(\zeta_s,s)
 \end{pmatrix}
 \6W_s + 
 \begin{pmatrix}
 \frac{1}{\eps} b_x(\zeta_s,s) \\ b_y(\zeta_s,s) \\ 0
 \end{pmatrix}
 \6s\;.
\end{equation} 
The principal solution of $\eps\dot\zeta = \cA(s)\zeta$ has the block structure
\begin{equation}
 \label{eq:fna04}
 U(s,r) = 
 \begin{pmatrix}
 V(s,r) & \displaystyle \frac1\eps \int_s^r V(s,u) c_1(u) \6u \\ 0 & 1
 \end{pmatrix}\;,
\end{equation} 
where $V(s,r)$ denotes the principal solution of $\eps\dot\xi = A(s)\xi$. 

\begin{lem}
\label{lem_V}
For $s_0 \leqs r \leqs s \leqs -\sqrt{\eps}$, the matrix elements of $V(s,r)$ satisfy
\begin{align}
\nonumber
V_{xx}(s,r) &= \cO\biggl( \frac{\left|a(r)\right|^{1+\mu}}{\left|a(s)\right|^{1+\mu}}
\e^{\alpha(s,r)/\eps} \biggr)\;, \\
\nonumber
V_{xy}(s,r) &= \cO\biggl( \frac{\left|a(r)\right|^{\mu}}{\left|a(s)\right|^{1+\mu}}
\e^{\alpha(s,r)/\eps} + \frac{\left|a(s)\right|^{\mu}}{\left|a(r)\right|^{1+\mu}} \biggr)\;, \\
\label{eq:fna05} 
V_{yx}(s,r) &= \cO\biggl( \eps\, \frac{\left|a(r)\right|^{1+\mu}}{\left|a(s)\right|^{2+\mu}}
\e^{\alpha(s,r)/\eps} + \eps\, \frac{\left|a(s)\right|^{1+\mu}}{\left|a(r)\right|^{2+\mu}} \biggr)\;, \\
V_{yy}(s,r) &= \cO\biggl( \frac{\left|a(s)\right|^{1+\mu}}{\left|a(r)\right|^{1+\mu}}
+ \eps\, \frac{\left|a(r)\right|^{\mu}}{\left|a(s)\right|^{2+\mu}} \e^{\alpha(s,r)/\eps} \biggr)\;,
\nonumber  
\end{align}
where
\begin{equation}
 \label{eq:fna06} 
 \alpha(s,r) = \int_r^s \Tr A(u) \6u 
 =  (s^2 - r^2) + \Order{(s-r)(s^2 + r^2)}\;.
\end{equation} 
\end{lem}

In the particular case where $A(s)= \bigl(
\begin{smallmatrix} 2s & 1 \\ -\eps(1+\mu) & 0 \end{smallmatrix} \bigr)$,
the equation $\eps\dot\xi = A(s)\xi$ for $\xi_{s}=\transpose{(\xi_{1,s},\xi_{2,s})}$ is equivalent to a Weber equation 
\begin{equation}
 \label{eq:fna07} 
 \eps \frac{\6^2\xi_1}{\6s^2} -2s \frac{\6\xi_1}{\6s}+(\mu-1)\xi_1=0\;,
\end{equation} 
and the estimates~\eqref{eq:fna05} follow directly from the asymptotics of
parabolic cylinder functions~\cite[p.~689]{AS}; see also~\cite[p.~449]{Wechselberger}. 
In Appendix~\ref{app_lem_V}, we
provide a proof of Lemma~\ref{lem_V} valid in the general case, which does
not rely on these asymptotics. With the above estimates, we obtain the following 
result on the size of fluctuations during the approach phase. 

\begin{prop}
\label{prop_fn_approach}
Define the stopping times
\begin{align}
\nonumber 
 \tau_{\xi,1} &= \inf\{s>s_0\colon\abs{\xi_{1,s}}>h\}\;, \\
\nonumber 
 \tau_{\xi,2} &= \inf\{s>s_0\colon\abs{\xi_{2,s}}>h_1\}\;, \\
 \tau_{\eta} &= \inf\{s>s_0\colon\abs{\eta_s}>h_2\}\;.
 \label{eq:fna08a}
\end{align} 
There exist constants $\kappa, h_0 > 0$ such that for all $s_0 \leqs s \leqs
-\sqrt{\eps}$, and all $h, h_1, h_2>0$ satisfying $h^2 + h_1^2 + h_2^2 \leqs
h_0\abs{s}h$ and $h^2 + h_1^2 + h_2^2 \leqs h_0h_1$, 
\begin{align}
\label{eq:fna08b}
\P \bigl\{ \tau_{\xi,1} \wedge \tau_{\xi,2} \wedge \tau_{\eta} < s \bigr\}
\leqs \biggintpartplus{\frac {s-s_0}{\eps}} \biggl[{}&\exp \biggl\{ -
\frac{\kappa h^2}{(\sigma^2+(\sigma')^2)\left|s\right|^{-1}}\biggr\} \\
&{}+{} \exp \biggl\{ -
\frac{\kappa h_1^2}{\sigma^2\eps\left|s\right|^{-1}+(\sigma')^{2}\left|s\right|}\biggr\} 
+ \exp \biggl\{ -
\frac{\kappa h_2^2}{(\sigma')^2}\biggr\}\biggr]\;.  
\nonumber
\end{align}
\end{prop}

\begin{proof}
The proof is similar to the proof of Proposition~\ref{prop_tau_fold}, so we
omit the details. Let us just remark that when evaluating the 
elements in~\eqref{eq:fna04}, one encounters integrals of the form
\begin{equation}
 \label{eq:fna09}
 \int_{s_0}^s \frac{\left|s\right|^\mu}{\left|u\right|^{1+\mu}} \6u 
 = \frac{1-y^\mu}{\mu}
 = \left|\log y\right| \frac{1-\e^{-\mu\left|\log y\right|}}{\mu \left|\log y\right|}
\end{equation} 
where $y=\left|s\right|/\left|s_0\right|$. The fraction on the right-hand side being bounded, the
integral is bounded by a constant times $\left|\log(\left|s\right|/\left|s_0\right|)\right|$. 
\end{proof}

Under the condition $\sigma, \sigma' = \Order{\eps^{3/4}}$, we obtain 
the typical spreadings
\begin{itemize}
 \item 	$(\sigma+\sigma')\left|s\right|^{-1/2}$ in the $x$-direction, 
 which reaches order $(\sigma+\sigma')\eps^{-1/4}$ for $s\asymp-\sqrt{\eps}$, 
 \item 	$\sigma\eps^{1/2}\left|s\right|^{-1/2} + \sigma'\left|s\right|^{1/2}$ in the $y$-direction, 
 which reaches order $(\sigma+\sigma')\eps^{1/4}$ for $s\asymp-\sqrt{\eps}$, 
 \item 	and $\sigma'$ in the $z$-direction. 
\end{itemize}
Note carefully that the integrals in~\eqref{eq:fna09} become unbounded when
$s\ra 0$ as $\mu\in(0,1)$ so we cannot use the same methods to control 
sample paths closer to the folded node.

\subsection{Neighbourhood -- Deterministic dynamics}
\label{ssec:fn-nbh-det}

\begin{figure}
\centerline{\includegraphics*[clip=true,width=\textwidth]{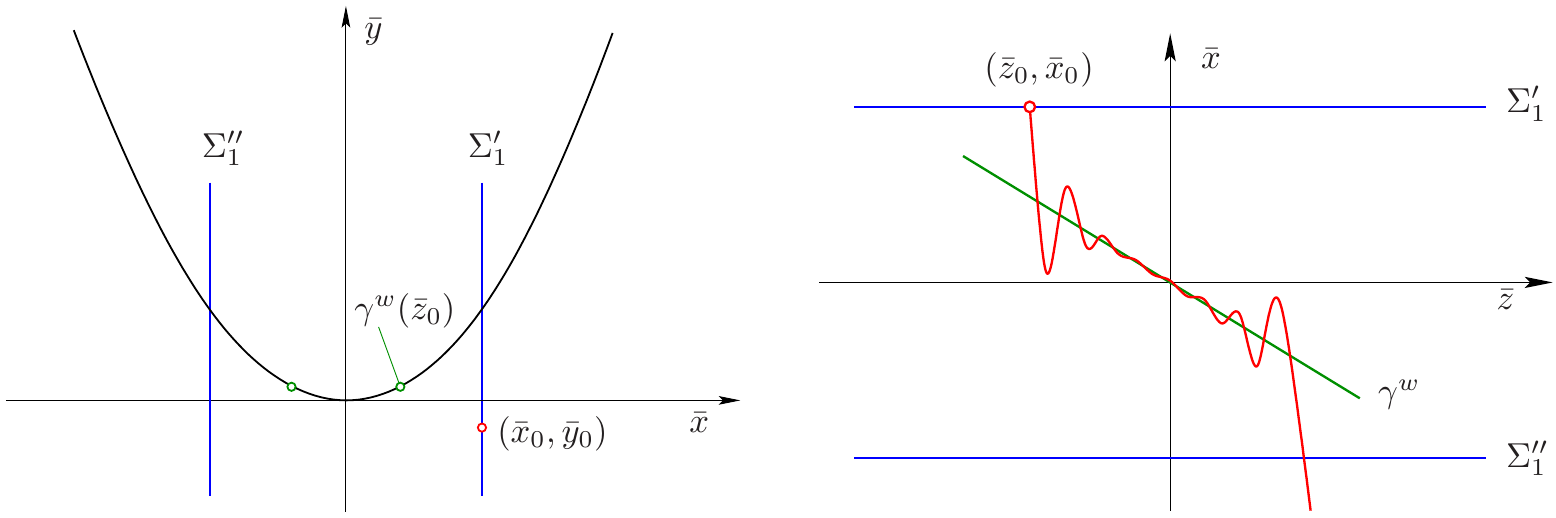}}
 \vspace{2mm}
\caption[]{Sketch of the geometry of the orbits near the folded-node
singularity.}
\label{fig_fn_poincare}
\end{figure}

In this section we briefly describe the behaviour of solutions of the normal
form~\eqref{eq:fn-nform} in the deterministic case $\sigma=\sigma'=0$. Recall
that standard results
(see~\cite[Section~4]{Wechselberger},~\cite[Theorem~2.3]{KuehnMMO} 
and the foundational work~\cite{Benoit1,Benoit5}) imply the existence of two 
primary canards and $k_\mu$ secondary canards~\cite{WechselbergerFN} where 
\begin{equation}
\label{eq:range_mu}
2k_\mu + 1 < \mu^{-1} < 2k_\mu + 3 \;,
\end{equation}
and each canard lies in $C_\eps^r\cap C^{a+}_\eps$. One strategy to
prove the existence of canards, as first suggested for the planar case
in~\cite{DRvdP}, 
is to look for transversal intersections of $C_\eps^r$ and $C^{a+}_\eps$ 
by extending the manifolds via the blow-up method~\cite{Wechselberger} into
a region near the folded node where the blow-up reduces to the
scaling (or zoom-in) transformation 
\begin{equation}
 \label{eq:fn_nbh_det:01}
 x = \sqrt{\eps}\, \bar x\;,\quad
 y = \eps\, \bar y\;,\quad
 z = \sqrt{\eps}\, \bar z\;.
\end{equation} 
The scaling~\eqref{eq:fn_nbh_det:01} transforms the deterministic version of 
the normal form~\eqref{eq:fn-nform} to 
\begin{align}
\nonumber
\mu \frac{\6\bar x}{\6\bar z} &= 2\bar y - 2\bar x^2 + \Order{\sqrt{\eps}\,}\;,
\\
\mu \frac{\6\bar y}{\6\bar z} &= -2(1+\mu)\bar x - 2\bar z +
\Order{\sqrt{\eps}\,}\;.
 \label{eq:fn_nbh_det:02}
\end{align}
We consider henceforth the dynamics for $\eps=0$, as results can be extended to
small positive $\eps$ by regular perturbation theory. Note that the system is 
symmetric under the transformation 
\begin{equation}
 \label{eq:fn_nbh_det:02B}
 (\bar x, \bar y, \bar z) \mapsto (-\bar x, \bar y, -\bar z)\;.
\end{equation}
The normal form admits a particular solution $\gamma^{w}$ given by 
\begin{equation}
 \label{eq:fn_nbh_det:03}
 \bar x = - \bar z\;, \qquad \bar y = \bar z^2 - \frac\mu2\;,
\end{equation} 
which is called the singular weak canard (there is also a singular strong
canard,
given by $\bar x = - \bar z/\mu$, $\bar y = (\bar z/\mu)^2 - 1/2$). Generic
solutions twist a certain number of times around the weak canard, see~\figref{fig_fn_poincare} for an illustration. One possibility to prove the persistence
of the weak and strong canards as well as secondary canards is to 
analyse the zeros of the variational Weber equation as shown
in~\cite{Wechselberger}. 
To also obtain estimates on individual non-canard orbits, our aim is to 
determine the map from an initial
condition $P_0=(\delta, \bar y_0, \bar z_0) \in \Sigma_1'$, close to the
attracting slow manifold, to the first-hitting point $P_1=(-\delta, \bar y_1,
\bar z_1) \in \Sigma_1''$. The key tool will be suitable coordinate 
transformations; we note that although the method only provides a small
refinement of previous results, it has the advantage of being quite explicit 
so we choose to record the results here. We will proceed in three steps, see~\figref{fig_fn_SAOeta}:

\begin{enumerate}
\item[(S1)] Estimate the coordinates of $P'=(0,\bar y',\bar z')$, the
first-hitting point of $\{\bar x=0\}$. 
\item[(S2)] Use an averaging-type transformation to describe the rotations of
this orbit
around the weak canard, until the last-hitting point $P''=(0,\bar y'',\bar z'')$
of $\{\bar x=0\}$. 
\item[(S3)] Determine the map from $P''$ to
$P_1$.  
\end{enumerate}

\begin{figure}
\centerline{\includegraphics*[clip=true,width=120mm]{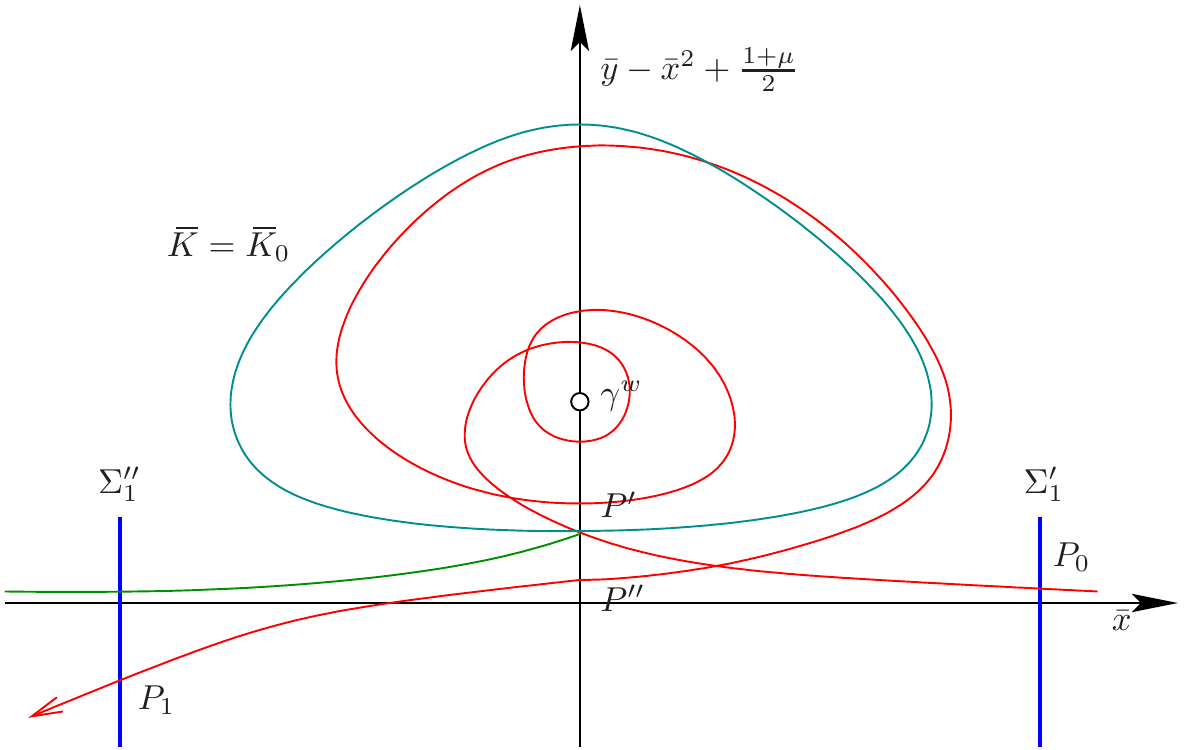}}
\caption[]{The map from $\Sigma'_1$ to $\Sigma''_1$ is decomposed into three
phases. 
}
\label{fig_fn_SAOeta}
\end{figure}

For Steps (S1) and (S3), it is useful
to introduce the rectified coordinate $\eta = \bar y - \bar x^2 + (1+\mu)/2$.
The normal form (with $\eps=0$) in rectified coordinates then reads
\begin{align}
\nonumber
\mu \frac{\6\bar x}{\6\bar z} &= 2\eta - (1+\mu)\;,
\\
\mu \frac{\6\eta}{\6\bar z} &= -4\bar x\eta - 2\bar z\;.
 \label{eq:fn_nbh_det:101}
\end{align}

\begin{lem}
\label{lem_fn_etax}
Fix constants $0<\beta<\alpha\leqs 1$ and
$0<L\leqs\sqrt{(\alpha-\beta)\abs{\log\mu}/2}$. 
Then for $\mu>0$ small enough the orbit of~\eqref{eq:fn_nbh_det:101} passing
through a point 
$(\bar x,\eta,\bar z)=(0,\eta_*,\bar z_*)$ with $\abs{\eta_*}\leqs\mu^\alpha$
and $\abs{\bar z_*}\leqs\mu^\beta$ is given by 
\begin{align}
\nonumber
\eta(\bar x) &= \e^{2\bar x^2}
\biggl[ \eta_* + \bar z_*\int_0^{2\bar x}\e^{-u^2/2}\6u + \Order{\mu}\biggr]
\bigl[ 1+\Order{\mu^\beta L^2}\bigr] \;, \\
\bar z(\bar x) &= \bar z_* + \mu\bar x \bigl[ 1+\Order{\mu^\beta} \bigr]\;,
\label{eq:fn_nbh_det:101a} 
\end{align}
for all $\bar x \in [-L,L]$. 
\end{lem}

\begin{proof}
For $\eta=\Order{\mu^\beta}$ 
the equations~\eqref{eq:fn_nbh_det:101} can be rewritten in the form  
\begin{align}
\nonumber
\frac{\6\bar z}{\6\bar x} &= \mu [1+\Order{\mu^\beta}]\;,
\\
\frac{\6\eta}{\6\bar x} &= a(\bar x)\eta + 2\bar z[1+\Order{\mu^\beta}]\;, 
 \label{eq:fn_nbh_det:101b}
\end{align}
where $a(\bar x) = 4\bar x[1+\Order{\mu^\beta}]$. Integrating the first
equation yields the expression for $\bar z(\bar x)$. To obtain the expression
for $\eta(\bar x)$, observe that $\exp(\int_0^{\bar x}a(y)\6y) = \e^{2\bar
x^2}[1+\Order{\mu^\beta L^2}]$ and solve the equation for $\eta$ by variation
of the constant. 
\end{proof}

From this result we immediately see that the map from $P_0$ to $P'$ is given
by 
\begin{align}
\nonumber
\eta' &= \e^{-2\delta^2} \eta_0 [1+\Order{\delta^2\mu^\beta}]
- \bar z_0\int_0^{2\delta}\e^{-u^2/2}\6u  + \Order{\mu}\;, \\
\bar z' &= \bar z_0 + \Order{\delta\mu}\;,
\label{eq:fn_nbh_det:101c} 
\end{align}
provided $\eta_0, \bar z_0 = \Order{\mu^\beta}$ for some $\beta\in(0,1)$. 
Similarly, the map from $P''$ to $P_1$ is given by 
\begin{align}
\nonumber
\eta_1 &= \e^{2\delta^2} \biggl[\eta'' 
- \bar z''\int_{-2\delta}^0\e^{-u^2/2}\6u  +
\Order{\mu}\biggr][1+\Order{\delta^2\mu^\beta}]\;, \\
\bar z_1 &= \bar z'' + \Order{\delta\mu}\;,
\label{eq:fn_nbh_det:101d} 
\end{align}
provided $\eta''=\Order{\mu^\alpha}$ and $\bar z''=\Order{\mu^\beta}$ for some
choice of $0<\beta<\alpha\leqs1$. In addition, Lemma~\ref{lem_fn_etax} shows
that for sufficiently small $\bar z$, the distance at $\bar x=0$ between the
invariant manifolds $C^{a+}_\eps$ and $C^r_\eps$ has order $\bar z$. This
follows from the fact that orbits in $C^{a+}_\eps$ should be such that
$\eta(\bar x)$ is close to $0$ for large positive $\bar x$, while orbits in
$C^r_\eps$ should be such that $\eta(\bar x)$ is close to $0$ for large negative
$\bar x$. 

We now turn to Step (S2), estimating the map from $P'$ to $P''$. 
The difference $u$ between a general solution of~\eqref{eq:fn_nbh_det:02} and
the weak canard $\gamma^w$ satisfies the variational equation 
\begin{align}
\nonumber
\mu \frac{\6 u_1}{\6\bar z} &= 4\bar z u_1  + 2 u_2 - 2u_1^2 \;,
\\
\mu \frac{\6 u_2}{\6\bar z} &= -2(1+\mu)u_1\;.
 \label{eq:fn_nbh_det:106}
\end{align}
Consider the variable 
\begin{equation}
 \label{eq:fn_nbh_det:107}
 K := \biggl[
 1+\frac{2}{1+\mu}(u_2-u_1^2)
 \biggr]
 \e^{-2u_2/(1+\mu)}\;,
\end{equation} 
which is a first integral of the system when $\bar z=0$. In fact, $K$ is just a 
version of the classical first integral near planar degenerate folded
singularities 
({cf.}~\cite[Lemma~3.3;~Figure~3.2]{KruSzm4},~\cite[Figure~5]
{KrupaPopovicKopell},~\cite[Figure~2]{MuratovVanden-Eijnden}). Although $K$ is
not a first integral 
for arbitrary $\bar z$ it turns out that the variable $K$ is still very useful
for obtaining explicit control over the oscillations near the folded node. 
A short computation yields  
\begin{equation}
 \label{eq:fn_nbh_det:108}
 \mu\frac{\6K}{\6\bar z} = -\frac{16 \bar z}{1+\mu} u_1^2 \e^{-2u_2/(1+\mu)}\;.
\end{equation} 
In~\cite[Section~D.2]{BGK12} we provided an averaging result
valid in a small neighbourhood of the weak canard (for the stochastic case). 
The following result extends this to the larger domain $\{K>0\}$. 

\begin{prop}[Averaged system]
\label{prop_fn_averaging}
Set $\rho(K) = (1-K)(1+\abs{\log K}^{3/2})$. 
For $\bar z$ in a neighbourhood of $0$ and $K>0$, there exist a variable
$\Kbar = K + \Order{\left|\bar z\right| \rho(K)}$, an angular variable
$\varphi$, a function $\bar g$ and  constants $c_\pm>0$ such that
\begin{equation}
 \label{eq:fn_nbh_det:109}
 \frac{c_-}{1+\abs{\log K}^{1/2}} 
 \leqs \mu \frac{\6\varphi}{\6 \bar z}
 \leqs c_+ (1+\abs{\log K}^{1/2})
\end{equation} 
for $K\geqs\Order{\left|\bar z\right|}$ 
and 
\begin{equation}
 \label{eq:fn_nbh_det:110}
 \frac{\6\Kbar}{\6\varphi}
 = \bar z \bar g(\Kbar,\bar z) + 
 \cO\Bigl( (\mu+\bar z^2) \rho(\Kbar)\Bigr)\;,
\end{equation} 
where $c_-(1-\Kbar) \leqs -\bar g(\Kbar,\bar z) \leqs c_+(1-\Kbar)$. 
\end{prop}

We give the proof in Appendix~\ref{appendix_averaging}. 

The averaged equation~\eqref{eq:fn_nbh_det:110} is
similar to the equation describing dynamic pitchfork or Hopf bifurcations, which
display a bifurcation delay.
Initially, i.e.\ when $\bar z=\bar z'<0$, $\Kbar$~has a value $\Kbar_0>0$ of order
$\bar z_0$.  As long as $\bar z<0$, $\Kbar$ will keep increasing, and thus get
so close to $1$ that a time of order $\abs{\bar z'}$ is needed, once $\bar z$
becomes positive, for $\Kbar$ to decrease to
the value $\Kbar_0$ again. We set 
\begin{equation}
 \label{eq:fn_nbh_det:111}
 \tau = \inf\setsuch{\bar z>0}{\Kbar = \Kbar_0}\;.
\end{equation} 
Note that the error term in~\eqref{eq:fn_nbh_det:110} is no longer negligible when $\left|\bar z\right|$ is of order $\mu$, but this only results in a shift of the
delay by a quantity of order $\mu$, which will be negligible.

\begin{cor}
\label{cor_averaging}
Let $\rho_1 = (\mu + \bar z_0^2)\abs{\log\bar z_0}^{3/2}$. 
We have 
\begin{equation}
 \label{eq:fn_nbh_det:112}
 \tau = -\bar z' + 
 \Order{\rho_1}\;.
\end{equation} 
Furthermore, the change in angle is given by $\varphi(\tau) - \varphi(\bar
z') = \phi(\bar z')/\mu$, where $\phi$ is monotonically decreasing for $\bar
z'<0$ and satisfies 
\begin{equation}
 \label{eq:fn_nbh_det:113}
 \frac{2c_-}{1+\abs{\log \bar z_0}^{1/2}}
 \bigl[ \abs{\bar z'} + \Order{\rho_1} \bigr]
 \leqs \phi(\bar z') 
 \leqs 2c_+ (1+\abs{\log \bar z_0}^{1/2})
 \bigl[ \abs{\bar z'} + \Order{\rho_1} \bigr]\;.
\end{equation} 
\end{cor}

\begin{proof}
Set $Q=1-\Kbar$. As long as $\bar z \leqs \tau$, we can bound $\abs{\log \Kbar}$
by $\abs{\log \bar z_0}$ and write 
\begin{equation}
 \label{eq:fn_nbh_det:114}
 \frac{\6Q}{\6\varphi} 
 \leqs Q
 \Bigl[ c_+ \bar z   + \cO\bigl( Q(\mu+\bar z^2) \abs{\log
\bar z_0}^{3/2}\bigr)\Bigr]\;. 
\end{equation} 
Using~\eqref{eq:fn_nbh_det:109} we obtain 
\begin{equation}
 \label{eq:fn_nbh_det:115}
 \frac{\6Q}{\6\bar z} 
 \leqs \frac1\mu Q
 \Bigl[ c_+^2 \bar z   
+ \cO\bigl( Q(\mu+\bar z^2) \abs{\log \bar z_0}^{3/2}\bigr)\Bigr]
\bigl[1+\Order{\abs{\log \bar z_0}^{1/2}}\bigr]\;. 
\end{equation} 
Integrating, we arrive at 
\begin{equation}
 \label{eq:fn_nbh_det:116}
 Q(\bar z) \leqs Q(\bar z') 
 \exp \biggl\{ \frac{c_+^2}{2\mu} (\bar z - \bar z')
 \bigl[ \bar z + \bar z' + \Order{\rho_1}\bigr]
\bigl[1+\Order{\abs{\log \bar z_0}^{1/2}}\bigr]
 \biggr\}\;.
\end{equation} 
This shows that $\tau \geqs -\bar z' + \Order{\rho_1}$. 
Using the corresponding lower bounds, we also get 
$\tau \leqs -\bar z' - \Order{\rho_1}$. This proves~\eqref{eq:fn_nbh_det:112},
and~\eqref{eq:fn_nbh_det:113} follows by using~\eqref{eq:fn_nbh_det:109} again.
\end{proof}

We can now draw consequences on the Poincar\'e map from the last results.
If 
\begin{equation}
 \label{eq:fn_nbh_det:117}
 \phi(\bar z') = 2\pi n\mu\;, 
 \qquad n\in\N\;,
\end{equation} 
then the orbit will hit the plane $\{\bar x=0\}$ at $P''=P'$, which is on
(or very near) the repelling slow manifold $C^r_\eps$, cf.~\eqref{eq:fn_nbh_det:101c}. Therefore~\eqref{eq:fn_nbh_det:117} gives a
condition on $\bar z'$ (and thus on $\bar z_0$) for the orbit being a canard.
If, on the other hand, 
\begin{equation}
 \label{eq:fn_nbh_det:118}
 \phi(\bar z') = (2\pi n-\theta)\mu\;,
 \qquad
 0 < \theta < 2\pi\;,
\end{equation} 
the orbit will leave the set $\{\Kbar>\Kbar_0\}$ far from the repelling slow
manifold $C^r_\eps$, see \figref{fig_fn_SAOeta}. One can then use
Proposition~\ref{prop_fn_averaging} to estimate $\bar
z''$, which is of the form $\bar z'' = \tau + \Order{\mu\theta}$. As
$\theta$ increases from $0$ to $2\pi$, $P''$ moves downwards until it
approaches the continuation of the attracting slow manifold $C^{a+}_\eps$.

Once orbits have hit $\bar x=0$ at some point $P''$ below $P'$, one can
use~\eqref{eq:fn_nbh_det:101d} to follow their future evolution. 
Note in particular that the domain $\{\bar x<0, \eta<0\}$ is positively
invariant, so that once orbits have reached this domain they will stay bounded
away from the repelling slow manifold. 

\subsection{Neighbourhood -- Stochastic dynamics}
\label{ssec:fn-nbh-stoch}

We now consider the stochastic dynamics of sample paths starting on $\Sigma'_1$
up to the first time they hit the section $\Sigma''_1 = \{x=-\delta\sqrt\eps\,\}$. 
The first step is again to apply the scaling (or zoom-in) 
\begin{equation}
 \label{eq:fn_nbh_stoch:01}
 x = \sqrt{\eps}\, \bar x\;,\quad
 y = \eps\, \bar y\;,\quad
 z = \sqrt{\eps}\, \bar z\;,\quad
 \frac{1}{2} \mu\sqrt{\eps}\, t = \theta\;,
\end{equation} 
which transforms the normal form~\eqref{eq:fn-nform} into 
\begin{alignat}{3}
\nonumber
\6\bar x_\theta  &{}=
\frac2\mu \bigl[ \bar y_\theta - \bar x_\theta^2 + \Order{\sqrt{\eps}\,} \bigr]
\6\theta 
&&{}+ 
\sqrt{\frac2\mu}\, \bigl[
\bar\sigma \widehat F_1(\bar x_\theta,\bar y_\theta,\bar z_\theta) +
\bar\sigma'\sqrt\eps\,
\widehat F_2(\bar x_\theta,\bar y_\theta,\bar z_\theta) \bigr] \6W_\theta\;,\\ 
\nonumber
\6\bar y_\theta  &{}= 
\frac2\mu \bigl[ -(1+\mu)\bar x_\theta - \bar z_\theta + \Order{\sqrt{\eps}\,}
\bigr]
\6\theta 
&&{}+ \sqrt{\frac2\mu}\,
\bar\sigma' \widehat G_1(\bar x_\theta,\bar y_\theta,\bar z_\theta)
\6W_\theta\;,\\ 
\6\bar z_\theta  &{}= \6\theta 
&&{}+ \sqrt{\frac2\mu}\,
\bar\sigma'\sqrt\eps\, \widehat G_2(\bar
x_\theta,\bar y_\theta,\bar z_\theta)\6W_\theta\;,
\label{eq:fn_nbh_stoch:02}
\end{alignat}
where
\begin{equation}
\label{eq:fn_nbh_stoch:03}
\bar\sigma = \eps^{-3/4}\sigma
\qquad\text{and}\qquad
\bar\sigma' = \eps^{-3/4}\sigma'\;.
\end{equation} 
The deviation $\zeta_\theta$ from the deterministic solution
$(\bar x^{\det}_\theta,\bar y^{\det}_\theta,\bar z^{\det}_\theta)$, defined as
in~\eqref{def_zeta_fn}, satisfies a SDE of the form 
\begin{equation}
\label{eq:fn_nbh_stoch:04} 
 \6\zeta_\theta = \frac{1}{\mu} \cA(\theta)\zeta_\theta \6\theta + 
 \frac{1}{\sqrt{\mu}} 
 \begin{pmatrix}
 \bar\sigma\cF_1(\zeta_\theta,\theta) + \bar\sigma'\sqrt\eps
\,\cF_2(\zeta_\theta,\theta) \\ 
 \bar\sigma' \cG_1(\zeta_\theta,\theta) \\
 \bar\sigma'\sqrt\eps\, \cG_2(\zeta_\theta,\theta)
 \end{pmatrix}
 \6W_\theta + \frac1\mu
 \begin{pmatrix}
 b_x(\zeta_\theta,\theta) \\ \sqrt\eps\,b_y(\zeta_\theta,\theta) \\ 0
 \end{pmatrix}
 \6\theta\;.
\end{equation} 
The principal solution of $\mu\dot\zeta = \cA(\theta)\zeta$ has a block
structure similar to~\eqref{eq:fna04}. Provided we take $\delta$ sufficiently
small, the upper left block $A(\theta)$ has complex conjugated eigenvalues
$a(\theta)\pm2\icx\omega(\theta)$, where 
\begin{equation}
 \label{eq:fn_nbh_stoch:05}  
 a(\theta) = -2\bar x^{\det}_\theta + \Order{\sqrt{\eps}\,}\;, \qquad 
 \omega(\theta) = \sqrt{1 - (\bar x^{\det}_\theta)^2 + \mu} \, +
\Order{\sqrt{\eps}\,}\;. 
\end{equation}
By~\cite[Theorem~4.3]{BGK12}, the principal solution $V(\theta,\phi)$ of
$\mu\dot\xi = A(\theta)\xi$ can be written in the form 
\begin{equation}
 \label{eq:fn_nbh_stoch:06} 
V(\theta,\theta_0) = \e^{\alpha(\theta,\theta_0)/\mu} S(\theta) 
\begin{pmatrix}
\cos(\phi(\theta,\theta_0)/\mu) & \sin(\phi(\theta,\theta_0)/\mu) \\
-\sin(\phi(\theta,\theta_0)/\mu) & \cos(\phi(\theta,\theta_0)/\mu)
\end{pmatrix}
S(\theta_0)^{-1}\;, 
\end{equation} 
where 
\begin{equation}
 \label{eq:fn_nbh_stoch:07}
 S(\theta) = \frac{1}{\sqrt{\omega(\theta)}}
 \begin{pmatrix}
 - \theta + \omega(\theta) & -\theta - \omega(\theta) \\
 1 & 1 
 \end{pmatrix} + \Order{\mu}\;,
\end{equation} 
and 
\begin{equation}
 \label{eq:fn_nbh_stoch:08}
\alpha(\theta,\theta_0) = \int_{\theta_0}^\theta a(\psi)\6\psi 
\;, 
\qquad
\phi(\theta,\theta_0) = \int_{\theta_0}^\theta 2\omega(\psi)\6\psi
+\Order{\mu}\;.  
\end{equation} 
The off-diagonal term in the principal solution of $\mu\dot\zeta =
\cA(\theta)\zeta$ has the form 
\begin{align}
\nonumber
U_{\xi\eta}(\theta,\theta_0) &= \frac{1}{\mu} \int_{\theta_0}^\theta
V(\theta,\psi)
c_1(\psi)\6\psi \\
&= \frac{1}{\mu} \int_{\theta_0}^\theta \e^{\alpha(\theta,\psi)/\mu}
\biggl[ \cos\biggl( \frac{\phi(\theta,\psi)}{\mu}\biggr) v_1 
+ \sin\biggl( \frac{\phi(\theta,\psi)}{\mu}\biggr) v_2\biggr]\6\psi
 \label{eq:fn_nbh_stoch:09}
\end{align} 
for some vectors $v_1, v_2$. Using integration by parts and the fact that the
eigenvalues $a\pm\icx\omega$ are bounded away from $0$, one shows that the
elements of $U_{\xi\eta}$ are of order $1$ at most. The next proposition then
follows in the same way as before.

\begin{prop}
\label{prop_fn_nbh}
Define stopping times
\begin{align}
\nonumber 
 \tau_{\xi} &= \inf\{\theta>\theta_0\colon\norm{\xi_\theta}>h\}\;, \\
 \tau_{\eta} &= \inf\{\theta>\theta_0\colon\abs{\eta_\theta}>h_1\}\;.
 \label{eq:fn_nbh_stoch:10a}
\end{align} 
There exist constants $\kappa, h_0 > 0$ such that for all $\theta_0 \leqs \theta
\leqs \sqrt{\mu}$, and all $0<h, h_1 \leqs \sqrt{\mu}$, 
\begin{multline}
\P \bigl\{ \tau_\xi \wedge \tau_\eta > \theta \bigr\}
\leqs \biggintpartplus{\frac{\theta-\theta_0}{\mu}} \biggl(\exp \biggl\{ -
\frac{\kappa
[h-\mu^{-1/2}(h^2+h_1^2)]^2}{(\bar\sigma^2+(\bar\sigma')^2)\mu^{-1/2}}\biggr\}
\\
+ \exp \biggl\{ -
\frac{\kappa
h^2}{(\bar\sigma^2+(\bar\sigma')^2)\eps\mu^{-1}(\theta-\theta_0)}\biggr\}
+ \exp \biggl\{ -
\frac{\kappa h_1^2}{(\bar\sigma')^2\eps\mu^{-1}(\theta-\theta_0)}\biggr\}
\biggr)\;.
\label{eq:fn_nbh_stoch:10b}
\end{multline}
\end{prop}

These estimates show that if the deterministic solution hits $\Sigma_1''$ at a
point such that $\bar z\leqs\sqrt\mu$, and provided $\eps(\theta-\theta_{0})\leqs\sqrt\mu$ and $\bar\sigma+\bar\sigma' \ll
\mu^{3/4}$, the typical spreadings are 
\begin{equation}
 \label{eq:fn_nbh_stoch:11}
 \Delta\bar y \asymp (\bar\sigma + \bar\sigma')
 \biggl(\frac{1}{\mu^{1/4}} +
\frac{\eps^{1/2}(\theta-\theta_0)^{1/2}}{\mu^{1/2}}\biggr)
 \qquad\text{and}\qquad 
 \Delta\bar z \asymp
\frac{\bar\sigma'\eps^{1/2}(\theta-\theta_0)^{1/2}}{\mu^{1/2}}\;.
\end{equation} 
In particular, if $\theta-\theta_0\asymp \mu^{1/2}$, going back to original
variables we find that provided $\sigma+\sigma' \ll
(\eps\mu)^{3/4}$, the typical spreadings on $\Sigma_1''$ are of order 
\begin{itemize}
 \item 	$(\sigma+\sigma')(\eps/\mu)^{1/4}$ in the $y$-direction, 
 \item 	and $\sigma'(\eps/\mu)^{1/4}$ in the $z$-direction. 
\end{itemize}

\begin{rem}
Theorem~6.2 in~\cite{BGK12} provides a more precise description of the
dynamics, in a slightly simpler setting (in particular without noise on the
$z$-variable): it shows that sample paths concentrate in a \lq\lq covariance
tube\rq\rq\ centred in the deterministic solution. The size of the tube is
compatible with the above estimates on noise-induced spreading. Such a refined
analysis is possible in the present setting as well, but it would require some
more work, mainly in order to control the effect of the position-dependence of 
the noise term. 
\end{rem}

\begin{rem}
\label{rem_extension_logsigma} 
It is possible to extend Estimate~\eqref{eq:fn_nbh_stoch:10b} to
slightly larger $\theta$, at the cost of replacing $\mu^{-1/2}$ in the
denominator by $\mu^{-1/2}\e^{2c\theta^2/\mu}$ for some $c>0$. This is due to
the exponential growth of the variance for $\theta>\sqrt\mu$. 
\end{rem}

\subsection{Escape}
\label{ssec:fn-escape2}

In this subsection, we fix an initial condition $(-\delta\sqrt\eps,y_0,z_0) \in 
\Sigma_1''$, sufficiently close to the folded-node point $p^*$, and estimate
the fluctuations of sample paths up to their first hitting of $\Sigma_2 =
\{x=-\delta_0\}$. 

\begin{prop}
\label{prop_fn_escape} 
Denote by $(y^*,z^*)$ the point where the deterministic solution starting in
$(x_0=-\delta\sqrt\eps,y_0,z_0)\in\Sigma''_1$ first hits $\Sigma_2 = \{
x=-\delta_0 \}$. Assume $y_0 \leqs x_0^2 - \eps(\frac{1+\mu}{2}+c_0)$ for a
constant $c_0>0$, and $0\leqs z_0\leqs\Order{\sqrt{\eps}}$. 
For sufficiently small $\delta, \delta_0>0$, there exist $C,
\kappa, h_0>0$ such that 
for all $h_1, h_2>0$ satisfying $h_1\leqs h_0\eps$ and $h_2\leqs
h_0\sqrt{h_1}$, 
the stochastic sample path starting in $(x_0,y_0,z_0)$ first hits 
$\Sigma_2$ at time $\tau=\tau_{\Sigma_2}$ in 
a point $(-\delta_0, y_{\tau}, z_{\tau})$ such that 
\begin{multline}
 \label{fnee_01}
 \P^{(x_0,y_0,z_0)} \left\{ \abs{y_{\tau} - y^*} > h_1 \text{ or }
 \abs{z_{\tau} - z^*} > h_2 
 \right\} \\
 \leqs C\left|{\log\eps}\right| \biggl(
 \exp \biggl\{ -\frac{\kappa h_1^2}{(\sigma^2 +
(\sigma')^2)\sqrt\eps\,}\biggr\}
 + \exp \biggl\{ -\frac{\kappa h_2^2}{
(\sigma')^2\sqrt\eps\,}\biggr\}
 + \exp \biggl\{ -\frac{\kappa\eps^{3/2}}{\sigma^2 + (\sigma')^2\eps}\biggr\}
 \biggr)\;.
\end{multline}
The result remains true uniformly in initial conditions $(x_0,y,z)$ such that
$\abs{y-y_0}\leqs h_1$ and $\abs{z-z_0}\leqs h_2$. 
\end{prop}

\begin{proof}
The proof is basically the same as the proof of
Proposition~\ref{prop_fold_escape}, so we will omit its details. We introduce
sections $\Sigma^*_n=\{x=-\delta\sqrt{\eps}\,2^n, (y,z)\in D_n\}$, where each
$D_{n+1}$ is obtained by enlarging the image of $D_n$ under the deterministic
flow by order $h_12^{-n/2}$ in the $y$-direction, and by order
$h_22^{-n/2}$ in the $z$-direction. Choosing $D_1$ as a rectangle of
size $2h_1\times2h_2$ allows to deal with more general initial conditions. Let
$\tau^{(n)}$ denote the first-exit time from a block of dimensions $2h\times
2h_1\times 2h_2$ centred in a given deterministic solution. 
Then the analogue of Lemma~\ref{lem_rfn1} reads
\begin{equation}
 \label{fnee_01:1}
 \P^{(x_n,y_n,z_n)}\bigl\{
\tau_{n+1}\wedge\tau^{(n)} > c\delta2^{-n}\mu
\bigr\} 
 \leqs \exp\biggl\{
-\frac{\kappa\delta^22^{3n}\eps^{3/2}}{\sigma^2+(\sigma')^2\eps}
\biggr\}\;,
\end{equation} 
while the equivalent of Lemma~\ref{lem_rfn2} is  
\begin{align}
&\P^{(x_n,y_n,z_n)}\bigl\{ \tau^{(n)}  <
c\delta 2^{-n}\mu\bigr\} \\ \nonumber 
 &\quad \leqs
 \exp\biggl\{ -\frac{\kappa
h^2\sqrt{\eps}}{\delta[\sigma^2+(\sigma')^22^{-2n}]} \biggr\}
+ \exp\biggl\{ -\frac{\kappa
h_1^2}{\delta\sqrt{\eps}[\sigma^2\delta^22^{-2n}+(\sigma')^2]}
\biggr\}
 + \exp\biggl\{ -\frac{\kappa h_2^2}{\delta(\sigma')^2\sqrt{\eps}} \biggr\}
 \;.
 \label{fnee_01:2}
\end{align}
The remainder of the proof is similar to the proof of
Proposition~\ref{prop_fold_escape}, after redefining~$\kappa$. 
\end{proof}

This result shows that if $\sigma+\sigma'\ll\eps^{3/4}$, for a given initial
condition on $\Sigma_1''$, the spreading in the $y$- and $z$-directions on
$\Sigma_2$ is of order 
\begin{equation}
 \label{fnee02}
 (\sigma + \sigma')\eps^{1/4}
 \qquad\text{and}\qquad
 \sigma'\eps^{1/4}\;.
\end{equation}


\section{From noisy returns to Markov chains}
\label{sec:MC}

In this section we combine the results from the last three sections to obtain
estimates on the kernel $K$ of the random Poincar\'e map on $\Sigma_1$.
Table~\ref{table_deviations} summarizes the results obtained so far. For each
part of the dynamics, it shows the typical size of fluctuations when starting
in a point on the previous section. Deviations will not necessarily add up
because of the contraction during some phases of the motion. 

\begin{table}[h]
\begin{center}
\begin{tabular}{|l|l|c|c|c|}
\hline
\vrule height 14pt depth 6pt width 0pt
Transition & Analysis in & $\Delta x$ & $\Delta y$ & $\Delta z$ \\
\hline
\hline
\vrule height 14pt depth 6pt width 0pt
$\Sigma_2 \to \Sigma_3$ & Section~\ref{ssec:fast} & $\sigma + \sigma'$ & 
& $\sigma\sqrt{\eps} + \sigma'$ \\
\hline
\vrule height 14pt depth 6pt width 0pt
$\Sigma_3 \to \Sigma_4$ & Section~\ref{ssec:slow} & $\sigma + \sigma'$ & 
& $\sigma\sqrt{\eps} + \sigma'$ \\
\hline
\vrule height 20pt depth 12pt width 0pt
$\Sigma_4 \to \Sigma'_4$ &
Section~\ref{ssec:rfa} & 
$\dfrac{\sigma}{\eps^{1/6}} + \dfrac{\sigma'}{\eps^{1/3}}$ & 
& $\sigma\sqrt{\eps\abs{\log\eps}} + \sigma'$ \\
\hline
\vrule height 14pt depth 6pt width 0pt
$\Sigma'_4 \to \Sigma_5$ & Section~\ref{ssec:rfn} & 
& $\sigma\sqrt{\eps} + \sigma'\eps^{1/6}$ 
& $\sigma\sqrt{\eps} + \sigma'\eps^{1/6}$ \\
\hline
\vrule height 14pt depth 6pt width 0pt
$\Sigma_5 \to \Sigma_6$ & Section~\ref{ssec:fast} & $\sigma + \sigma'$ & 
& $\sigma\sqrt{\eps} + \sigma'$ \\
\hline
\vrule height 14pt depth 6pt width 0pt
$\Sigma_6 \to \Sigma_1$ & Section~\ref{ssec:slow} & $\sigma + \sigma'$ & 
& $\sigma\sqrt{\eps} + \sigma'$ \\
\hline
\hline
\vrule height 14pt depth 6pt width 0pt
$\Sigma_1 \to \Sigma'_1$ & Section~\ref{ssec:fn-approach} & 
& $(\sigma+\sigma')\eps^{1/4}$ & $\sigma'$ \\
\hline
\vrule height 14pt depth 6pt width 0pt
$\Sigma'_1 \to \Sigma''_1$ & Section~\ref{ssec:fn-nbh-stoch} &
& $(\sigma+\sigma')(\eps/\mu)^{1/4}$ & $\sigma'(\eps/\mu)^{1/4}$ \\
if $z=\Order{\sqrt{\mu\eps}}$ & & & & \\
\hline
\vrule height 14pt depth 6pt width 0pt
$\Sigma''_1 \to \Sigma_2$ & Section~\ref{ssec:fn-escape2} & 
& $(\sigma+\sigma')\eps^{1/4}$ & $\sigma'\eps^{1/4}$ \\
\hline
\end{tabular}
\end{center}
\vspace{2mm}
\caption{Summary of results on the size of fluctuations at the time of first hitting a section~$\Sigma_j$, when starting from a specific point on $\Sigma_i$, under the assumption $\sigma+\sigma' \ll \eps^{3/4}$, cf.~Figure~\ref{fig_sections}.}
\label{table_deviations} 
\end{table}

\subsection{The global return map}
\label{ssec_global_return_map} 

The following result describes the global return map $\Sigma_2\to\Sigma_1$.

\begin{thm}[Global return map]
\label{thm_global_returns} 
Fix $P_2=(x^*_2,y^*_2,z^*_2)\in\Sigma_2$. Assume the deterministic orbit
starting in $P_2$ hits $\Sigma_1$ for the first time in 
$P_1=(x^*_1,y^*_1,z^*_1)$. 
Then there exist constants $h_0, \kappa, C >0$ such that for all $h\leqs h_0$ and
$h^2/h_0 \leqs h_1 \leqs h$, the stochastic sample path starting in $P_2$ hits
$\Sigma_1$ for the first time in a point $(x_1,y^*_1,z_1)$ satisfying 
\begin{multline}
\P^{P_2} \bigl\{ \abs{x_1 - x^*_1} > h \text{ or } \abs{z_1 - z^*_1} > h_1
\bigr\} \\
\leqs 
\frac{C}{\eps} \biggl(
\exp\biggl\{ -\frac{\kappa h^2}{\sigma^2 + (\sigma')^2}\biggr\}
+ \exp\biggl\{ -\frac{\kappa h_1^2}{\sigma^2\eps\abs{\log\eps} +
(\sigma')^2}\biggr\} 
+ \exp\biggl\{ -\frac{\kappa\eps}{\sigma^2 +
(\sigma')^2\eps^{-1/3}}\biggr\} \biggr)\;.
\label{eq:return01} 
\end{multline}
\end{thm}

\begin{proof}
Denote by $(x^*_i,y^*_i,z^*_i)$ the deterministic first-hitting point of
section $\Sigma_i$, and by $(x_i,y_i,z_i)$ the corresponding random
first-hitting point. We use similar notations for $\Sigma'_4$. 
We will decompose the dynamics into three main steps, and
introduce the events 
\begin{align}
\nonumber
\Omega_1(h,h_1) &= 
\bigl\{ \abs{x'_4 - (x'_4)^*} \leqs h, 
\abs{z'_4 - (z'_4)^*} \leqs h_1\bigr\}\;, \\
\Omega_2(H_1) &= \bigl\{ \norm{(y_5,z_5) - (y^*_5,z^*_5)} \leqs H_1 \bigr\}\;. 
\label{eq:return02:1} 
\end{align}
\begin{itemize}
 \item {\em Step 1:\/} $\Sigma_2 \to \Sigma_3 \to \Sigma_4 \to \Sigma'_4$.
Propositions~\ref{prop_tau} and~\ref{prop_tau_fold} can be applied
simultaneously, because they are based on the same kind of estimates of the
principal solution. This directly yields the bound
\begin{equation}
 \label{eq:return02:2}
 \P^{P_2}\bigl(\Omega_1(h,h_1)^c\bigr) \leqs \frac{C}{\eps}
\biggl(
 \exp\biggl\{ -\frac{\kappa h^2}{\sigma^2\eps^{-1/3} +
(\sigma')^2\eps^{-2/3}}\biggr\}
 + \exp\biggl\{ -\frac{\kappa h_1^2}{\sigma^2\eps\abs{\log\eps} +
(\sigma')^2}\biggr\}\biggr)
\end{equation} 
for some $C>0$,
which is valid for all $h, h_1$ satisfying $h\leqs h_0\eps^{1/3}$, 
$h_1\leqs h_0$, $h^2\leqs h_0h_1$ and $h_1^2\leqs h_0h\eps^{1/3}$. 

 \item {\em Step 2:\/} $\Sigma'_4 \to \Sigma_5$. 
The difference $\zeta_\theta$ between two deterministic solutions starting on
$\Sigma'_4$ satisfies a relation of the form 
\begin{equation}
 \label{eq:return02:3}
 \zeta_\theta = U(\theta,\theta_0) \zeta_{\theta_0}
 + \int_{\theta_0}^\theta U(\theta,\phi) b(\zeta_\phi,\phi) \6\phi\;,
\end{equation} 
where $b(\zeta,\phi)$ is a nonlinear term. Using the estimates on the principal
solution (cf.~\eqref{sec4:behaviourU} in the proof of
Lemma~\ref{lem_rfn2}), one obtains that the deterministic orbit
starting in $(\hat x_4,(y'_4)^*,\hat z_4)\in\Sigma'_4$ hits $\Sigma_5$ at a
point $(x^*_5,\hat y_5,\hat z_5)$ satisfying 
\begin{equation}
 \label{eq:return02:4}
 \norm{(\hat y_5,\hat z_5)-(y^*_5,z^*_5)} 
 \leqs M\eps^{2/3}\abs{\hat x_4 - (x'_4)^*} + M\abs{\hat z_4 - (z'_4)^*}
\end{equation} 
for some constant $M>0$. Proposition~\ref{prop_fold_escape} yields that for
$H_1\leqs h_0\eps^{2/3}$, 
\begin{align}
\nonumber
&\P^{P_2}\Bigl(\Omega_1(h,h_1) \cap \Omega_2(H_1)^c \Bigr) \\
& \qquad {}\leqs{} C\abs{\log\eps}\biggl(\exp\biggl\{
-\frac{\kappa(H_1-M[\eps^{2/3}h+h_1])^2}{\sigma^2\eps +
(\sigma')^2\eps^{1/3}}\biggr\} 
+ \exp\biggl\{ -\frac{\kappa\eps}{\sigma^2+(\sigma')^2\eps}\biggr\}
\biggr)\;.
\label{eq:return02:5} 
\end{align}
We now choose $h_1 = H_1/(3M)$ and $h = \eps^{1/3} \wedge H_1/(3M\eps^{2/3})$. 
Distinguishing the cases $h = \eps^{1/3}$ and $h = H_1/(3M\eps^{2/3})$ when
using~\eqref{eq:return02:2}, this yields 
\begin{equation}
 \label{eq:return02:6}
 \P^{P_2}\bigl(\Omega_2(H_1)^c\bigr) \leqs 
 \frac{C}{\eps} \biggl(
 \exp\biggl\{ -\frac{\kappa H_1^2}{\sigma^2\eps\abs{\log\eps} +
(\sigma')^2}\biggr\}
 + \exp\biggl\{ -\frac{\kappa \eps}{\sigma^2 +
(\sigma')^2\eps^{-1/3}}\biggr\}\biggr)\;,
\end{equation} 
where $\kappa$ has been redefined. 

 \item {\em Step 3:\/} $\Sigma_5 \to \Sigma_6 \to \Sigma_1$. A similar argument
as above shows that two deterministic solutions starting at distance $H_1$ on
$\Sigma_5$ hit $\Sigma_1$ at a distance of order $H_1$. The result then follows
from Proposition~\ref{prop_tau} and~\eqref{eq:return02:6}, choosing
$H_1=\eps^{2/3}\wedge ch_1$ for a sufficiently small constant $c$. 
\qed
\end{itemize}
\renewcommand{\qed}{}
\end{proof}

This result is useful if $\sigma\ll\sqrt{\eps}$ and $\sigma'\ll\eps^{2/3}$. It
shows that stochastic sample paths are likely to hit $\Sigma_1$ at a distance of
order $\sigma+\sigma'$ in the fast directions from the deterministic solution,
and at a distance of order $\sigma\sqrt{\eps\abs{\log\eps}}+\sigma'$ in the
slow direction. 

\subsection{The local map}
\label{ssec_local_map} 

We know from the deterministic analysis (cf.~\eqref{eq:range_mu})
that the section $\Sigma_1$ can be 
subdivided into $k_\mu \simeq 1/(2\mu)$ sectors of rotation. An
orbit starting in the $k$\/th sector makes $2k+1$ half-turns before
hitting $\Sigma_2$. The width of the $k$\/th sector has order $\eps^{(1-\mu)/2}$, cf.~\cite{BronsKrupaWechselberger}.
The analysis of Section~\ref{ssec:fn-nbh-det} shows that the images of these
sectors on $\Sigma'_1$ have a size of order~$\mu\eps$. 

\begin{figure}
\centerline{\includegraphics*[clip=true,height=75mm]{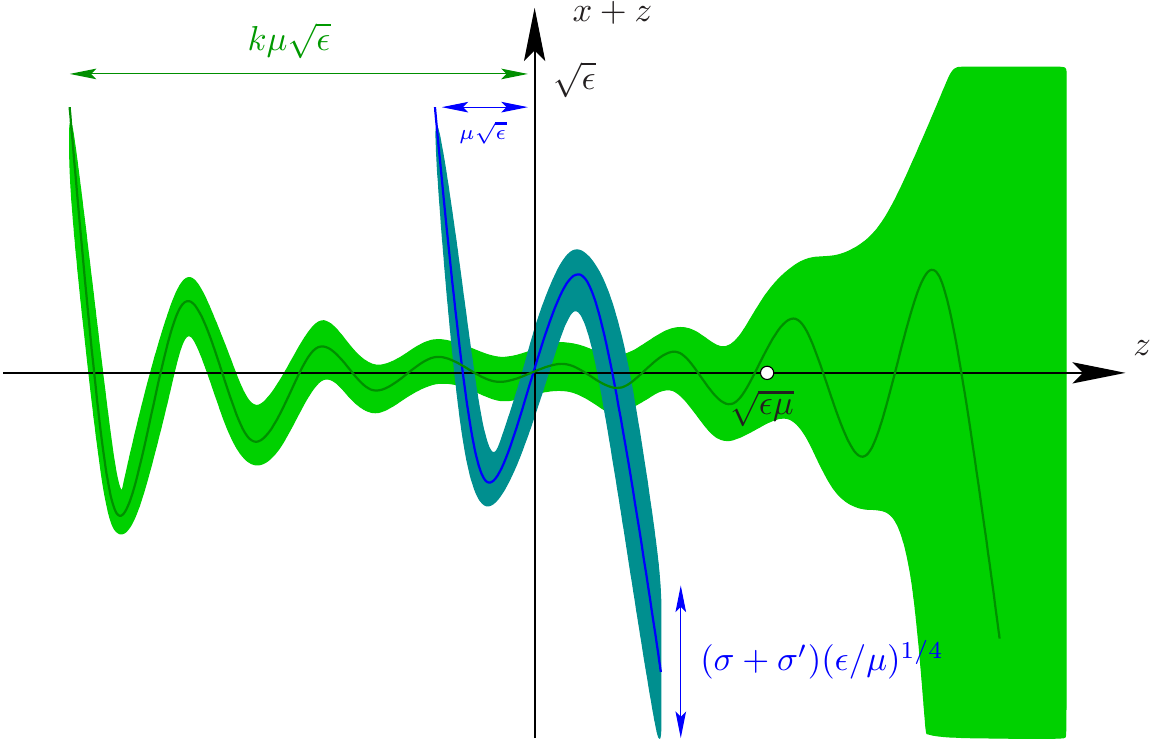}}
 \vspace{2mm}
\caption[]{Inner (blue) and outer (green) canards. 
The shaded sets indicate the extension of typical fluctuations.}
\label{fig_inner_outer}
\end{figure}

For the stochastic system, it will be relevant to distinguish between 
\begin{itemize}
\item 	{\bf inner sectors}, which are sectors with
$k\leqs\Order{1/\sqrt{\mu}}$; orbits starting in these sectors hit
$\Sigma_1''$ for $z\leqs\Order{\sqrt{\eps\mu}}$;
\item 	{\bf outer sectors}, which are sectors with $k>\Order{1/\sqrt{\mu}}$.
\end{itemize}

\begin{thm}[Local return map for inner sectors]
\label{thm_local_returns_inner}
Fix $P_0=(x^*_0,y^*_0,z^*_0)\in\Sigma_1$, and suppose that $P_0$ lies in an inner
sector. Assume the deterministic orbit
starting in $P_0$ hits $\Sigma_2$ for the first time in 
$P_2=(x^*_2,y^*_2,z^*_2)$. Further assume $\eps\leqs\mu$. 
Then there exist constants $h_0, \kappa, C >0$ such that for all $h\leqs
h_0\eps\sqrt{\mu}$ and
$h_2\leqs h_0 (\sqrt{h_1}\wedge\sqrt{\eps\mu}\,)$, the stochastic sample path
starting in $P_0$ hits
$\Sigma_2$ for the first time in a point $(x^*_2,y_2,z_2)$ satisfying 
\begin{multline}
\P^{P_0} \bigl\{ \abs{y_2 - y^*_2} > h_1 \text{ or } \abs{z_2 - z^*_2} > h_2
\bigr\} \\
\leqs 
\frac{C}{\eps} \biggl(
\exp\biggl\{ -\frac{\kappa h_1^2\sqrt{\mu}}{(\sigma^2 +
(\sigma')^2)\sqrt{\eps}}\biggr\}
+ \exp\biggl\{ -\frac{\kappa h_2^2}{(\sigma')^2}\biggr\} 
+ \exp\biggl\{ -\frac{\kappa\eps^{3/2}}{\sigma^2 +
(\sigma')^2}\biggr\} \biggr)\;.
\label{eq:return03} 
\end{multline}
\end{thm}

\begin{proof}
We introduce the events 
\begin{align}
\nonumber
\Omega_1(h_1,h_2) &= \bigl\{ \abs{y'_1 - (y'_1)^*} \leqs h_1, 
\abs{z'_1 - (z'_1)^*} \leqs h_2\bigr\}\;, \\
\Omega_2(H_1,H_2) &= \bigl\{ \abs{y''_1 - (y''_1)^*} \leqs H_1, 
\abs{z''_1 - (z''_1)^*} \leqs H_2\bigr\}\;, 
\label{eq:return04:1} 
\end{align}
where $(y'_1,z'_1)$ and $(y''_1,z''_1)$ denote the first-hitting points of the
stochastic path with $\Sigma'_1$ and $\Sigma''_1$, and the starred quantities
are the corresponding deterministic hitting points. 
Proposition~\ref{prop_fn_approach} with $h$ of order $\sqrt{\eps}$ implies 
\begin{equation}
 \label{eq:return04:2}
 \P^{P_0}(\Omega_1(h_1,h_2)^c) \leqs \frac{C}{\eps} \biggl(
 \exp\biggl\{ -\frac{\kappa h_1^2}{(\sigma^2 +
(\sigma')^2)\sqrt\eps} \biggr\} 
+ \exp\biggl\{ -\frac{\kappa h_2^2}{(\sigma')^2}\biggr\} 
+ \exp\biggl\{ -\frac{\kappa \eps^{3/2}}{\sigma^2+(\sigma')^2}\biggr\}
\biggr)
\end{equation} 
for some $C>0$.
Using this bound with $h_1$ of order $H_1$ and $h_2$ of order $H_1\wedge H_2$, 
together with Proposition~\ref{prop_fn_nbh} to estimate the probability 
of $\Omega_1\cap\Omega_2^c$ and the assumption $\eps\leqs\mu$ yield 
\begin{equation}
 \label{eq:return04:3}
 \P^{P_0}(\Omega_2(H_1,H_2)^c) \leqs \frac{C}{\eps} \biggl(
 \exp\biggl\{ -\frac{\kappa H_1^2\sqrt{\mu}}{(\sigma^2+(\sigma')^2)\sqrt{\eps}}
\biggr\}
+ \exp\biggl\{ -\frac{\kappa H_2^2}{(\sigma')^2}\biggr\}
+ \exp\biggl\{ -\frac{\kappa \eps^{3/2}}{\sigma^2+(\sigma')^2}\biggr\}
\biggr)\;.
\end{equation}
The result then follows from Proposition~\ref{prop_fn_escape}, taking $h_1=H_1$
and $h_2=H_2$.  
\end{proof}

This result is useful if $\sigma, \sigma'\ll(\eps\mu)^{3/4}$. It
shows that stochastic sample paths are likely to hit $\Sigma_2$ at a distance of
order $(\sigma+\sigma')(\eps/\mu)^{1/4}$ in the $y$-direction from the
deterministic solution, and at a distance of order $\sigma'$ in the
$z$-direction. Combining this with Theorem~\ref{thm_global_returns} on the
global return map, we conclude that for initial conditions $P_0\in\Sigma_1$,
starting in a inner sector, stochastic sample paths will return to $\Sigma_1$
in a neighbourhood of the deterministic solution, of width  
\begin{itemize}
 \item $\sigma+\sigma'$ in the fast $x$-direction,
 \item $\sigma\sqrt{\eps\abs{\log\eps}} + \sigma'$ in the $z$-direction.
\end{itemize}

\begin{rem}
\label{rem:kstar} 
The limitation $k\leqs\Order{1/\sqrt{\mu}}$ is due to the fact
that Proposition~\ref{prop_fn_nbh} is formulated for $\theta\leqs\sqrt{\mu}$. 
Using Remark~\ref{rem_extension_logsigma},
Theorem~\ref{thm_local_returns_inner} can be extended to sectors
$k=\sqrt{a\abs{\log(\sigma+\sigma')}/\mu}$, which results in fluctuations of $y$
of order $(\sigma+\sigma')^{1-ca}(\eps/\mu)^{1/4}$. This does not affect the
order of fluctuations in the $z$-direction as long as $a$ is small enough. 
\end{rem}

\begin{figure}
\centerline{\includegraphics*[clip=true,width=90mm]{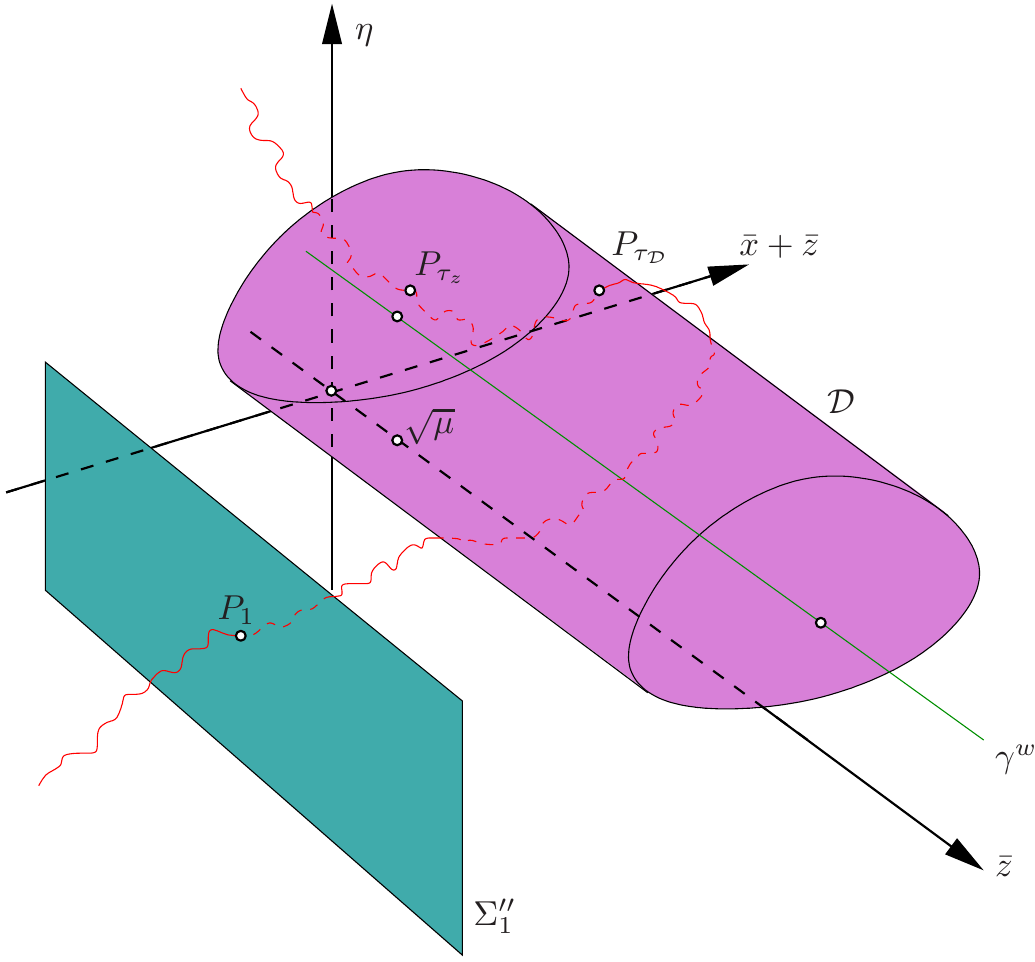}}
 \vspace{2mm}
\caption[]{Escape of sample paths starting in an outer sector from a
neighbourhood of the weak canard $\gamma^w$.}
\label{fig_escape_fn}
\end{figure}

Finally, we consider what happens to sample paths starting on $\Sigma_1$ in an
outer sector. 

\begin{thm}[Local return map for outer sectors]
\label{thm_local_returns_outer}
Fix $P_0=(x^*_0,y^*_0,z^*_0)\in\Sigma_1$, and assume $P_0$ lies in an outer
sector $k$, with $k\geqs k_0/\sqrt{\mu}$ for some $k_0>0$. If $k_0$ is large
enough, there exist constants $\kappa,C_0,C,\gamma>0$ such that the stochastic
sample path starting in $P_0$ hits $\Sigma_2$ for the first time in a point
$(x^*_2,y_2,z_2)$ satisfying 
\begin{align}
\label{eq:return05a} 
\P^{P_0} \bigl\{ z_2 \leqs \sqrt{\eps\mu} \bigr\}
&\leqs \frac{C}{\eps} \exp\biggl\{ -\frac{\kappa
\mu^{1/2}\eps^{3/2}}{\sigma^2 + (\sigma')^2}\biggr\} \;, \\
\P^{P_0} \bigl\{ z_2 \geqs z \bigr\}
&\leqs C_0\abs{\log\sigma}^\gamma 
\exp\biggl\{ -\frac{\kappa(z^2-\eps\mu)}{\eps\mu\abs{\log(\sigma +
\sigma')}}\biggr\}
\qquad \forall z\geqs\sqrt{\eps\mu}\;.
\label{eq:return05b} 
\end{align}
Furthermore, there exist a constant $h_0>0$ and an interval $I$ of size of order
$\eps$, independent of $k\geqs k_0/\sqrt{\mu}$, such that 
\begin{equation}
 \label{eq:return06}
 \P^{P_0} \bigl\{ \dist(y_2,I) > h_1 \bigr\}
\leqs \frac{C}{\eps} \biggl( \exp\biggl\{ -\frac{\kappa
h_1^2}{(\sigma^2 + (\sigma')^2)\sqrt{\eps}}\biggr\} 
+ \exp\biggl\{ -\frac{\kappa
\eps^{3/2}}{\sigma^2 + (\sigma')^2\eps}\biggr\} \biggr) 
\end{equation} 
holds for all $h_1\leqs h_0\eps$.
\end{thm}

\begin{proof}
We shall work in the zoomed-in coordinates $(\bar x,\bar y, \bar z)$,
cf.~\eqref{eq:fn_nbh_det:01}. Fix a $K_0\in(0,1)$ and introduce a neighbourhood
$\cD$ of the weak canard $\gamma^w$ given by 
\begin{equation}
 \label{eq:return07:1}
 \cD = \bigl\{ (\bar x, \bar y, \bar z) \colon  
 K \geqs K_0, \bar z \geqs \sqrt{\mu} \bigr\} = \cD_0 \times
[\sqrt{\mu},\infty)\;,
\end{equation} 
where $K$ is the first integral introduced in~\eqref{eq:fn_nbh_det:107} (recall
that $u_1$ and $u_2$ measure the deviation of $(\bar x,\bar y)$ from the weak
canard) -- see~\figref{fig_escape_fn}.  
The proof proceeds in four main steps:
\begin{itemize}
 \item {\em Step 1: Entering $\cD$.\/} From Theorem~4.4 in~\cite{BGK12}
describing the spacing of canards, we know that the deterministic solution
starting in $P_0$ hits $\cD$ in a point $P^*_0$ at a distance of order
$\e^{-c(2k+1)^2\mu} < \e^{-4ck_0^2}$ from the weak canard. Taking $k_0$
sufficiently large, we may assume that $P^*_0$ is bounded away from the
boundary $\partial\cD_0$. Combining, as in the previous theorem, 
Proposition~\ref{prop_fn_approach} and Proposition~\ref{prop_fn_nbh}, we obtain
that the stochastic sample path first hits $\cD$ at a point $P_{\tau_z}$ such
that 
\begin{equation}
 \label{eq:return07:2}
 \P^{P_0} \bigl\{ (\bar x_{\tau_z}, \bar y_{\tau_z}) \notin
\cD_0 \bigr\}
\leqs \frac{C}{\eps} \exp\biggl\{ -\frac{\kappa
\mu^{1/2}\eps^{3/2}}{\sigma^2 + (\sigma')^2}\biggr\}\;.
\end{equation} 
This proves in particular~\eqref{eq:return05a}. 

 \item {\em Step 2: Leaving $\cD$.\/} Theorem~6.4 in~\cite{BGK12} estimates the
probability of sample paths not leaving a neighbourhood of the weak canard.
Because we worked in polar coordinates, the result only applied to a small
neighbourhood of size proportional to $\sqrt{\bar z}$. However by using the
coordinate $K$ instead of the distance to $\gamma^w$, the same proof applies to
the exit from $\cD$. It suffices to realize that the nonlinear
drift term~$\beta$ in Equation~(D.33) of~\cite{BGK12} is replaced by a term of
order $\bar z$ as a consequence of Proposition~\ref{prop_fn_averaging}. We thus
conclude that the sample path leaves $\cD$ at a point $P_{\tau_\cD}$ whose
$\bar z$-coordinate satisfies 
\begin{equation}
 \label{eq:return07:3}
 \P^{P_0} \bigl\{ \bar z_{\tau_\cD} \geqs \bar z \bigr\}
\leqs C_0\abs{\log\bar\sigma}^\gamma 
\exp\biggl\{ -\frac{\kappa(\bar z^2-\mu)}{\mu\log\abs{\bar\sigma +
\bar\sigma'}}\biggr\}
\qquad \forall \bar z\geqs\sqrt{\mu}
\end{equation} 
for some constants $C_0, \gamma>0$. 

 \item {\em Step 3: Transition from $\cD$ to $\Sigma_1''$.\/}
Since $P_{\tau_\cD}$ is at distance of order $1$ from the weak canard, we know
that the deterministic solutions starting in $P_{\tau_\cD}$ will take a time of
order $\mu$ to reach $\Sigma_1''$, in a point that we will denote
$P_1^*=(\bar x^*_1,\bar y^*_1, \bar z^*_1)$. Let $P_1=(\bar x^*_1,\bar y_1,
\bar z_1)$ denote the point where the stochastic sample path first hits
$\Sigma_1''$. Starting from System~\eqref{eq:fn_nbh_stoch:02} and applying the
usual procedure, we obtain the estimate 
\begin{align}
\nonumber
  &\P^{P_{\tau_\cD}} \bigl\{ \abs{\bar y_1-\bar y^*_1} > \bar h_1 
  \text{ or } \abs{\bar z_1-\bar z^*_1} > \bar h_2 \bigr\} \\
&\qquad\qquad\leqs {C} \biggl(
\exp\biggl\{ -\frac{\kappa[\bar h_1 - M(\bar h_1^2+\bar h_2^2)]^2}
{\bar\sigma^2+(\bar\sigma')^2}\biggr\} + 
\exp\biggl\{ -\frac{\kappa \bar h_2^2}{(\bar \sigma')^2\eps}\biggr\}
\biggr)\;.
 \label{eq:return07:4} 
\end{align} 
Note that this implies fluctuations of size $\bar\sigma+\bar\sigma'$ in the
$\bar y$-direction, and of size $\bar\sigma'\sqrt{\eps}$ in the $\bar
z$-direction. Going back to original variables, this entails fluctuations of
size $(\sigma+\sigma')\eps^{1/4}$ in the $y$-direction, and $\sigma'\eps^{1/4}$
in the $z$-direction. 

We also have to take into account the fact that we do not know the $(\bar
x,\bar y)$-coordinates of $P_{\tau_\cD}$. In fact all exit points on
$\partial\cD_0$ might have a comparable probability. Hence the
coordinate $\bar y^*_1$ can vary in an interval $I_1$, which is the image of
$\partial\cD_0$ under the deterministic flow. It follows
from~\eqref{eq:fn_nbh_det:101d} that $I_1$ has a size of order $\bar z^*_1$ in
$\bar
y$-coordinates. 

 \item {\em Step 4: Transition from $\Sigma_1''$ to $\Sigma_2$.\/}
If $P_1$ satisfied $y_1 \leqs x_1^2 - \eps(\frac{1+\mu}{2}+c_0)$
for some $c_0>0$, or equivalently  $\eta_1\leqs -c_0$, 
we could directly apply Proposition~\ref{prop_fn_escape} to
estimate the fluctuations during the last transition step, which would remain of
the same order as in Step 3. 

The estimate~\eqref{eq:fn_nbh_det:101d} shows that $P_1$ is too close to the
repelling slow manifold $C^r_\eps$ to apply Proposition~\ref{prop_fn_escape}
directly. However, using~\eqref{eq:fn_nbh_stoch:02}, we obtain that the variable
$\eta$ measuring the distance to $C^r_\eps$ satisfies an equation of the form 
\begin{equation}
 \label{eq:return07:5}
 \6\eta_\theta = \frac{2}{\mu} 
 \Bigl[ -2\bar x_\theta\eta_\theta - \bar z_\theta + 
 \cO\bigl(\sqrt{\eps}(1+\abs{\bar x_\theta})\bigr)\Bigr] \6\theta 
 + \sqrt{\frac{2}{\mu}} \bigl[ \bar\sigma \widetilde G_1 + \bar\sigma'
\widetilde G_2\bigr] \6W_\theta\;.
\end{equation} 
We have used that we may assume $\sigma, \sigma' \ll \sqrt{\eps}$ to simplify
the error term (for larger noise intensities, the main results of the theorem
become meaningless). Using the same approach as in~\cite[Section~3.2]{BGbook}
or~\cite[Section~D]{BGK12}, one can show that $\eta_\theta$ is likely to leave
$[-c_0,c_0]$ in a time $\theta$ of order
$\mu\sqrt{\abs{\log(\bar\sigma+\bar\sigma')}}$. During this time interval, $\bar
x_\theta$ decreases by an amount of order
$\sqrt{\abs{\log(\bar\sigma+\bar\sigma')}}$. Either this exit takes place in the
direction of negative $\eta$, and we can apply again
Proposition~\ref{prop_fn_escape}. Or it takes place in the direction of positive
$\eta$, and the sample path makes one more excursion towards $C_\eps^{a+}$
(backward canard). In this case we have to use one more time the analysis of Step 3
before applying again Proposition~\ref{prop_fn_escape}. 
Finally, one can check that the deterministic flow maps the set of 
points where paths escape $\{-c_0\leqs\eta\leqs c_0\}$ to points in $\Sigma_2$
with a $y$-coordinate in an interval $I$ of size $\Order{\eps}$. 
\qed
\end{itemize}
\renewcommand{\qed}{}
\end{proof}

The important point of Theorem~\ref{thm_local_returns_outer} is that the bounds
on the distribution of $(y_2,z_2)$ are \emph{independent}\/ of the starting
sector number $k$, as soon as $k>k_0/\sqrt{\mu}$. Thus we observe a saturation
effect, in the sense that the stochastic Poincar\'e map becomes independent of
the initial condition -- see \figref{fig_Koper_poincare}. Combining the local result
with Theorem~\ref{thm_global_returns}, we see in particular that the size of
fluctuations in the $z$-direction is at most of order 
\begin{equation}
 \label{eq:return08}
 \sqrt{\eps\mu\abs{\log(\sigma+\sigma')}}
 + \sigma\sqrt{\eps\abs{\log\eps}}
 + \sigma'\;.
\end{equation} 
Disregarding logarithms, we observe that unless $\mu<\sigma^2\wedge((\sigma')^{2}/\eps)$, the first term will be the dominating one.
We conclude that in this regime, the noise-induced fluctuations in the
$z$-direction are at most of order $\sqrt{\eps\mu\abs{\log(\sigma+\sigma')}}$. 
However this bound is certainly not sharp, since it uses $z=\sqrt{\eps\mu}$ as
lower bound of typical exits from a neighbourhood of the weak canard, which may
underestimate typical exit times if the noise is weak. 


\subsection{Consequences for the MMO patterns}
\label{ssec_MMO_patterns}

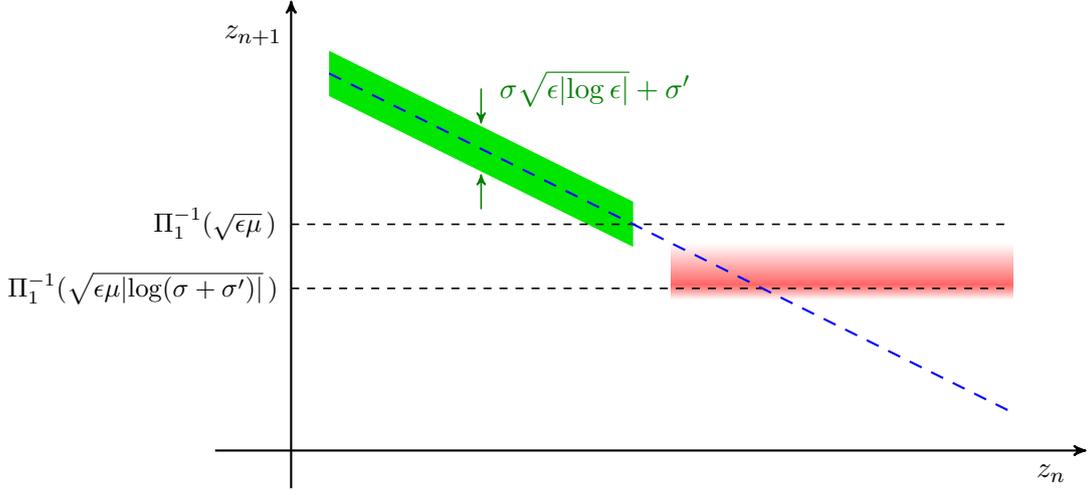
\begin{figure}
\begin{center}
\begin{tikzpicture}[>=stealth',shorten >=1pt,auto,main
node/.style={circle,minimum size=0.08cm,inner sep=0pt,fill=white,draw}]

\draw[->,thick] (-1,0) -- (10.5,0);
\draw[->,thick] (0,-0.5) -- (0,6);

\path[fill=green!90!black] (0.5,5.3) -- (4.5,3.3) -- (4.5,2.7) -- (0.5,4.7);
\shade[top color=white, bottom color=red!60]
     (5,2.75) -- (9.5,2.75) -- (9.5,2.2) -- (5,2.2);
\shade[top color=red!60, bottom color=white]
     (5,2.2) -- (9.5,2.2) -- (9.5,2) -- (5,2);

\draw[dashed,semithick] (0,2.15) -- (9.5,2.15);
\draw[dashed,semithick] (0,3) -- (9.5,3);

\draw[blue,dashed,thick,dash pattern=on 5pt off 4pt] (0.5,5) -- (9.5,0.5); 

\draw[->,semithick,green!50!black] (2.5,4.8) -- (2.5,4.3);
\draw[->,semithick,green!50!black] (2.5,3.2) -- (2.5,3.7);

\node[green!50!black] at (4,4.8) {$\sigma\sqrt{\eps\abs{\log\eps}}+\sigma'$};
\node[] at (10,-0.3) {$z_n$};
\node[] at (-0.5,5.5) {$z_{n+1}$};
\node[] at (-1,3) {\small $\Pi_1^{-1}(\sqrt{\eps\mu}\,)$};
\node[] at (-1.95,2.15)
{\small $\Pi_1^{-1}(\sqrt{\eps\mu\abs{\log(\sigma+\sigma')}}\,)$};
\end{tikzpicture}
\end{center}
\vspace{-5mm}
\caption[]{Sketch of the Poincar\'e map $z_n\mapsto z_{n+1}$ on the section
$\Sigma_1$. The dashed blue line indicates the position of the deterministic
map, disregarding the canards. $\Pi_1$ denotes the $z$-component of the map
$\Sigma_1\to\Sigma_2$. The saturation effect sets in when $\Pi_1(z_n)$ reaches
$\sqrt{\eps\mu}$.}
\label{fig_Poincare_map_sketch}
\end{figure}

As mentioned in the introduction, which MMO patterns will be observed depends
on the following factors:
\begin{enum}
\item  	in which rotation sector, if any, the Poincar\'e map
admits a fixed point;
\item 	how many SAOs the stochastic system performs when starting at
that fixed point;
\item 	whether or not stochastic fluctuations mask the smallest oscillations.
\end{enum}

\figref{fig_Poincare_map_sketch} gives a schematic view of the situation.
Assume that the deterministic Poincar\'e map on the section $\Sigma_1$ is
such that $z_{n+1}$ is a decreasing function of $z_n$ on average (that is,
disregarding the dips caused by canards), as indicated by the blue dashed line.
Here we assume that the $z$-axis is oriented in such a way that larger values of
$z_n$ lead to more SAOs. When $z_n$ belongs to an inner sector,
Theorems~\ref{thm_global_returns} and~\ref{thm_local_returns_inner} apply,
and show that stochastic sample paths are likely to return to $\Sigma_1$ in the
green shaded set, that is, at a distance of order
$\sigma\sqrt{\eps\abs{\log\eps}}+\sigma'$ from the deterministic orbit.
When $z_n$ belongs to an outer sector, however, the saturation effect sets in,
meaning that stochastic sample paths tend to return to $\Sigma_1$ in the red
shaded set, at a coordinate $z_{n+1}$ that no longer depends on the starting
point.

In order to quantify Point 1 listed above, we may consider that the
average map $z_n\mapsto z_{n+1}$ induces a map $k_n\mapsto k_{n+1}$ between
sectors of rotation given approximately by 
\[
 \Pi(k) = 
 \begin{cases}
  \Pi^{\det}(k) & \text{if $k < k^*$\;} \\
  \Pi^{\det}(k^*) & \text{if $k\geqs k^*$\;,}
 \end{cases}
\]
where $\Pi^{\det}$ is the deterministic map, and $k^*$ is the number of the
sector in which the saturation effect sets in. According to the discussion in
the previous section (see in particular Remark~\ref{rem:kstar}), we have
\[
\Order{1/\sqrt{\mu}\,} \leqs 
 k^*(\mu,\sigma,\sigma') \leqs
\Order{\sqrt{\abs{\log(\sigma+\sigma')}/\mu}\,}\;.
\]
Assume that $\Pi^{\det}$ is decreasing, and admits a unique fixed point
$k^{\det}$. Then the map $\Pi$ admits a fixed point in the sector  
\[
 \max\{k^{\det},\Pi^{\det}(k^*)\}\;.
\]

Consider now Point 2, i.e., determine the number $n^{\stoch}$ of SAOs
associated with the fixed point. Recall that in the deterministic case, the
system performs $n^{\det}=k^{\det}$ SAOs (this being the rounded value of the
$2k^{\det}+1$ half-turns). If $k^{\det} < k^*$, we have
$n^{\stoch}=n^{\det}=k^{\det}$ . Otherwise, the number of SAOs will be given by 
\[
 n^{\stoch} = \frac12\bigl( \Pi^{\det}(k^*) + k^* \bigr)
\]
because the system starts in the sector $\Pi^{\det}(k^*)$, performs
$\Pi^{\det}(k^*)$ half-turns for $z<0$, and only $k^*$ half-turns for $z>0$
before escaping. Using the fact that $\Pi^{\det}(k^{\det})=k^{\det}$, it is
easy to see that 
\begin{equation}
 \label{eq:MMO_patterns1}
 n^{\stoch} > n^{\det}
 \quad \Leftrightarrow\quad
 \Pi^{\det}(k^*) + k^* > \Pi^{\det}(k^{\det})+ k^{\det}\;.
\end{equation}
In other words, the number of SAOs may increase in the presence of noise
provided the map $k\mapsto\Pi^{\det}(k) + k$ is decreasing.

Consider finally Point 3, namely whether the amplitude of SAOs may become so
small as to be indistinguishable from random
fluctuations due to the noise. In fact, this phenomenon has been analysed
in~\cite[Section~6.3]{BGK12}. The results obtained there show that
for orbits starting in the sector $k$, fluctuations start dominating
the small oscillations near $z=0$ if 
\begin{equation}
 \label{eq:MMO_patterns2}
 k^2 \mu \geqs
\log\biggl(\frac{\mu^{1/4}\eps^{3/4}}{\sigma}\biggr)\;,
\end{equation} 
where we have already incorporated the zoom-out transformation,
cf~\eqref{eq:fn_nbh_stoch:03}. First note that owing to our assumption
$\sigma\ll(\eps\mu)^{3/4}$, the right-hand side of~\eqref{eq:MMO_patterns2} will
always be larger than $1$. In the saturated regime, the left-hand side is
bounded below by $(k^*)^2\mu$, and thus at least of order $1$. However,
$(k^*)^2\mu$ can be as large as order $\abs{\log(\sigma+\sigma')} =
\log(1/(\sigma+\sigma'))$. Thus whether or not the SAOs are masked by
fluctutations depends crucially on where $k$ lies in the window of possible
values: the SAOs will still be visible as long as $k$ is sufficiently close to
$1/\sqrt{\mu}$.

It is important to note that one key component in the previous analysis is the
\lq\lq dynamical skeleton\rq\rq\ provided by the global return map for the
deterministic system. For example, if the global deterministic return map
already generates a highly complicated multi-stable scenario with several
possible MMO patterns, then the noise-induced effects can become even more
complicated.


\section{An example -- The Koper model}
\label{sec:Koper}

In order to illustrate some of our results numerically, we consider the example of the
Koper model~\cite{Koper}. Its deterministic version~\cite{KuehnRetMaps} 
is given by
\begin{align}
\nonumber
 \eps_1\dot{x}&{}= y-x^3+3x\;,\\  \label{eq:Koper}
 \dot{y} &{}= kx-2(y+\lambda)+z\;,\\ \nonumber
 \dot{z}&{}= \eps_2(\lambda+y-z)\;,
\end{align}
with parameters $k$, $\lambda$, $\eps_1$, $\eps_2$. Note that there is a
symmetry
\benn
(x,y,z,\lambda,k)\mapsto (-x,-y-z,-\lambda,k)
\eenn
so that we can restrict the parameter space. We shall assume that
$0<\eps_1=:\eps\ll1$ 
and $\eps_2=1$ so that~\eqref{eq:Koper} has the structure~\eqref{eq:global_main}
and Assumption~(A0)
obviously holds. For a detailed bifurcation analysis we refer
to~\cite{KuehnMMO,KuehnRetMaps}.

\begin{figure}[htbp]
	\centering
		\includegraphics[width=0.95\textwidth]{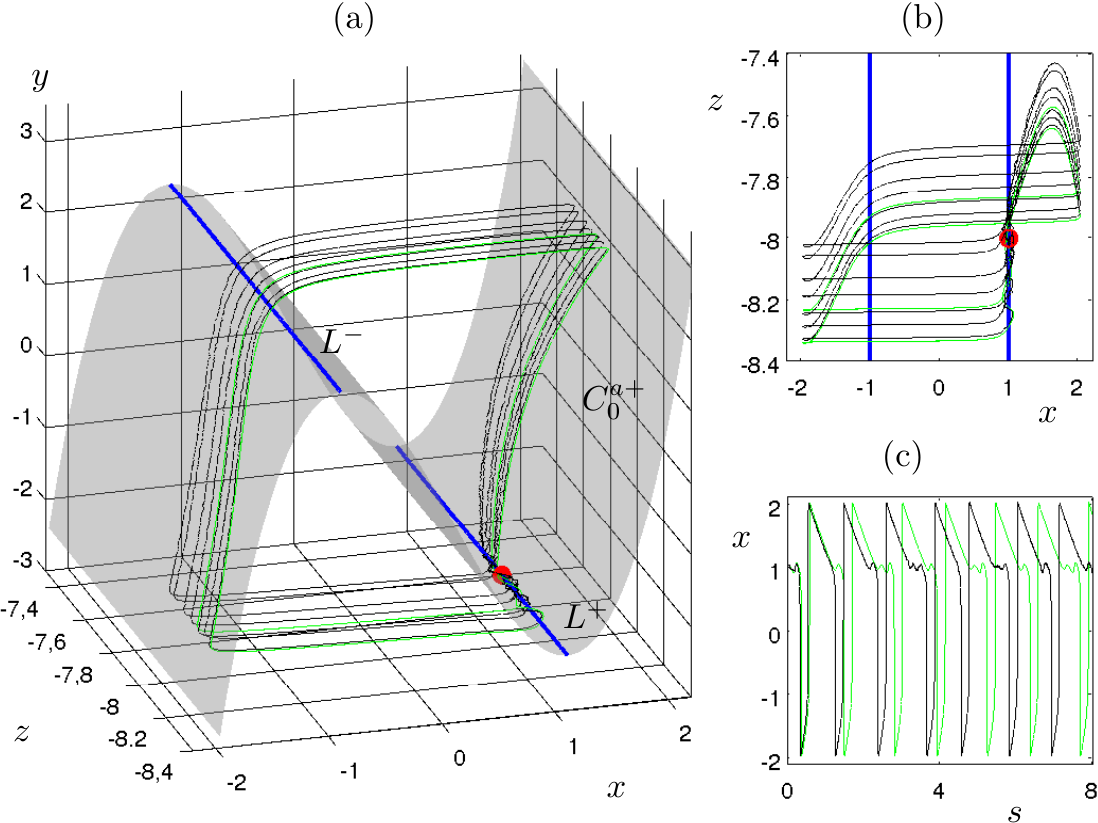}
	\caption{\label{Koper_fig01}Basic structure of the dynamics for 
	the Koper model \eqref{eq:Koper_noisy}. (a) Phase space with deterministic 
	critical manifold $C_0$ (grey) and the two fold lines $L^\pm$ (blue). Two 
	orbits are also shown, one for the deterministic system 
	($\sigma=0=\sigma'$, green) and one  for the stochastic system 
	($\sigma=0.01=\sigma'$, black). For both, the parameter values are $\eps=0.01$, 
	$k=-10$, $\lambda=-7$, and $M$ is given by~\eqref{Mmatrix}. (b) Projection 
	of the full system onto the $(x,z)$-plane. (c) Time series for the two 
	orbits.}	
\end{figure}

Of course, if $0<\eps_2\ll 1$ one may still simulate the three-scale system 
numerically, and it is even known via explicit asymptotic analysis which MMO 
patterns one expects to observe in certain classes of three-scale 
systems~\cite{KrupaPopovicKopell}. The first variant of the Koper model was 
a planar system due to Boissonade and De Kepper~\cite{BoissonadeDeKepper}. 
Koper~\cite{Koper} introduced the third variable and studied MMOs via numerical 
continuation. In fact, the system~\eqref{eq:Koper} has been suggested
independently by various other research groups as a standard model for
MMOs~\cite{BronsKrupaWechselberger,GuckenheimerSH,KawczynskiStrizhak}. 
Therefore it certainly provides an excellent test case.

\begin{figure}[htbp]
	\centering
		\includegraphics[width=1\textwidth]{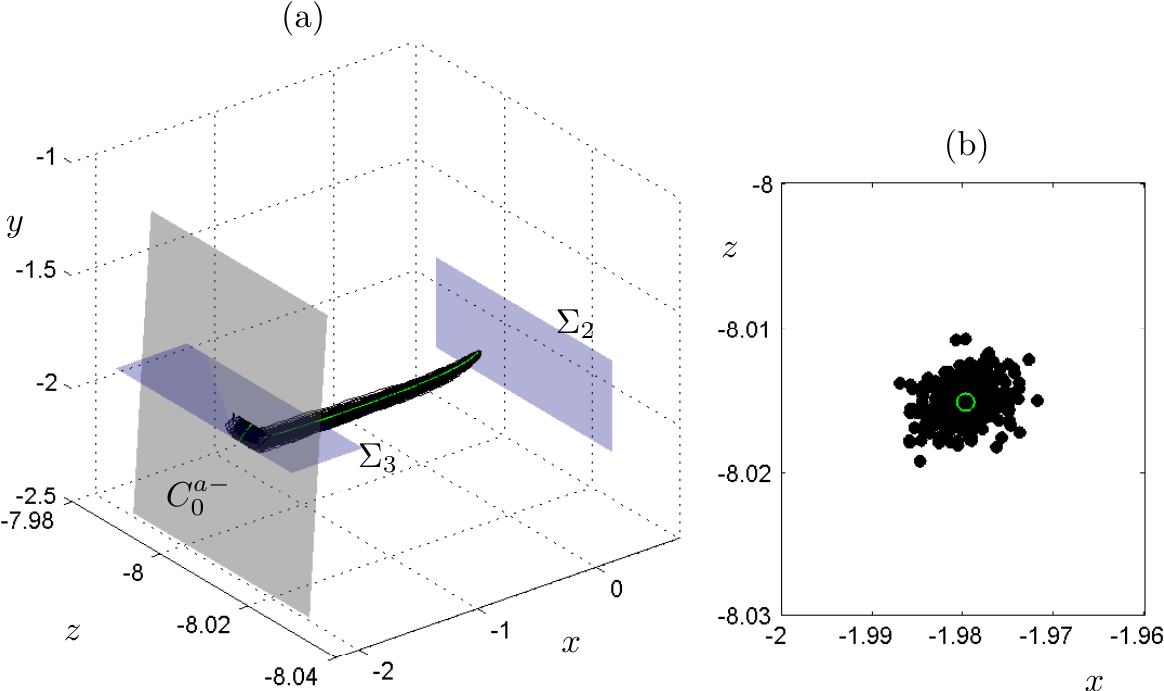}
	\caption{\label{Koper_fig02}Computation of the transition map 
	$\Sigma_2\ra \Sigma_3$ where the sections (blue) have been defined 
	by the conditions $x=0.5$ and $y=-1.8$ respectively. All paths have
	been started at $(x,y,z)=(0.5,-2.1,-8)$. There is one deterministic path 
	(green) and 300 different realizations (black) for the stochastic case 
	with $M$ given by \eqref{Mmatrix} and $\sigma=0.1=\sigma'$. The relevant 
	part of the critical manifold $C_0^{a+}$ (grey) is shown as well. (a) 
	Phase space. (b) View of the landing points on the section $\Sigma_3$
	where the landing point of the deterministic path is in the center of
	the green circle.}	
\end{figure}

The critical manifold of~\eqref{eq:Koper} is given by
$C_0=\{(x,y,z)\in\R^3\colon y=c(x)\}$ with $c(x):= x^3-3x$, and the two fold curves are 
$L^\pm=\{(x,y,z)\in\R^3\colon x=\pm 1,y=\mp 2\}$. This yields a decomposition 
\benn
C_0=C^{a-}_0\cup L^- \cup C^r_0 \cup L^+ \cup C_0^{a+}\;,
\eenn
where $C_0^{a-}=C_0\cap \{x<-1\}$, $C^r=C_0\cap \{-1<x<1\}$ and
$C_0^{a+}=C_0\cap \{1<x\}$ are normally hyperbolic. It is easy to verify
that Assumption~(A1) is satisfied. 

The desingularized slow subsystem is of the form
\begin{align}
\label{eq:koper_sf2}
\nonumber
\dot{x}&{}=kx-2(c(x)+\lambda)+z\;,\\
\dot{z}&{}= (3x^2-3)(\lambda+c(x)-z)\;.
\end{align}
Note that in~\eqref{eq:koper_sf2} the direction of time is reversed on $C_0^r$. The only
folded 
equilibria are $(x,z)=(1,2\lambda-4-k)\in L^+$ and $(x,z)=(-1,2\lambda+4+k)\in
L^-$. From the linearization of the slow subsystem 
\be
\label{eq:Apm}
A^\pm=\begin{pmatrix}k & 1 \\ 6(2+k\mp\lambda) & 0\end{pmatrix}
\ee
at the folded singularities one may determine the parameter values for which we have a folded node on $L^+$. It turns out that there exist 
parameter regimes where this is the case, and the only passages of 
deterministically stable MMO orbits near $L^-$ are via nondegenerate 
folds~\cite{KuehnMMO,KuehnRetMaps}. Furthermore, the fast flow is transverse 
to the relevant drop curves in such a regime~\cite{KuehnMMO,KuehnRetMaps}. From now on, we
shall restrict our attention to this parameter regime so that Assumptions~(A2)--(A4) are satisfied in 
a suitable compact absorbing set in phase space.

Since~\eqref{eq:Koper} is a phenomenological model, it is not immediate how 
to derive noise terms so we will just choose correlated additive noise as a first 
benchmark, setting
\begin{align}
\label{eq:Koper_noisy} \nonumber
 \6 x_s &{}= \frac1\eps (y_{s}-x_{s}^3+3x_{s})\6 s+\frac{\sigma}{\sqrt{\eps}} F\6W_s\;,\\
 \6y_s&{}= (kx_{s}-2(y_{s}+\lambda)+z_{s})\6s + \sigma' G_1\6W_s\;,\\[6pt]
 \6z_s&{}= \eps_2(\lambda+y_{s}-z_{s})\6s + \sigma' G_2\6W_s\;,\nonumber
\end{align}
where the Brownian motion $(W_s)_{s}$ is assumed to be three-dimensional, and $\frac{\sigma}{\sqrt{\eps}} F$,
$\sigma' G_1$, $\sigma' G_2$ may be viewed as rows of a constant matrix $M\in\R^{3\times
3}$. \figref{Koper_fig01} shows the basic geometry of the Koper model 
including two orbits computed for $(\eps,k,\lambda)=(0.01,-10,-7)$. One of these orbits
 is deterministic 
($\sigma=0=\sigma'$) and the other one shows a realization of 
a stochastic sample path computed for $\sigma=0.01=\sigma'$ and 
\be
\label{Mmatrix}
M=\begin{pmatrix}F\\ G_1\\ G_2\end{pmatrix}
=
\begin{pmatrix}1.0 & 0.5 & 0.2\\ 0.5 & 1.0 & 0.3\\ 0.2 & 0.3 &
1.0\end{pmatrix}\;.
\ee
Note that the deterministic orbit exhibits an MMO of type $1^11^2$ while the 
stochastic sample path shows combinations of patterns of the form $1^0$, $1^1$ and
$1^2$. Since we proved results about separate phases of the flow we investigate the
estimates for each phase as summarized in Table~\ref{table_deviations}.

\begin{figure}[htbp]
	\centering
		\includegraphics[width=0.8\textwidth]{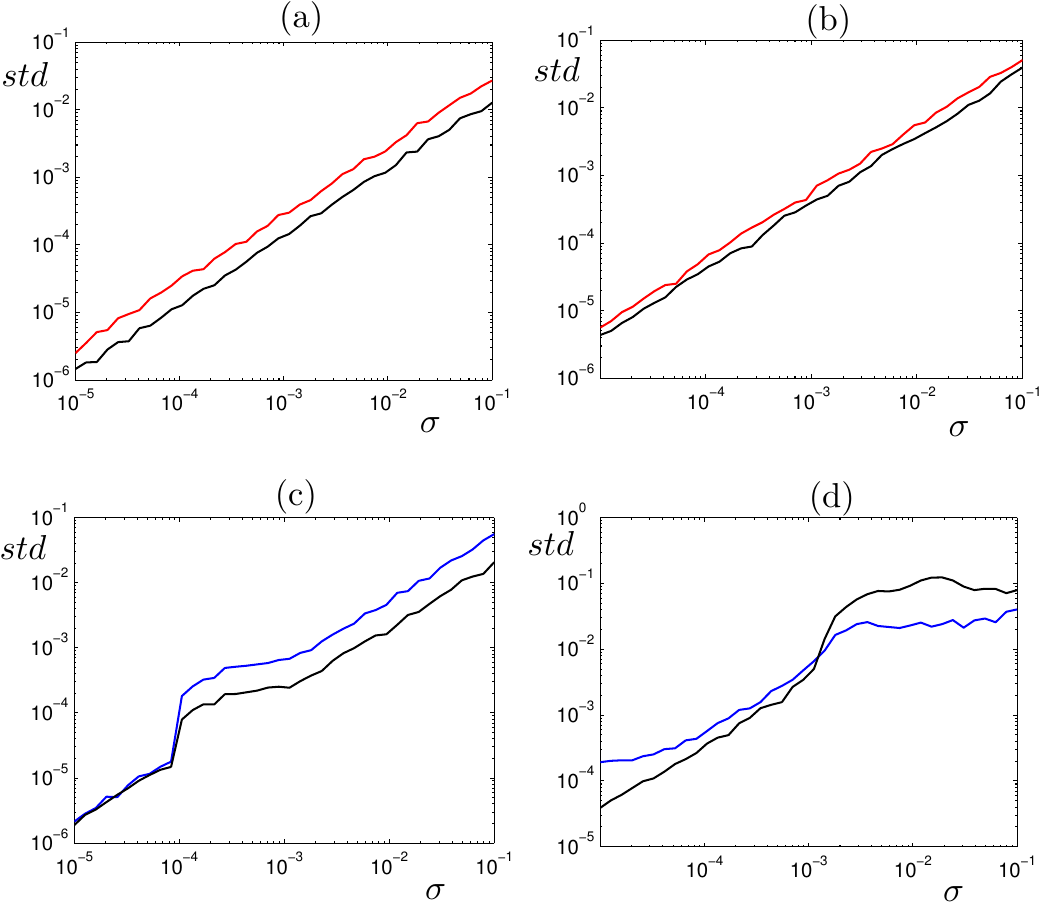}
	\caption{\label{Koper_fig03}%
	Dependence of the standard deviation of the distance between the stochastic and 
	the deterministic transition maps  on the  noise intensity for an attracting deterministic periodic MMO during various phases of the flow.
	We have always fixed the matrix $M$ by~\eqref{Mmatrix}, 
	$\eps=0.01$ and $\sigma=\sigma'$. The results are shown on a $\log$--$\log$ 
	scale with two directions out of $x$ (red), $y$ (blue) and $z$ (black).
	The standard deviation {\it std\/} has been computed from $100$ realizations of sample paths 
	and the domain for the noise level was subdivided into $40$ points. (a) 
	Transition map from $\Sigma_2\ra \Sigma_3$ with $\Sigma_2=\{x=0.5\}$, 
	$\Sigma_3=\{y=-1.8\}$ and $(x_0,y_0,z_0)=(0.5,-2.1,-8)$. (b) Transition 
	map from $\Sigma_3\ra\Sigma_4$ with $\Sigma_4=\{y=1.8\}$ and 
	$(x_0,y_0,z_0)=(-2,-1.8,-8)$. (c) Transition map from 
	$\Sigma_4\ra \Sigma_5$ with $\Sigma_5=\{x=-0.5\}$ and 
	$(x_0,y_0,z_0)=(-1.3,1.8,-7.8)$. (d) Transition map from 
	$\Sigma_1\ra \Sigma_2$ with $\Sigma_1=\{y=-1.8\}$ and 
	$(x_0,y_0,z_0)=(1.3,-1.8,-7.7)$.}	
\end{figure}

\figref{Koper_fig02} illustrates the map $\Sigma_2\ra \Sigma_3$ which
describes the fast flow towards the critical manifold $C_0^{a-}$. Several 
stochastic sample paths are compared with the deterministic solution. 
In Section~\ref{ssec:fast} we derived the typical spreading of stochastic sample paths around their deterministic counterpart. It was shown that the typical spreading has an upper bound $\cO(\sigma+\sigma')$ in the $x$-coordinate and $\cO(\sigma'+\sigma \sqrt\eps)$ in the $z$-coordinate. Since the typical spreading can be understood as standard deviation, cf.~\cite[Prop.~3.1.13]{BGbook}, 
\figref{Koper_fig02}(b) confirms that the theoretical results indeed provide upper bounds (note the scaling on the axes and that 
$\sigma=0.01=\sigma'$). 

To investigate the scaling results further, we  computed sample paths going from $\Sigma_2$ to $\Sigma_3$ numerically for a much wider range of noise values as shown in~\figref{Koper_fig03}(a). For the hitting point on $\Sigma_3$ we plotted the 
standard deviation of the hitting point's distance to its deterministic counterpart for the $(x,z)$-coordinates 
in a $\log$--$\log$ plot for different noise levels with $\sigma=\sigma'$. 
The slope of $1$ for both coordinates  in \figref{Koper_fig03}(a)--(b) 
is expected from the upper bounds in Table~\ref{table_deviations}. However, the
overall spreading is smaller than expected since we have started the orbits in 
the vicinity of an attracting deterministic periodic orbit,
{cf.}~\figref{Koper_fig01}. Therefore, contraction transverse to the periodic 
orbit shrinks the stochastic neighbourhood more than the general upper-bound 
estimates predict. Similarly, we may also study the remaining phases of the 
flow which are analyzed in \figref{Koper_fig03}(b)--(d). 
We observe not only the correct asymptotic decrease in size of the stochastic neighbourhood as $\sigma\ra 0$, but also a larger
spreading of sample paths near the folded node, see \figref{Koper_fig03}(d).
This is related to the mechanism that sample 
paths may jump only with high probability during certain parts of the SAOs 
after the folded node; this effect has already been discussed in detail
in~\cite{BGK12} with associated numerics in~\cite[Section~7]{BGK12} so 
we shall not detail it here.

\begin{figure}[htbp]
	\centering
		\includegraphics[width=0.8\textwidth]{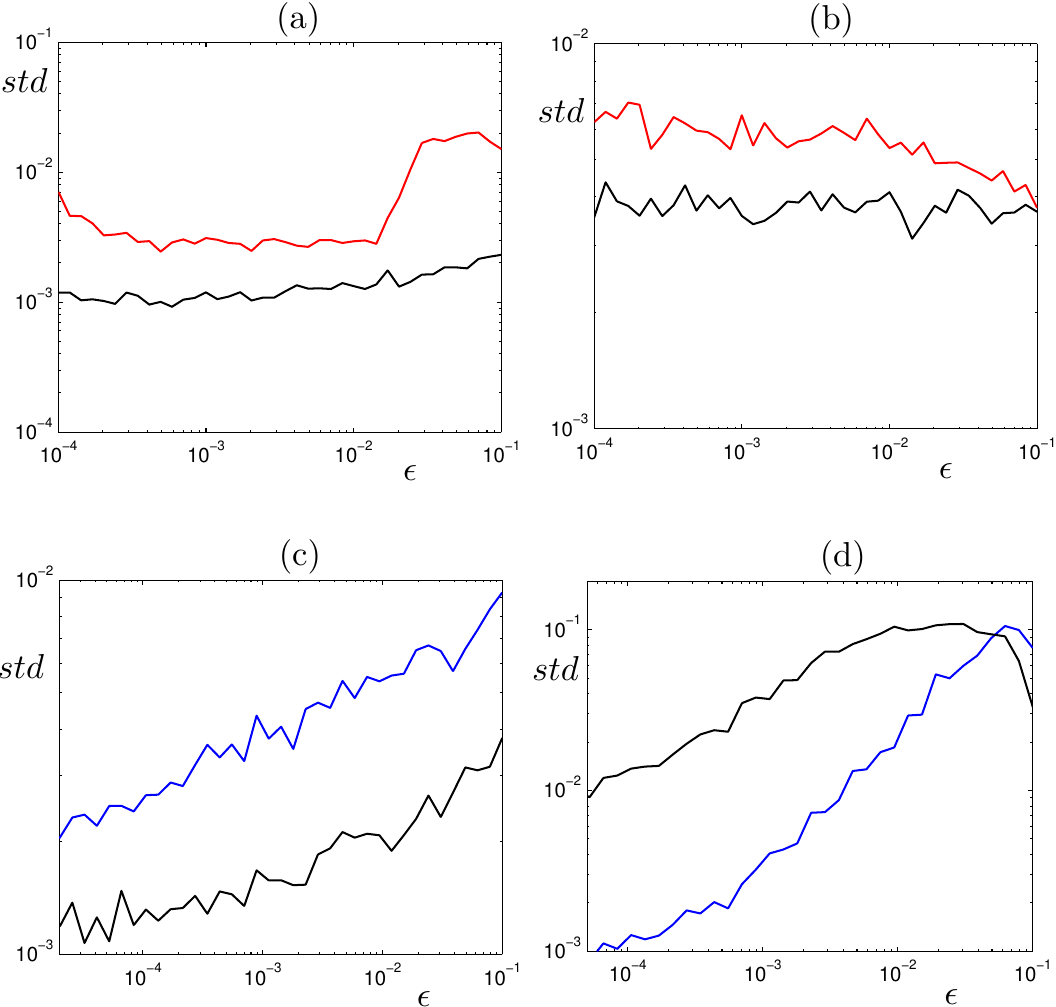}
	\caption{\label{Koper_fig04}%
	Dependence of the standard deviation of the distance between the stochastic and 
	the deterministic transition maps  on~$\eps$ for an attracting deterministic periodic MMO during various phases of the flow.
	We have always fixed $\sigma=0.01=\sigma'$ and viewed 
	$\eps$ as a parameter. Otherwise, the same conventions as in 
	\figref{Koper_fig03} apply.}	
\end{figure}

Via the same strategy as in \figref{Koper_fig03} one may also numerically
investigate the dependence upon $\eps$. \figref{Koper_fig04} shows the 
results for this computation. Again, the results are consistent with the derived upper bounds. 
\figref{Koper_fig04}(a)--(b) verifies that for 
the maps $\Sigma_2\ra\Sigma_3$ and $\Sigma_3\ra \Sigma_4$ the stochastic 
spreadings of order $\cO(\sigma+\sigma')$ and $\cO(\sigma\sqrt\eps+\sigma')$ are
dominated by $\sigma,\sigma'$ if the noise level is fixed. 
 For the map 
$\Sigma_4\ra \Sigma_5$ analyzed in \figref{Koper_fig04}(c) we expect 
from Table~\ref{table_deviations} that the spreading is dominated by a scaling
$\cO(\eps^{1/6})$ since we have fixed $\sigma=0.01=\sigma'$ and
$\eps^{1/2}\ll\eps^{1/6}$ as $\eps\ra 0$. 
Inspection of~\figref{Koper_fig04}(c) shows indeed a corresponding slope of 
approximately~$1/6$. 
\figref{Koper_fig04}(d) is also consistent with the expected 
scaling $\cO(\eps^{1/4})$ near the folded node for noise level and $\mu$ 
fixed. These results provide very good evidence that our theoretical
estimates may also form a practical guideline to analyze the 
spreading due to noise.

Of course, one may also consider the influence of noise on sample 
paths for global returns. \figref{fig_Koper_poincare} shows the global
return map $\Sigma_1\ra \Sigma_1$. For this computation, in contrast
to the previous computations in this section, we have chosen a regime 
with many secondary canards~\cite{BGK12,WechselbergerFN}. Indeed, the 
parameter values have been fixed to $k=-10$, $\lambda=-7.6$ so that the 
folded node on $L^+$ is given by $(x,z)=(1,-9.2)$ with eigenvalues of 
the matrix $A$ from~\eqref{eq:Apm} given by $\rho_s<-1<\rho_w<0$. The 
eigenvalue ratio $\mu:=\rho_w/\rho_s$ is approximately given by 
$\mu\approx 0.0252$ so that~\cite[Thm.~2.3]{KuehnMMO} implies that there are
two primary and $19$ secondary canards. The deterministic return map
has been analyzed numerically in~\cite{GuckenheimerFNFSN,KuehnRetMaps}, 
and the structure of the different rotational sectors separated by canard 
orbits is well understood; see also~\cite{GuckenheimerScheper}. However, 
\figref{fig_Koper_poincare} shows that the attracting deterministic periodic 
orbit corresponding to a fixed point of the return map, can shift due 
to noise, even to a higher sector of rotation. This effect can be seen
directly from  Theorem~\ref{thm_local_returns_outer} above. 

In summary, we may conclude that there is a highly non-trivial 
interplay between the number of SAOs, the global return map and 
the noise level (cf.~also~\cite[Cor.~6.3]{BGK12}). As discussed already
in Section \ref{sec:intro}, the natural next step is to consider the 
analysis of the discrete-time Markov chain on a finite state space
of MMO patterns. The results in this paper and in~\cite{BGK12} provide
the necessary estimates for the kernel of the Markov
chain and may form the starting point for future work.


\appendix

\section{Technical lemmas}
\label{appendix}

\begin{lem}[Scaling behaviour]
\label{lem_app_scaling} 
Assume $a(s)\asymp -(\abs{s}^{1/2}+\eps^{1/3})$ for $-1\leqs s\leqs
\eps^{2/3}$, and define $\alpha(s,r) = \int_r^s a(u)\6u$. Then for any
$\nu\geqs-1$, 
\begin{equation}
 \label{app_scaling01}
 \int_{s_0}^s \e^{\alpha(s,r)/\eps} \abs{a(r)}^\nu \6r 
 \asymp \eps \abs{a(s)}^{2\nu-1}
\end{equation} 
holds for $-1\leqs s_0 \leqs s_0+\Order{\eps\abs{\log\eps}} \leqs s \leqs
\eps^{2/3}$. 
\end{lem}

\begin{lem}[Bernstein-type inequality]
\label{lem_app_Bernstein} 
Let $W^{(1)}_t, \dots, W^{(k)}_t$ be $k$ independent standard Brownian motions,
and let $X_t$ be adapted to the filtration generated by the $W^{(j)}_t$. 
For measurable functions  $g_{ij}$ and deterministic bounds $G_{i}$ satisfying almost surely
\begin{equation}
 \label{app_Bernstein02}
 \sum_{j=1}^k g_{ij}(X_{s},s)^2 \leqs G_i(s)^2\;, 
 \qquad
 V_i(t) = \int_0^t G_i(s)^2 \6s < \infty\;,
\end{equation} 
consider the $n$ martingales 
\begin{equation}
 \label{app_Bernstein01}
 M^{(i)}_t = \sum_{j=1}^k \int_0^t g_{ij}(X_s,s) \6W^{(j)}_s\;,
 \qquad 
 i=1,\dots,n\;.
\end{equation}
Then for any $h>0$ and any choice of $\gamma_1,\dots,\gamma_n\in[0,1]$ such that
$\gamma_1+\dots+\gamma_n=1$, one has 
\begin{equation}
 \label{app_Bernstein03}
 \P\biggl\{ \sup_{0\leqs s\leqs t} \norm{M_s}\geqs h \biggr\} 
 \leqs 
 2\sum_{i=1}^n \exp \biggl\{ -\frac{\gamma_i h^2}{2V_i(t)}\biggr\}\;.
\end{equation} 
\end{lem}

\begin{proof}
The left-hand side of~\eqref{app_Bernstein03} is bounded above by 
\begin{equation}
\P\biggl\{ \sum_{i=1}^n \sup_{0\leqs s\leqs t} (M^{(i)}_s)^2 \geqs h^2 \biggr\} 
\leqs 
\sum_{i=1}^n \P\biggl\{ \sup_{0\leqs s\leqs t} (M^{(i)}_s)^2 \geqs \gamma_i h^2
\biggr\}\;.
\end{equation}
Each term in this sum is bounded by $2\e^{-\gamma_i h^2/(2V_i(t))}$,
cf.~\cite[Lemma~D.8]{BGK12}.
\end{proof}

\begin{rem}
\label{rem_Bernstein} 
It is possible to obtain sharper estimates of the form 
\begin{equation}
 \label{app_Bernstein04}
 \P\biggl\{ \sup_{0\leqs s\leqs t} \frac{\norm{M_s}}{\norm{V(s)}^{1/2}}\geqs h
\biggr\} 
 \leqs C(t,\kappa) \e^{-\kappa h^2/2\sigma^2}
\end{equation} 
for any $\kappa<1$, using a decomposition of $[0,t]$ in small intervals
(see~\cite[Section~5.1.2]{BGbook}). 
\end{rem}

\begin{lem}[Random time change]
\label{lem_random_time_change} 
Consider an\/ $\R^n$-valued diffusion $(Y_t)_{t\geqs 0}$ given by  
\begin{equation}
 \label{l_rtc:1}
 \6Y_t = f(Y_t)\6t + g(Y_t)\6W_t\;,
\end{equation} 
where $f:\R^n\to\R^n$ is such that the $n$\/th component $f_{n}$ of $f$ satisfies $f_n(y)>0$ for all~$y$. We further assume that 
$g:\R^n\to\R^{n\times k}$, and that $(W_t)_{t\geqs0}$ is a $k$-dimensional standard
Brownian motion. Fix $\gamma>0$ and let 
\begin{equation}
 \label{l_rtc:2}
 \beta(t,\omega) = \gamma \int_0^t f_n(Y_s(\omega))\6s\;.
\end{equation} 
Then $(Y_{\beta^{-1}(t)})_{t\geqs0}$ is equal in distribution to
$(X_t)_{t\geqs0}$, where $X_t$ satisfies $X_{0}=Y_{0}$ and 
\begin{equation}
 \label{l_rtc:3}
 \6X_t = \frac{1}{\gamma f_n(X_t)}f(X_t)\6t +
\frac{1}{\sqrt{\gamma f_n(X_t)}}g(X_t)\6W_t\;.
\end{equation} 
If the condition $f_n(y)>0$ is only satisfied in a subset $\cD$ of $\R^n$, then
the result remains true for $0\leqs t\leqs \tau_\cD =
\inf\setsuch{s\geqs0}{Y_s\not\in\cD}$. 
\end{lem}

\begin{proof}
Write $\6X_t = \tilde f(X_t)\6t + \tilde g(X_t)\6W_t$ for the SDE~\eqref{l_rtc:3}. 
The stochastic process $c(t,\omega) = \gamma f_n(Y_t(\omega))$ is adapted to the filtration
$(\cF_t)_{t}$ of the Brownian motion. Since $t\mapsto \beta(t)$ is almost surely invertible, \cite[Theorem~8.5.1, p.~154]{Oksendal} implies that $(X_t)_{t}$ is equal in distribution to $(Y_{\beta^{-1}(t)})_{t}$, with $Y_{t}$ is given by 
\begin{equation}
 \label{l_rtc:4}
 \6Y_t = u(t,\omega)\6t + v(t,\omega)\6W_t\;,
\end{equation} 
provided 
\begin{align}
\nonumber
u(t,\omega) &= c(t,\omega) \tilde f(Y_t)\;, \\
v\transpose{v}(t,\omega) &= c(t,\omega) \tilde g \transpose{\tilde g}(Y_t)\;.
\end{align}
Setting $u(t,\omega)=f(Y_t(\omega))$ and $v(t,\omega)=g(Y_t(\omega))$ this is clearly the case. To prove the last statement, it suffices to consider $Y_{t\wedge\tau_\cD}$. 
\end{proof}

\section{Proof of Lemma~\ref{lem_V}}
\label{app_lem_V} 

We would like to estimate the principal solution of $\eps\dot\xi = A(s)\xi$
for times $s\leqs-\sqrt{\eps}$, where 
\begin{equation}
 \label{eq:lem_V01}
 A(s) = 
 \begin{pmatrix}
 a_0(s) & d(s) \\ c(s) & \eps a_1(s) 
 \end{pmatrix}
 =
 \begin{pmatrix}
 s + \Order{s^2} & 1 + \Order{s^2} \\
 -\eps(1+\mu) + \Order{\eps s} & \eps \Order{1}
 \end{pmatrix}\;.
\end{equation} 
By changing $s$ into $-s$ and $\xi_2$ into $-\xi_2$, we can restrict the
analysis to positive $s$. The proof of the lemma is close in spirit to the
proof of~\cite[Theorem~4.3]{BGK12}, and consists in a number of changes of
variables bringing the system into diagonal form. 
A first transformation 
\begin{equation}
\xi = \exp\biggl\{\frac{1}{2\eps}\int_0^s \Tr A(u)\6u\biggr\} \xi_1
=: \e^{\alpha(s)/2\eps} \xi_1
\end{equation} 
yields the system $\eps\dot\xi_1 = A_1(s)\xi_1$, where 
\begin{equation}
 A_1(s) =
 \begin{pmatrix}
 \frac12 a(s) & d(s) \\ c(s) & -\frac12 a(s) 
 \end{pmatrix}\;,
 \qquad
 a(s) = a_0(s) - \eps a_1(s) = s + \Order{s^2}\;.
\end{equation} 
Next we set $\xi_1 = S_1(s) \xi_2$, where 
\begin{equation}
 S_1(s) = \frac{1}{\sqrt{d(s)}}
 \begin{pmatrix}
 d(s) & 0 \\ -\frac12 a(s) + \frac12 \eps \frac{\dot d(s)}{d(s)} & 1
 \end{pmatrix}\;,
\end{equation} 
which yields $\eps\dot\xi_2 = A_2(s)\xi_2$, with 
\begin{equation}
 A_2(s) = 
 \begin{pmatrix}
 0 & 1 \\ h(s) & 0
 \end{pmatrix}\;,
\end{equation} 
and 
\begin{equation}
h = \frac{a^2}{4} + cd + \eps\frac{\dot a}{2} - \eps\frac{a\dot d}{2d}
+ \frac{\eps^2}{2}\frac{\ddot d}{d} + \frac{3}{4} \eps^2 \frac{\dot d^2}{d^2} 
= \frac{a^2}{4} - \eps \biggl( \frac12 + \mu \biggr) + \Order{\eps s}\;.
\end{equation}
Note that this system is equivalent to $\eps^2 \ddot\xi_{2,1} = h(s)\xi_{2,1}$,
which reduces to Weber's equation in the particular case $a(s)=s$. 
The next step is to set 
$\xi_2 = S_2(s) \xi_3$, where 
\begin{equation}
 S_2(s) = \frac{1}{\sqrt{2}}
 \begin{pmatrix}
 h(s)^{-1/4} & -h(s)^{-1/4} \\ h(s)^{1/4} & h(s)^{1/4}
 \end{pmatrix}
\end{equation}
is such that $S_2^{-1}A_2S_2$ is diagonal. This 
yields $\eps\dot\xi_3 = A_3(s)\xi_3$, with 
\begin{equation}
 A_3(s) = 
 \begin{pmatrix}
 h(s)^{1/2} & -\frac{\eps}{4}\frac{\dot h(s)}{h(s)} \\ 
 -\frac{\eps}{4}\frac{\dot h(s)}{h(s)} & -h(s)^{1/2}
 \end{pmatrix}\;.
\end{equation} 
The last transformation $\xi_3 = S_3(s) \xi_4$ makes the system diagonal. This final transformation is given by
\begin{equation}
 S_3(s) = 
 \begin{pmatrix}
 1 & p_2(s) \\ p_1(s) & 1
 \end{pmatrix}\;,
\end{equation} 
where $p_1$ and $p_2$ satisfy the ODEs
\begin{align}
 \eps\dot{p_1} 
 &= -2h(s)^{1/2}p_1 + \frac{\eps}{4}\frac{\dot h(s)}{h(s)} (p_1^2 - 1)\;, \\
 \eps\dot{p_2} 
 &= 2h(s)^{1/2}p_2 + \frac{\eps}{4}\frac{\dot h(s)}{h(s)} (p_2^2 - 1)\;.
\end{align} 
One can show that these ODEs admit solutions of order $\eps/s^2$. The resulting
system has the form $\eps\dot\xi_4 = A_4(s)\xi_4$, where 
\begin{equation}
 A_4(s) = 
 \begin{pmatrix}
 h(s)^{1/2} -\frac{\eps}{4}\frac{\dot h(s)}{h(s)}p_1(s) & 0 \\ 
 0 & -h(s)^{1/2} -\frac{\eps}{4}\frac{\dot h(s)}{h(s)} p_2(s) 
 \end{pmatrix}\;,
\end{equation} 
and the principal solution is thus of the form 
$V(s,r) = V(s)V(r)^{-1}$, where 
\begin{equation}
\label{eq:lem_V99}
 V(s) = \e^{\alpha(s)/2\eps}S_1(s)S_2(s)S_3(s) 
 \begin{pmatrix}
 \e^{\alpha_+(s)/\eps} & 0 \\ 0 & \e^{\alpha_-(s)/\eps}
 \end{pmatrix}
\end{equation}
with  
\begin{equation}
 \alpha_\pm(s) = \pm\int_{1}^s h(u)^{1/2}\6u + 
 \cO\biggl(\frac{\eps^2}{s^2}\biggr)\;.
\end{equation} 
Expanding $h(s)^{1/2}$ and using the fact that $\dot a(s)=1+\Order{s}$, one
obtains 
\begin{equation}
 \e^{\alpha_\pm(s)/\eps} \asymp a(s)^{\mp(1/2+\mu)} \e^{\pm \alpha(s)/2\eps}\;.
\end{equation} 
The result follows by evaluating the matrix products in~\eqref{eq:lem_V99}.
\qed

\section{Proof of Proposition~\ref{prop_fn_averaging}}
\label{appendix_averaging}

We shall use a parametrization of the level curves of $K$ for $K>0$ which was
introduced in~\cite{BerglundLandon}. It is given by 
\begin{align}
\nonumber
u_1 &= \sqrt{\frac{1+\mu}{2} \abs{\log K}} \,\sin\varphi\;, \\
u_2 &= u_1^2 + \frac{1+\mu}{2} 
f\bigl( X \bigr)\;,
\qquad
X := \sqrt{\frac{\abs{\log K}}{2} } \cos\varphi\;,
\label{app_avrg:01} 
\end{align}
where $f(t)$ is the solution of 
\begin{equation}
 \label{app_avrg:02}
 \log(1+f(t)) = f(t) - 2t^2
\end{equation} 
satisfying $\sign f(t) = \sign t$. It is easy to check
(see~\cite[Section~5.2.1]{Landon_Thesis}) that 
\begin{itemize}
 \item $f(t) = 2t + \Order{t^2}$ near $t=0$;
 \item $-1 + \e^{-1-2t^2} \leqs f(t) \leqs -1 + \e^{-1-2t^2} +
\Order{(\e^{-1-2t^2})^2}$ for $t\leqs0$;
 \item $f(t) \leqs 2t^2 + \Order{t}$ for all $t\geqs0$. 
\end{itemize}

\begin{lem}
\label{lem_avrg1}
The variational equations~\eqref{eq:fn_nbh_det:106} are equivalent to 
\begin{align}
\nonumber
\mu\frac{\6K}{\6\bar z} 
&= -8\bar z \frac{K\abs{\log K}}{1+f(X)} \sin^2\varphi\;, \\
\mu\frac{\6\varphi}{\6\bar z} 
&= \frac{1}{X}
\biggl[
\sqrt{1+\mu} f(X) - 4\bar z \sqrt{\frac{\abs{\log K}}{2} } \sin\varphi
\biggl( \frac{\sin^2\varphi}{1+f(X)}-1\biggr)
\biggr]\;.
\label{app_avrg:03} 
\end{align}
\end{lem}

\begin{proof}
The first equation is a direct consequence of~\eqref{eq:fn_nbh_det:108}, using
the fact that $\e^{-f(X)}=K^{\cos^2\varphi}/[1+f(X)]$. The second one is
obtained by differentiating the first equation in~\eqref{app_avrg:01} and
solving for $\mu\6\varphi/\6\bar z$. 
\end{proof}

The first term on the right-hand side of $\mu\6\varphi/\6\bar z$ is of order
$f(X)/X$, which varies between $1/\abs{X}$ and $2\abs{X}$. As for the second
term, it is easy to see that it is negligible for $X\geqs 0$. For
$X<0$, setting $u=\sin\varphi$, we have $1+f(X)=K^{1-u^2}$. The function
$u\mapsto u^3K^{u^2-1}$ has a maximal value $1/(K\abs{\log K}^{3/2})$, reached
when $u^2=3/(2\abs{\log K})$. This allows to bound the second term for $X<0$
and proves~\eqref{eq:fn_nbh_det:109}.

It follows immediately from~\eqref{app_avrg:03} that we can write 
\begin{equation}
 \label{app_avrg:04}
 \frac{\6K}{\6\varphi} = \bar z g(K,\varphi,\bar z)\;,
\end{equation} 
where 
\begin{equation}
 \label{app_avrg:05}
 g(K,\varphi,\bar z) 
 = -8 \frac{K\abs{\log K}}{1+f(X)}
 \frac{X\sin^2\varphi}{\sqrt{1+\mu}f(X) - 4\bar z R\sin\varphi
 \Bigl[\frac{\sin^2\varphi}{(1+f(X))}-1\Bigr]}\;,
\end{equation} 
and we have set $R=\sqrt{\abs{\log K}/2}$. Note that $\mu\6K/\6\bar z$ has
order $(1-K)$ near $K=1$, and order $\abs{\log K}$ near $K=0$, because
$1+f(X)\geqs \e^{-1} K$. This shows that 
\begin{equation}
 \label{app_avrg:06}
 \abs{g(K,\varphi,\bar z)} \leqs \Order{(1-K)(1+\abs{\log K})^{3/2}}
 =: \Order{\rho(K)}\;.
\end{equation}

We now perform the averaging transformation 
\begin{equation}
 \label{app_avrg:07}
 \Kbar = K + \bar z w(K,\varphi,\bar z)\;,
\end{equation} 
where 
\begin{equation}
 \label{app_avrg:08}
 \frac{\partial}{\partial \varphi}w(K,\varphi,\bar z) = -g(K,\varphi,\bar z) + \bar g(K,\bar z)\;, 
 \qquad
 \bar g(K,\bar z) = \frac{1}{2\pi} \int_0^{2\pi} g(K,\varphi,\bar z)
\6\varphi\;.
\end{equation} 
This yields 
\begin{equation}
 \label{app_avrg:09}
 \frac{\6\Kbar}{\6\varphi}
 = \bar z \bar g(K,\bar z) 
 + \bar z \frac{\partial w}{\partial K} \frac{\6K}{\6\varphi}
 + \biggl( \bar z \frac{\partial w}{\partial \bar z} + w \biggr)
 \frac{\6\bar z}{\6\varphi}\;.
\end{equation} 
Recalling that $\6K/\6\varphi$ has order $\bar z g$ and $\6\bar z/\6\varphi$ has
order $\mu\abs{\log K}^{1/2}$, together with~\eqref{app_avrg:06}  allows to bound the
last two terms. 

Finally, we show that $c_-(1-K) \leqs -\bar g \leqs c_+(1-K)$ for
$K\geqs c\bar z$. The result will then follow by expressing $K$ in terms of 
$\Kbar$ (note that $1-K$ and $1-\Kbar$ are comparable for $K\geqs c\bar z$). It
will be sufficient to consider the behaviour of $\bar g$ near $K=1$ and near
$K=0$. Near $K=1$ we have 
\begin{equation}
 \label{app_avrg:10}
 g(K,\varphi,\bar z) = -\frac{4(1-K)}{\sqrt{1+\mu}} 
 \bigl[ 1+\Order{1-K} + \Order{\bar z}\bigr] \sin^2\varphi\;,
\end{equation} 
and $\sin^2\varphi$ averages to $1/2$. Near $K=0$, the integral defining $\bar
g$ is dominated by $\varphi$ near $\pi$. Performing the change of variables 
$u=2R\sin\varphi$, we obtain 
\begin{align}
\nonumber
\int_{\pi/2}^{3\pi/2} g(K,\varphi,\bar z) \6\varphi 
 &= -\frac{2\e}{\sqrt{1+\mu}} \int_{-2R}^{2R} u^2 \e^{-u^2/2}
 \biggl[
 1 + \Order{K\e^{u^2/2}} + \cO \biggl( \frac{\bar zu^3}{KR^2}\e^{-u^2/2}\biggr)
 \biggr] \\
 &= -\frac{2\sqrt{2\pi}\e}{\sqrt{1+\mu}}
 \Bigl[
 1 + \Order{K\abs{\log K}} + \Order{\bar z/(K\abs{\log K})}
 \Bigr]\;,
 \label{app_avrg:11}
\end{align}
while the integral over $[-\pi/2,\pi/2]$ has order $K\abs{\log K}$. 
\qed

\begin{small}
\bibliographystyle{abbrv}
\bibliography{BDGK}

\def\cprime{$'$}
\begin{thebibliography}{10}

\bibitem{AS}
M.~Abramowitz and I.~Stegun.
\newblock {\em Handbook of Mathematical Functions}.
\newblock Dover, 10th edition, 1972.

\bibitem{ArnoldEncy}
V.~I. Arnold, V.~S. Afrajmovich, Y.~S. Il{\cprime}yashenko, and L.~P.
  Shil{\cprime}nikov.
\newblock {\em Bifurcation theory and catastrophe theory}.
\newblock Springer, Berlin, 1999.
\newblock Translated from the 1986 Russian original, Reprint of the 1994
  English edition from the series Encyclopaedia of Mathematical Sciences
  [{{\i}t Dynamical systems. V}, Encyclopaedia Math. Sci., 5, Springer, Berlin,
  1994].

\bibitem{Arrhenius}
S.~Arrhenius.
\newblock On the reaction velocity of the inversion of cane sugar by acids.
\newblock {\em J.~Phys.\ Chem.}, 4:226, 1889.
\newblock In German. Translated and published in: Selected Readings in Chemical
  Kinetics, M.H. Back and K.J. Laider (eds.), Pergamon, Oxford, 1967.

\bibitem{AshwinWieczorekVitoloCox}
P.~Ashwin, S.~Wieczorek, R.~Vitolo, and P.~Cox.
\newblock Tipping points in open systems: bifurcation, noise-induced and
  rate-dependent examples in the climate system.
\newblock {\em Phil. Trans. R. Soc. A}, 370:1166--1184, 2012.

\bibitem{AvrachenkovLasserre99}
K.~E. Avrachenkov and J.~B. Lasserre.
\newblock The fundamental matrix of singularly perturbed {M}arkov chains.
\newblock {\em Adv. in Appl. Probab.}, 31(3):679--697, 1999.

\bibitem{BenArous_Kusuoka_Stroock_1984}
G.~Ben~Arous, S.~Kusuoka, and D.~W. Stroock.
\newblock The {P}oisson kernel for certain degenerate elliptic operators.
\newblock {\em J. Funct. Anal.}, 56(2):171--209, 1984.

\bibitem{Benoit1}
E.~Beno\^{i}t.
\newblock Syst\`emes lents-rapides dans $\mathbb{R}^3$ et leurs canards.
\newblock In {\em Third Snepfenried geometry conference}, volume~2, pages
  159--191. Soc. Math. France, 1982.

\bibitem{Benoit5}
E.~Beno\^{i}t.
\newblock Enlacements de canards.
\newblock {\em C.R. Acad. Sc. Paris}, 300(8):225--230, 1985.

\bibitem{Benoit2}
E.~Beno\^{i}t.
\newblock Canards et enlacements.
\newblock {\em Publ. Math. IHES}, 72:63--91, 1990.

\bibitem{BenoitCallotDienerDiener}
E.~Beno\^{i}t, J.~Callot, F.~Diener, and M.~Diener.
\newblock Chasse au canards.
\newblock {\em Collect. Math.}, 31:37--119, 1981.

\bibitem{Benoit4}
E.~Beno\^{i}t and C.~Lobry.
\newblock Les canards de $\mathbb{R}^3$.
\newblock {\em C.R. Acad. Sc. Paris}, 294:483--488, 1982.

\bibitem{BSV}
R.~Benzi, A.~Sutera, and A.~Vulpiani.
\newblock The mechanism of stochastic resonance.
\newblock {\em J. Phys.\ A}, 14(11):L453--L457, 1981.

\bibitem{Berglund_irs_MPRF}
N.~Berglund.
\newblock {K}ramers' law: Validity, derivations and generalisations.
\newblock {\em Markov Process. Related Fields}, 19(3):459--490, 2013.

\bibitem{BG6}
N.~Berglund and B.~Gentz.
\newblock Geometric singular perturbation theory for stochastic differential
  equations.
\newblock {\em J.~{Differential} {Equations}}, 191:1--54, 2003.

\bibitem{BGbook}
N.~Berglund and B.~Gentz.
\newblock {\em Noise-induced phenomena in slow--fast dynamical systems. A
  sample-paths approach}.
\newblock Probability and its Applications. Springer-Verlag, London, 2006.

\bibitem{BG_periodic2}
N.~Berglund and B.~Gentz.
\newblock On the noise-induced passage through an unstable periodic orbit {II}:
  {G}eneral case.
\newblock {\em SIAM J. Math. Anal.}, 46(1):310--352, 2014.

\bibitem{BGK12}
N.~Berglund, B.~Gentz, and C.~Kuehn.
\newblock Hunting {F}rench ducks in a noisy environment.
\newblock {\em J. Differential Equations}, 252(9):4786--4841, 2012.

\bibitem{BerglundLandon}
N.~Berglund and D.~Landon.
\newblock Mixed-mode oscillations and interspike interval statistics in the
  stochastic {F}itz{H}ugh--{N}agumo model.
\newblock {\em Nonlinearity}, 25(8):2303--2335, 2012.

\bibitem{BoissonadeDeKepper}
J.~Boissonade and P.~DeKepper.
\newblock {Transitions from bistability to limit cycle oscillations.
  Theoretical analysis and experimental evidence in an open chemical system}.
\newblock {\em J. Phys. Chem.}, 84:501--506, 1980.

\bibitem{GuckFvdP2}
K.~Bold, C.~Edwards, J.~Guckenheimer, S.~Guharay, K.~Hoffman, J.~Hubbard,
  R.~Oliva, and W.~Weckesser.
\newblock The forced van der {Pol} equation~2: {C}anards in the reduced system.
\newblock {\em SIAM Journal of Applied Dynamical Systems}, 2(4):570--608, 2003.

\bibitem{BEGK}
A.~Bovier, M.~Eckhoff, V.~Gayrard, and M.~Klein.
\newblock Metastability in reversible diffusion processes. {I}. {S}harp
  asymptotics for capacities and exit times.
\newblock {\em J. Eur. Math. Soc. (JEMS)}, 6(4):399--424, 2004.

\bibitem{BGK}
A.~Bovier, V.~Gayrard, and M.~Klein.
\newblock Metastability in reversible diffusion processes. {II}. {P}recise
  asymptotics for small eigenvalues.
\newblock {\em J. Eur. Math. Soc. (JEMS)}, 7(1):69--99, 2005.

\bibitem{BronsKrupaRotstein}
M.~Brons, T.~Kaper, and H.~Rotstein.
\newblock {Introduction to focus issue -- mixed mode oscillations: experiment,
  computation, and analysis}.
\newblock {\em Chaos}, 18:015101, 2008.

\bibitem{BronsKrupaWechselberger}
M.~Br{\o}ns, M.~Krupa, and M.~Wechselberger.
\newblock Mixed mode oscillations due to the generalized canard phenomenon.
\newblock {\em Fields Institute Communications}, 49:39--63, 2006.

\bibitem{DegnOlsenPerram}
H.~Degn, L.~Olsen, and J.~Perram.
\newblock Bistability, oscillation, and chaos in an enzyme reaction.
\newblock {\em Annals of the New York Academy of Sciences}, 316(1):623--637,
  1979.

\bibitem{KuehnMMO}
M.~Desroches, J.~Guckenheimer, C.~Kuehn, B.~Krauskopf, H.~Osinga, and
  M.~Wechselberger.
\newblock Mixed-mode oscillations with multiple time scales.
\newblock {\em SIAM Review}, 54(2):211--288, 2012.

\bibitem{DesrochesKrauskopfOsinga1}
M.~Desroches, B.~Krauskopf, and H.~Osinga.
\newblock {The geometry of mixed-mode oscillations in the Olsen model for the
  perioxidase-oxidase reaction}.
\newblock {\em DCDS-S}, 2(4):807--827, 2009.

\bibitem{Dicksonetal}
C.~Dickson, J.~Magistretti, M.~Shalisnky, B.~Hamam, and A.~Alonso.
\newblock {Oscillatory activity in entorhinal neurons and circuits: Mechanisms
  and function}.
\newblock {\em Ann. N.Y. Acad. Sci.}, 911:127--150, 2006.

\bibitem{DRvdP}
F.~Dumortier and R.~Roussarie.
\newblock Canard cycles and center manifolds.
\newblock {\em Memoirs of the American Mathematical Society}, 121(577), 1996.

\bibitem{ErmentroutTerman}
G.~Ermentrout and D.~Terman.
\newblock {\em Mathematical Foundations of Neuroscience}.
\newblock Springer, 2010.

\bibitem{Eyring}
H.~Eyring.
\newblock The activated complex in chemical reactions.
\newblock {\em Journal of Chemical Physics}, 3:107--115, 1935.

\bibitem{Fenichel}
N.~Fenichel.
\newblock Geometric singular perturbation theory for ordinary differential
  equations.
\newblock {\em J. Differential Equations}, 31(1):53--98, 1979.

\bibitem{FreidlinWentzell}
M.~Freidlin and A.~Wentzell.
\newblock {\em Random Perturbations of Dynamical Systems}.
\newblock Springer, 1998.

\bibitem{GHM}
L.~Gammaitoni, P.~H\"anggi, P.~Jung, and F.~Marchesoni.
\newblock Stochastic resonance.
\newblock {\em Rev.\ Mod.\ Phys.}, 70:223--287, 1998.

\bibitem{GoryachevStrizhakKapral}
A.~Goryachev, P.~Strizhak, and R.~Kapral.
\newblock Slow manifold structure and the emergence of mixed-mode oscillations.
\newblock {\em J. Chem. Phys.}, 107(18):2881--2889, 1997.

\bibitem{GuckFT}
J.~Guckenheimer.
\newblock Global bifurcations of periodic orbits in the forced van der {P}ol
  equation.
\newblock {\em in: Global Analysis of Dynamical Systems --- Festschrift
  dedicated to Floris Takens. Eds.: Henk W. Broer, Bernd Krauskopf and Gert
  Vegter}, pages 1--16, 2003.

\bibitem{GuckenheimerFNFSN}
J.~Guckenheimer.
\newblock Return maps of folded nodes and folded saddle-nodes.
\newblock {\em Chaos}, 18:015108, 2008.

\bibitem{GuckenheimerSH}
J.~Guckenheimer.
\newblock Singular {H}opf bifurcation in systems with two slow variables.
\newblock {\em SIAM J. Appl. Dyn. Syst.}, 7(4):1355--1377, 2008.

\bibitem{GuckenheimerMeerkamp}
J.~Guckenheimer and P.~Meerkamp.
\newblock Bifurcation analysis of singular hopf bifurcation in
  {$\mathbb{R}^3$}.
\newblock {\em SIAM J. Appl. Dyn. Syst.}, 11(4):1325--1359, 2012.

\bibitem{GuckenheimerScheper}
J.~Guckenheimer and C.~Scheper.
\newblock A geometric model for mixed-mode oscillations in a chemical system.
\newblock {\em SIAM J. Appl. Dyn. Sys.}, 10(1):92--128, 2011.

\bibitem{Hassin_Haviv_92}
R.~Hassin and M.~Haviv.
\newblock Mean passage times and nearly uncoupled {M}arkov chains.
\newblock {\em SIAM J. Discrete Math.}, 5(3):386--397, 1992.

\bibitem{HauckSchneider}
T.~Hauck and F.~Schneider.
\newblock Mixed-mode and quasiperiodic oscillations in the peroxidase-oxidase
  reaction.
\newblock {\em J. Phys. Chem.}, 97:391--397, 1993.

\bibitem{HitczenkoMedvedev}
P.~Hitczenko and G.~Medvedev.
\newblock Bursting oscillations induced by small noise.
\newblock {\em SIAM J. Appl. Math.}, 69:1359--1392, 2009.

\bibitem{HitczenkoMedvedev1}
P.~Hitczenko and G.~Medvedev.
\newblock The {Poincar\'e} map of randomly perturbed periodic motion.
\newblock {\em J. Nonlinear Sci.}, 23(5):835--861, 2013.

\bibitem{HodgkinHuxley}
A.~Hodgin and A.~Huxley.
\newblock A quantitative description of membrane current and its application to
  conduction and excitation in nerve.
\newblock {\em J. Physiol.}, 117:500--505, 1952.

\bibitem{Hoepfner_Loecherbach_Thieullen}
R.~H\"opfner, E.~L\"ocherbach, and M.~Thieullen.
\newblock Ergodicity for a stochastic {H}odgkin--{H}uxley model driven by
  {O}rnstein--{U}hlenbeck type input.
\newblock {\em to appear: Ann. Inst. H. Poincar{\'e}}, 2014.
\newblock see also: arXiv:1311.3458v1.

\bibitem{HudsonHartMarinko}
J.~Hudson, M.~Hart, and D.~Marinko.
\newblock An experimental study of multiple peak periodic and nonperiodic
  oscillations in the {Belousov--Zhabotinskii} reaction.
\newblock {\em J. Chem. Phys.}, 71(4):1601--1606, 1979.

\bibitem{Izhikevich}
E.~Izhikevich.
\newblock Neural excitability, spiking, and bursting.
\newblock {\em Int. J. Bif. Chaos}, 10:1171--1266, 2000.

\bibitem{KaratzasShreve}
I.~Karatzas and S.~E. Shreve.
\newblock {\em Brownian motion and stochastic calculus}, volume 113 of {\em
  Graduate Texts in Mathematics}.
\newblock Springer-Verlag, New York, second edition, 1991.

\bibitem{KawczynskiStrizhak}
A.~Kawczynski and P.~Strizhak.
\newblock {Period adding and broken Farey tree sequences of bifurcations for
  mixed-mode oscillations and chaos in the simplest three-variable nonlinear
  system}.
\newblock {\em J. of Chem. Phys.}, 112(14):6122--6130, 2000.

\bibitem{Koper}
M.~Koper.
\newblock {Bifurcations of mixed-mode oscillations in a three-variable
  autonomous Van der {P}ol--{D}uffing model with a cross-shaped phase diagram}.
\newblock {\em Physica D}, 80:72--94, 1995.

\bibitem{KosmidisPakdaman}
E.~K. Kosmidis and K.~Pakdaman.
\newblock An analysis of the reliability phenomenon in the
  {F}itz{H}ugh--{N}agumo model.
\newblock {\em J. Comput.\ Neuroscience}, 14:5--22, 2003.

\bibitem{Kramers}
H.~A. Kramers.
\newblock Brownian motion in a field of force and the diffusion model of
  chemical reactions.
\newblock {\em Physica}, 7:284--304, 1940.

\bibitem{KrupaPopovicKopell}
M.~Krupa, N.~Popovic, and N.~Kopell.
\newblock Mixed-mode oscillations in three time-scale systems: A prototypical
  example.
\newblock {\em SIAM J. Applied Dynamical Systems}, 7(2):361--420, 2008.

\bibitem{KruSzm3}
M.~Krupa and P.~Szmolyan.
\newblock Extending geometric singular perturbation theory to nonhyperbolic
  points - fold and canard points in two dimensions.
\newblock {\em SIAM J. Math. Anal.}, 33(2):286--314, 2001.

\bibitem{KruSzm4}
M.~Krupa and P.~Szmolyan.
\newblock Extending slow manifolds near transcritical and pitchfork
  singularities.
\newblock {\em Nonlinearity}, 14:1473--1491, 2001.

\bibitem{KruSzm1}
M.~Krupa and P.~Szmolyan.
\newblock Geometric analysis of the singularly perturbed fold.
\newblock {\em In: Multiple-Time-Scale Dynamical Systems}, IMA Vol.
  122:89--116, 2001.

\bibitem{KrupaWechsFSN}
M.~Krupa and M.~Wechselberger.
\newblock Local analysis near a folded saddle-node singularity.
\newblock {\em J. Diff. Eq.}, 248(12):2841--2488, 2010.

\bibitem{KuehnRetMaps}
C.~Kuehn.
\newblock On decomposing mixed-mode oscillations and their return maps.
\newblock {\em Chaos}, 21(3):033107, 2011.

\bibitem{KuehnCT2}
C.~Kuehn.
\newblock {A mathematical framework for critical transitions: normal forms,
  variance and applications}.
\newblock {\em J. Nonlinear Sci.}, 23(3):457--510, 2013.

\bibitem{KuehnUM}
C.~Kuehn.
\newblock Loss of normal hyperbolicity of unbounded critical manifolds.
\newblock {\em Nonlinearity}, 27(6):1351--1366, 2014.

\bibitem{KuskeBorowski}
R.~Kuske and R.~Borowski.
\newblock Survival of subthreshold oscillations: the interplay of noise,
  bifurcation structure, and return mechanism.
\newblock {\em Discrete and Continous Dynamical Systems S}, 2(4):873--895,
  2009.

\bibitem{Landon_Thesis}
D.~Landon.
\newblock {\em Perturbation et excitabilit\'e dans des mod\`eles stochastiques
  de transmission de l'influx nerveux}.
\newblock PhD thesis, Universit\'e d'Orl\'eans, 2012.
\newblock \\ {\tt http://tel.archives-ouvertes.fr/tel-00752088}.

\bibitem{McNW}
B.~McNamara and K.~Wiesenfeld.
\newblock Theory of stochastic resonance.
\newblock {\em Phys.\ Rev.~A}, 39:4854--4869, 1989.

\bibitem{Mikikian}
M.~Mikikian, M.~Cavarroc, L.~Couedel, Y.~Tessier, and L.~Boufendi.
\newblock Mixed-mode oscillations in complex plasma instabilities.
\newblock {\em Physical Review Letters}, 100(22):225005, 2008.

\bibitem{MilikSzmolyanMap}
A.~Milik, P.~Szmolyan, H.~Loeffelmann, and E.~Groeller.
\newblock Geometry of mixed-mode oscillations in the 3-d autocatalator.
\newblock {\em Int. J. of Bif. and Chaos}, 8(3):505--519, 1998.

\bibitem{MKKR_B}
E.~Mishchenko, Y.~Kolesov, A.~Kolesov, and N.~Rozov.
\newblock {\em Asymptotic Methods in Singularly Perturbed Systems}.
\newblock Plenum Press, 1994.

\bibitem{MisRoz}
E.~Mishchenko and N.~Rozov.
\newblock {\em Differential Equations with Small Parameters and Relaxation
  Oscillations}.
\newblock Plenum Press, 1980.
\newblock (translated from Russian).

\bibitem{MuratovVanden-Eijnden}
C.~Muratov and E.~Vanden-Eijnden.
\newblock Noise-induced mixed-mode oscillations in a relaxation oscillator near
  the onset of a limit cycle.
\newblock {\em Chaos}, 18:015111, 2008.

\bibitem{Oksendal}
B.~{\O}ksendal.
\newblock {\em Stochastic Differential Equations}.
\newblock Springer, 2003.

\bibitem{PetrovScottShowalter}
V.~Petrov, S.~Scott, and K.~Showalter.
\newblock Mixed-mode oscillations in chemical systems.
\newblock {\em J. Chem. Phys.}, 97(9):6191--6198, 1992.

\bibitem{Roessler1}
O.~R{\"{o}}ssler.
\newblock Chaos in abstract kinetics: Two prototypes.
\newblock {\em Bulletin of Mathematical Biology}, 39:275--289, 1977.

\bibitem{RubinWechselberger1}
J.~Rubin and M.~Wechselberger.
\newblock Giant squid -- hidden canard: the 3d geometry of the
  {H}odgin--{H}uxley model.
\newblock {\em Biological Cybernetics}, 97(1), 2007.

\bibitem{Schefferetal}
M.~Scheffer, J.~Bascompte, W.~Brock, V.~Brovkhin, S.~Carpenter, V.~Dakos,
  H.~Held, E.~van Nes, M.~Rietkerk, and G.~Sugihara.
\newblock Early-warning signals for critical transitions.
\newblock {\em Nature}, 461:53--59, 2009.

\bibitem{Schweitzer_68}
P.~J. Schweitzer.
\newblock Perturbation theory and finite {M}arkov chains.
\newblock {\em J. Appl. Probability}, 5:401--413, 1968.

\bibitem{SuRubinTerman}
J.~Su, J.~Rubin, and D.~Terman.
\newblock Effects of noise on elliptic bursters.
\newblock {\em Nonlinearity}, 17(1):133--157, 2004.

\bibitem{Wechselberger}
P.~Szmolyan and M.~Wechselberger.
\newblock Canards in $\mathbb{R}^3$.
\newblock {\em Journal of Differential Equations}, 177:419--453, 2001.

\bibitem{SzmolyanWechselbergerRelax}
P.~Szmolyan and M.~Wechselberger.
\newblock Relaxation oscillations in $\mathbb{R}^3$.
\newblock {\em Journal of Differential Equations}, 200:69--104, 2004.

\bibitem{vanderPol_forced}
B.~{Van der Pol}.
\newblock The nonlinear theory of electric oscillations.
\newblock {\em Proc. IRE}, 22:1051--1086, 1934.

\bibitem{VanGoorZivadinovicMartinez-FuentesStojilkovic}
F.~van Goor, D.~Zivadinovic, A.~Martinez-Fuentes, and S.~Stojilkovic.
\newblock Dependence of pituitary hormone secretion on the pattern of
  spontaneous voltage-gated calcium influx.
\newblock {\em J. Biol. Chem.}, 276:33840--33846, 2001.

\bibitem{Wallet3}
G.~Wallet.
\newblock {Entr\'ee--sortie dans un tourbillon}.
\newblock {\em Ann. Inst. Fourier}, 36:157--184, 1986.

\bibitem{WechselbergerThesis}
M.~Wechselberger.
\newblock {\em Singularly perturbed folds and canards in $\R^3$}.
\newblock PhD thesis, Vienna University of Technology, Vienna, Austria, 1998.

\bibitem{WechselbergerFN}
M.~Wechselberger.
\newblock Existence and bifurcation of canards in $\mathbb{R}^3$ in the case of
  a folded node.
\newblock {\em SIAM J. Applied Dynamical Systems}, 4(1):101--139, 2005.

\bibitem{WechselbergerWeckesser}
M.~Wechselberger and W.~Weckesser.
\newblock Bifurcations of mixed-mode oscillations in a stellate cell model.
\newblock {\em Physica D}, 238:1598--1614, 2009.

\bibitem{WeissKnobloch}
J.~Weiss and E.~Knobloch.
\newblock A stochastic return map for stochastic differential equations.
\newblock {\em J. Stat. Phys.}, 58(5):863--883, 1990.

\bibitem{YinZhang1}
G.~Yin and Q.~Zhang.
\newblock {\em Discrete-time Markov Chains: Two-Time-Scale Methods and
  Applications}.
\newblock Springer, 2005.

\bibitem{YuKuskeLi}
N.~Yu, R.~Kuske, and Y.~Li.
\newblock Stochastic phase dynamics and noise-induced mixed-mode oscillations
  in coupled oscillators.
\newblock {\em Chaos}, 18:015112, 2008.

\end{thebibliography}
\end{small}

\newpage
\tableofcontents

\end{document}